\documentclass[a4paper,10pt]{article}
\usepackage{benstyle}
\usepackage{tcolorbox}
\usepackage{times}
\usepackage{enumitem}
\usepackage{scrextend}
\usepackage{cite}
\newtheorem{algorithm}[theorem]{Algorithm}

\usepackage{authblk}
\usepackage{pdfpages}

\DefineNamedColor{named}{Purple}{cmyk}{0.45,0.86,0,0}

\DefineNamedColor{named}{orange} {rgb}{1,0.55,0}

\title{
On the infinite-dimensional QR algorithm}

\author[$\dagger$]{Matthew J. Colbrook\footnote{Corresponding author: m.colbrook@damtp.cam.ac.uk}}
\author[$\dagger$]{Anders C. Hansen}

\affil[$\dagger$]{DAMTP, Centre for Mathematical Sciences, University of Cambridge, Wilberforce Rd, Cambridge CB3 0WA, United Kingdom}

\begin{document}

\date{}
\maketitle

\begin{abstract}
Spectral computations of infinite-dimensional operators are notoriously difficult, yet ubiquitous in the sciences. Indeed, despite more than half a century of research, it is still unknown which classes of operators allow for computation of spectra and eigenvectors with convergence rates and error control. Recent progress in classifying the difficulty of spectral problems into complexity hierarchies has revealed that the most difficult spectral problems are so hard that one needs three limits in the computation, and no convergence rates nor error control is possible. This begs the question: which classes of operators allow for computations with convergence rates and error control? In this paper we address this basic question, and the algorithm used is an infinite-dimensional version of the QR algorithm. Indeed, we generalise the QR algorithm to infinite-dimensional operators. We prove that not only is the algorithm executable on a finite machine, but one can also recover the extremal parts of the spectrum and corresponding eigenvectors, with convergence rates and error control. This allows for new classification results in the hierarchy of computational problems that existing algorithms have not been able to capture. The algorithm and convergence theorems are demonstrated on a wealth of examples with comparisons to standard approaches (that are notorious for providing false solutions).We also find that in some cases the IQR algorithm performs better than predicted by theory and make conjectures for future study. 
\end{abstract}

\vspace{1pc}
\noindent
{\it Keywords:} spectra, eigenvectors, computation, infinite-dimensional Hilbert spaces, hierarchies of computational problems

\vspace{0.5pc} \noindent
{\it Mathematics Subject Classification (2010):} 	47A10, 65J10, 46N40, 03D55

%\vspace{1pc} \noindent

%\tableofcontents

\section{Introduction}\label{intro}

Spectral computations are ubiquitous in the sciences with applications in solutions to differential and integral equations, spline functions, orthogonal polynomials, quantum mechanics, quantum chemistry, statistical mechanics, Hermitian and non-Hermitian Hamiltonians, optics etc. \cite{ponomarenko2013cloning,dean2013hofstadter,szabo2012modern,bender2007making,davies2007linear,shivakumar2016infinite,simon1998classical,gray2006toeplitz}. The computational problem is as follows. Letting $T$ denote a bounded linear operator on the canonical separable Hilbert space $l^2(\mathbb{N})$, one wants to design algorithms to compute the spectrum of $T$, denoted by $\sigma(T)$. Given the many applications, this problem has been investigated intensely since the 1950s \cite{aronszajn1951approximation, Arveson_cnum_lin94, Arveson_role_of94,Arveson_Improper93, Bottcher_appr, brown2007quasi, Lyonell_2, Sharg, deift1985toda, digernes1994finite, hansen2008, hansen2011, Levitin, Marletta_pollution, Marletta_Spec_gaps, trefethen2005spectra, Arieh1, Arieh2, Davies_sec_order, Seidel_JFA, SeSi_Pseudospectra, Eugene3, Silbermann22, Hansen_PRS, riddell1967spectral}, and we can only cite a small subset here. 

In the paper ``On the Solvability Complexity Index, the $n$-pseudospectrum and approximations of spectra of operators'' \cite{hansen2011} the Solvability Complexity Index (SCI) was introduced. The SCI provides a classification hierarchy \cite{hansen2011,SCI, CRAS,colbrook4,colbrook3,colbrook2019computing,colbrook2020pseudoergodic} of spectral problems according to their computational difficulty. The SCI of a class of spectral problems is the least number of limits needed in order to compute the spectrum of operators in this class. From a classical numerical analysis point of view such a concept may seem foreign. Indeed, the traditional sentiment is that one should have an algorithm, $\Gamma_n$, such that for an operator $T \in \mathcal{B}(l^2(\mathbb{N}))$,
\begin{equation}\label{eq:one_limit}
\Gamma_n(T) \longrightarrow \sigma(T), \quad n \rightarrow \infty,
\end{equation}
preferably with some form of error control of the convergence. As this philosophy forms the basics of numerical analysis, it naturally permeates the classical literature on the computational spectral problem. However, as is shown in \cite{hansen2011,SCI, CRAS}, an algorithm satisfying \eqref{eq:one_limit} is impossible even for the class of self-adjoint operators. Indeed, in the general case, the best possible alternative is an algorithm depending on three indices $n_1, n_2, n_3$ such that 
  \begin{equation*}
\lim_{n_3 \rightarrow \infty} \lim_{n_2 \rightarrow \infty} \lim_{n_1 \rightarrow \infty} \Gamma_{n_3,n_2,n_1}(T) = \sigma(T).
\end{equation*}
In fact, any algorithm with fewer than three limits will fail on the general class of operators. Moreover, no error control nor convergence rate on any of the limits are possible, since any such error control would reduce the number of limits needed. However, for the self-adjoint and normal cases two limits suffice in order to recover the spectrum. This phenomenon implies that the only way to characterise the computational spectral problem is through a hierarchy classifying the difficulty of computing spectra of different subclasses of operators. This is the motivation behind the SCI hierarchy, which also covers general numerical analysis problems. Indeed, the SCI hierarchy is closely related to Smale's question on the existence of purely iterative generally convergent algorithm for polynomial zero
finding \cite{Smale2}. As demonstrated by McMullen \cite{McMullen1, McMullen2} and Doyle \& McMullen \cite{Doyle_McMullen}, this is a case where several limits are needed in the computation, and their results become special cases of classification in the SCI hierarchy \cite{SCI, CRAS}.  

Informally, the SCI hierarchy is characterised as follows (see the Appendix \S \ref{SCI_basics} for a more detailed summary describing the SCI hierarchy).  

\begin{labeling}{\,\,\, $\Delta_0$:}
\itemsep0em
\item[\,\,\, $\Delta_0$:] The set of problems that can be computed in finite time, the SCI $=0$.
\item[\,\,\, $\Delta_1$:] The set of problems that can be computed using one limit, the SCI $=1$, however 
one has error control and one knows an error bound that tends to zero as the algorithm progresses.
\item[\,\,\, $\Delta_2$:] The set of problems that can be computed using one limit, the SCI $=1$, but error control may not be possible. 
\item[\,\,\, $\Delta_{m+1}$:] For $m \in \mathbb{N}$, the set of problems that can be computed by using $m$ limits, the SCI $\leq m$.
\end{labeling}

The class $\Delta_1$ is of course a highly desired class, however, most spectral problems are much higher in the hierarchy. For example. we have the following known classifications \cite{hansen2011,SCI, CRAS}.
\begin{itemize}
\itemsep0em
\item[(i)] The general spectral problem is in $\Delta_4\setminus \Delta_3$.
\item[(ii)] The self-adjoint spectral problem is in $\Delta_3\setminus \Delta_2$.
\item[(iii)] The compact spectral problem is in $\Delta_2\setminus \Delta_1$.
\end{itemize}
Here, the notation $\setminus$ indicates the standard ``setminus''. Note that the SCI hierarchy can be refined. We will not consider the full generalisation in the higher part of the hierarchy in this paper, but recall the class $\Sigma_1$ \cite{colb1}. This class is defined as follows.
\begin{labeling}{\,\,\, $\Sigma_1$:}
\itemsep0em
\item[\,\,\, $\Sigma_1$:] We have $\Delta_1 \subset \Sigma_1 \subset \Delta_2 $ and $\Sigma_1$ is the set of problems that can be computed by passing to one limit. Error control may not be possible, however, there exists an algorithm for these problems that converges and for which its output is included in the spectrum (up to an arbitrarily small accuracy parameter $\epsilon$).
\end{labeling}

In the context of computing $\sigma(T)$, a $\Sigma_1$ classification means the existence of an algorithm $\Gamma_n$ such that
$$
\Gamma_n(T)\subset \sigma(T)+B_{2^{-n}}(0)
$$
and $\Gamma_n$ converges to $\sigma(T)$ in the Hausdorff metric. The $\Sigma_1$ class is very important as it allows for algorithms that never make a mistake. In particular, one is always sure that the output is sound but we do not know if we have everything yet. The simplest infinite-dimensional spectral problem is that of computing the spectrum of an infinite diagonal matrix and, as is easy to see, we have the following.

\begin{itemize}
\itemsep0em
\item[(iv)] The problem of computing spectra of infinite diagonal matrices is in $\Sigma_1\setminus \Delta_1$.
\end{itemize}

Hence, the computational spectral problem becomes an infinite classification theory in order to characterise the above hierarchy. In order to do so, there will, necessarily, have to be many different types of algorithms. Indeed, characterising the hierarchy will yield a myriad of different approaches, as different structures on the various classes of operators will require specific algorithms. The key contribution of this paper is to investigate the convergence properties of the Infinite-dimensional QR (IQR) algorithm, its implementation properties, and how this algorithm provides classification results in the SCI hierarchy.

\subsection{Main contribution and novelty of the paper}

The main contributions of the paper can be summarised as follows: New convergence results, algorithmic results (the IQR algorithm can be implemented), classification results in the SCI hierarchy and numerical examples. 
\begin{labeling}{}
\itemsep0em
\item[\, (1) {\it Convergence results}:] We provide new convergence theorems for the IQR algorithm with convergence rates and error control. The results include eigenvalues, eigenvectors and invariant subspaces. 

\item[\, (2) {\it Algorithmic implementation}:] We prove that for infinite matrices with finitely many non-zero entries in each column, it is possible to implement the IQR algorithm exactly (on a finite machine) as if one had an infinite computer at one's disposal. This can be extended to implementing the IQR algorithm with error control for general invertible operators.
\item[\, (3) {\it SCI hierarchy classifications}:] As a result of (1) and (2), we provide new classification results for the SCI hierarchy. In particular, the convergence properties of the IQR algorithm capture key structures that allow for sharp $\Delta_1$ classification of the problem of computing extremal points in the spectrum. Moreover we establish sharp $\Sigma_1$ classification of the problem of computing spectra of subclasses of compact operators. 
\item[\, (4) {\it Numerical examples}:] Finally, we demonstrate the IQR algorithm and the proven convergence results on a variety of difficult problems in practical computation, illustrating how the IQR algorithm is much more than a theoretical concept. Moreover, the examples demonstrate that the IQR algorithm performs much better than what our theory covers and works on much larger classes of operators than our theorems predict. Hence, we are left with many open problems on the theoretical understanding of the potential and limitations of this algorithm. The computational experiments include examples from
\begin{itemize}
\item[(i)] Toeplitz/Laurent operators and their perturbations,
\item[(ii)] $PT$-symmetry in quantum mechanics,
\item[(iii)] Hopping sign model in sparse neural networks,
\item[(iv)] NSA Anderson model in superconductors.
\end{itemize}
\end{labeling}

\subsection{Connection to previous work}
Our results connect to many different approaches in the vast literature on spectral computation in infinite dimensions. The infinite-dimensional computational spectral problem is very different from the finite-dimensional computational eigenvalue problem, and even though the IQR algorithm is inspired by the finite-dimensional version, this paper solely focuses on the infinite-dimensional problem. Thus, the paper is aimed at the analysis and numerical analysis audience focusing on infinite-dimensional problems rather than the finite-dimensional numerical linear algebra discipline. 

\begin{labeling}{}
\itemsep0em
\item[\,\,\,{\it Finite sections}:]The IQR algorithm provides an alternative to the standard finite section method in several cases where it fails. Whereas the finite section method would extract a finite section from the infinite matrix and then apply, for example, the finite-dimensional QR algorithm, the IQR algorithm first performs the infinite QR iterations and then extracts a finite section. In general these two processes do not commute. The finite section method (or any derivative of it) cannot work in general because of the general classification results in the SCI hierarchy mentioned in \S \ref{intro}. Typically, it may provide false solutions. However, in the cases where it converges, it provides invaluable $\Delta_2$ classifications in the SCI hierarchy. The finite section method has often been viewed in connection with Toeplitz theory and the reader may want to consult the work by B{\"o}ttcher \cite{Albrecht_Fields, Bottcher_pseu}, B{\"o}ttcher \& Silberman \cite{Bottcher_book}, B{\"o}ttcher, Brunner, Iserles \& N{\o}rsett \cite{Arieh2}, Brunner, Iserles \& N{\o}rsett \cite{Arieh1}, Hagen, Roch \& Silbermann \cite{Silbermann22}, Lindner \cite{Lindner}, Marletta \cite{Marletta_pollution} and Marletta \& Scheichl \cite{Marletta_Spec_gaps}. From the operator algebra point of view the work of Arveson \cite{Arveson_cnum_lin94, Arveson_noncommute93,Arveson_role_of94} has been influential as well as the work of Brown \cite{brown2007quasi}.  

\item[\,\,\,{\it Infinite-dimensional Toda flow}:] Deift, Li and Tomei \cite{deift1985toda} provided the first results on the IQR algorithm in connection with Toda flows with infinitely many variables. Their results are purely functional analytic and do not take implementation and computability issues into account. However, these results provide the fundamentals of the IQR algorithm. In \cite{hansen2008} these results were expanded with a convergence result for eigenvectors corresponding to eigenvalues outside the essential numerical range for normal operators. Yet, this paper did not consider convergence rates, actual numerical calculation nor any classification results.

\item[\,\,\,{\it Infinite-dimensional QL algorithm}:] Olver, Townsend and Webb have provided a practical framework for infinite-dimensional linear algebra and foundational results on computations with infinite data structures \cite{webb2017spectra,Olver_Townsend_Proceedings, Olver_SIAM_Rev, Olver_code1, Olver_code2}. This includes efficient codes as well as theoretical results. The infinite-dimensional QL (IQL) algorithm is an important part of this program. The IQL algorithm is rather different from the IQR algorithm, although they are similar in spirit. In particular, both the implementation and the convergence results are somewhat contrasting.  

\item[\,\,\,{\it Infinite-dimensional spectral computation}:] The results in this paper follow in the long tradition of infinite-dimensional spectral computations. This field contains a vast literature that spans more than half a century, and the references that we have cited in the first paragraph of \S \ref{intro} represent a small sample. However, we would like to highlight the recent work by B{\"o}gli, Brown, Marletta, Tretter \& Wagenhofer \cite{Sabine} who were able to computationally confirm, with absolute certainty, a conjecture on a certain oscillatory behaviour of higher auto-ionizing resonances of atoms. Note that problems that are classified as $\Delta_1$ and $\Sigma_1$ problems in the SCI hierarchy may allow for computer assisted proofs.  
\end{labeling}

\subsection{Background and notation}
\label{background}

Here we briefly recall some definitions used in the paper. We will consider the canonical separable Hilbert space $\mathcal{H} = l^2(\mathbb{N})$ (the set of square summable sequences). Moreover, we write $\mathcal{B}(\mathcal{H})$ for the set of bounded operators on $\mathcal{H}.$ For orthogonal projections $E,F$, we will write $E\leq F$ if the range of $E$ is a subspace of the range of $F$. We denote the canonical orthonormal basis of $\mathcal{H}$ by $\{ e_j \}_{j \in \mathbb{N}}$, and if $\xi \in \mathcal{H}$ we write $\xi(j) = \langle \xi, e_j \rangle.$ Note that $T \in \mathcal{B}(\mathcal{H})$ is uniquely determined by its matrix elements $t_{ij} = \langle Te_j, e_i\rangle$. Hence we will use the words bounded operator and infinite matrix interchangeably. Given a sequence of operators $\{T_n\}$ we will use the notation 
$$
T_n \stackrel{\text{SOT}}{\longrightarrow} T, \qquad T_n \stackrel{\text{WOT}}{\longrightarrow} T
$$
to mean convergence in the strong and weak operator topology respectively. 
The spectrum of $T \in \mathcal{B}(\mathcal{H})$ will be denoted by $\sigma(T)$, and $\sigma_d(T)$ denotes the set of isolated eigenvalues with finite multiplicity (the discrete spectrum). 

In connection with the spectrum we need to recall some definitions which will appear in the statement of our theorems. We recall that, for $T \in \mathcal{B}(\mathcal{H})$, the essential spectrum\footnote{Of course in the case of non-normal $T$ there are different definitions of the essential spectrum. However, these differences will not matter regarding the results of this paper.} and the essential spectral radius are given by
$$
\sigma_{\mathrm{ess}}(T) = \{z\in\mathbb{C}:T-zI\text{ is not Fredholm}\}, \quad r_{\mathrm{ess}}(T) = \sup \{|z|: z \in \sigma_{\mathrm{ess}}(T)\}.
$$
Moreover, the numerical range and the essential numerical range of $T$ are defined by
\[
W(T) = \{\langle T\xi,\xi\rangle:\|\xi\| = 1\}, \quad W_e(T) = \bigcap_{K \, \text{compact}} \overline{W(T + K)}.
\]
In addition, we need the Hausdorff metric as defined by the following. Let $\mathcal{S}, \mathcal{T} \subset \mathbb{C}$, be compact. Then their
Hausdorff distance is 
\begin{equation}\label{eq:Hausdorff}
d_H(\mathcal{S},\mathcal{T}) = \max \{\sup_{\lambda \in \mathcal{S}}
d(\lambda,\mathcal{T}), \sup_{\lambda \in \mathcal{T}} d(\lambda,\mathcal{S})\},
\end{equation} 
where $d(\lambda,\mathcal{T}) = \inf_{\rho \in\mathcal{T}}|\rho-\lambda|.$ 
 We also recall a generalisation of the spectrum, known as the pseudospectrum.
Indeed, for $\epsilon>0$ define the $\epsilon$-pseudospectrum as
$$
\sigma_{\epsilon}(T)=\{z\in\mathbb{C} :\left\|(T - zI)^{-1}\right\|\geq{\epsilon}^{-1}\},
$$
where we interpret $\left\|S^{-1}\right\|$ as $+\infty$ if $S$ does not have a bounded inverse.
This is easier to compute than the spectrum, converges in the Hausdorff metric to the spectrum as $\epsilon\downarrow0$ and gives an indication of the instability of the spectrum of $T$. We shall use it as comparison for the IQR algorithm and as a means to detect spectral pollution for finite section methods.

Finally, we need a notion of convergence of subspaces. We follow the notation in \cite{kato2013perturbation}. Let $M \subset \mathcal{B}$ and $N \subset \mathcal{B}$ be two non-trivial closed subspaces of a Banach space $\mathcal{B}.$ The distance between them is defined by 
\[
\delta(M,N) = \sup_{\stackrel{x \in M}{\|x\| = 1}}\inf_{y \in N} \|x -y\|, \qquad \hat \delta(M,N) = \max[\delta(M,N),\delta(N,M)].
\]
Given subspaces $M$ and $\{M_k\}$ such that $\hat \delta(M_k,M)
\rightarrow 0$ as $k \rightarrow \infty,$ we will use the notation $M_k {\rightarrow} M$. If we replace $\mathcal{B}$ with a Hilbert space $\mathcal{H}$ we can express $\delta$ and $\hat \delta$ conveniently in terms of projections and operator norms. In particular, if $E$ and $F$ are the orthogonal projections onto subspaces $M \subset \mathcal{H}$ and $N \subset \mathcal{H}$ respectively then
\[
\delta(M,N) = \sup_{\stackrel{x \in M}{\|x\| = 1}}\inf_{y \in N}
\|x -y\| = \sup_{\stackrel{x \in M}{\|x\| = 1}}
\|F^{\perp}x\| = \|F^{\perp}E\|.
\]  
Since the operator $E-F = F^{\perp}E - FE^{\perp}$ is essentially the
direct sum of operators $F^{\perp}E \oplus (- FE^{\perp}),$ its norm
is $\hat \delta(M,N),$ i.e.
\begin{equation}\label{cam}
\hat \delta(M,N) = \max(\|F^{\perp}E\|, \|E^{\perp}F\|) 
= \max(\|F^{\perp}E\|, \|FE^{\perp}\|) = \|E-F\|.
\end{equation}
This allows us to extend the definition to allow the trivial subspace $\{0\}$ and gives rise to a metric on the set of all closed subspaces of $\mathcal{H}$ (first introduced by Krein and Krasnoselski in \cite{krein1947fundamental}). We also define the (maximal) subspace angle, $\phi(M,N)\in[0,\pi/2]$, between $M$ and $N$ by
\begin{equation}
\label{subspace_angle}
\sin\big(\phi(M,N)\big)=\hat \delta(M,N). 
\end{equation}
Finally, we will use two further well known properties in the Hilbert space setting. First, if $M$ and $N$ are both finite $l$-dimensional subspaces, then
\begin{equation}
\label{fd_property}
\delta(M,N)\leq l^{\frac{1}{2}}\delta(N,M),
\end{equation}
which shows that to prove convergence of finite dimensional subspaces, it is enough to prove $\delta$-convergence. Second, suppose we have
$$
M=\bigoplus_{j=1}^n M_j,\quad N^{(k)}=N_1^{(k)}+...+N_n^{(k)},
$$
where the $N_j^{(k)}$ need not be orthogonal. Then a simple application of H\"older's inequality yields
\begin{equation}
\label{sum_property}
\delta(M,N^{(k)})\leq \Big(\sum_{j=1}^n\delta(M_j,N_j^{(k)})^2\Big)^{\frac{1}{2}},
\end{equation}
which shows that if the dimensions of $M_j$ and $N_j^{(k)}$ are finite and equal, then to prove convergence $N^{(k)}\rightarrow M$ we only need to prove that $\delta(M_j,N_j^{(k)})\rightarrow0$ as $k\rightarrow\infty$. For further properties (including other notions of distances between subspaces) and a discussion on two projections theory, we refer the reader to the excellent article of B{\"o}ttcher and Spitkovsky \cite{bottcher2010gentle}.

\subsection{Organisation of the paper}

The paper is organised as follows. In Section \ref{IQR_s1} we define the IQR algorithm (simple codes are also provided in the appendix). Section \ref{conv_main} contains and proves our main theorems including convergence rates. The outcome is more elaborate than the finite-dimensional case, as the infinite-dimensional setting includes more intricate instances. Our key practical result is that, despite being an algorithm dealing with infinite amount of information, it can be implemented on any standard computer and this is discussed in Section \ref{implement}. The fact that the IQR algorithm can be computed allows for its use in order to provide new classification in the SCI hierarchy as discussed in Section \ref{class_thms}. In particular, we demonstrate $\Delta_1$ classification for the extremal part of the spectrum and dominant invariant subspaces, as well as $\Sigma_1$ results for spectra of certain classes of compact operators. Note that the general spectral problem for compact operators is not in $\Sigma_1$. The IQR algorithm and convergence theorems are demonstrated on a large collection of examples from the sciences on difficult computational spectral problems in Section \ref{sec_5} with comparisons to the finite section method. The IQR algorithm is also found to perform better than theory predicts and we conjecture conditions on the operator for this to be the case. Finally, we conclude with a discussion of the opportunities and limits of the IQR algorithm in Section \ref{conc_s}.

\section{The infinite-dimensional QR algorithm (IQR)}
\label{IQR_s1}
The IQR algorithm has existed as a pure mathematical concept for more
than thirty years and it first appeared in the paper ``Toda Flows with
Infinitely Many Variables'' \cite{deift1985toda} in 1985. However, the analysis in \cite{deift1985toda} 
covers only self-adjoint infinite matrices with real entries, and since the analysis is done from 
a pure mathematical perspective, the question regarding the actual numerical algorithm is left out. 
We will in this paper extend the analysis to more general operators and answer the crucial question: can one actually implement the IQR algorithm? The answer is affirmative, and we also prove convergence theorems, generalising the well known finite dimensional case.

\subsection{The QR decomposition} 
The QR decomposition is the core of the QR algorithm. If $T \in
\mathbb{C}^{n\times n},$ one may apply the Gram-Schmidt procedure
to the columns of $T$ and store these columns in a matrix $Q$ and this
gives us the QR decomposition
\begin{equation}\label{decom}
T = QR,
\end{equation} 
where $Q$ is a unitary matrix and $R$ upper triangular. It is no surprise that a QR decomposition should exist in the infinite dimensional case, however, we need more than just the existence. A key ingredient in the QR algorithm is the Householder transformation used for computational reasons (they are backwards stable). It is crucial that we can adopt these tools in the infinite dimensional setting. Our goal is to extend the construction of the QR decomposition, via Householder transformations, to infinite matrices and to find a way so that one can implement the procedure on a finite machine. To do this we need to introduce the concept of Householder reflections in the infinite-dimensional setting. 
\begin{definition}
A Householder reflection is an operator 
$S \in \mathcal{B}(\mathcal{H})$ of the form 
\begin{equation}\label{theS}
S = I - \frac{2}{\|\xi\|^2} \xi\otimes \bar \xi, \qquad  \xi \in
\mathcal{H},
\end{equation}
where $\bar \xi$ denotes the associated functional in $\mathcal{H}^*$ given by $x\rightarrow\langle x,\xi \rangle$. In the case where $\mathcal{H} = \mathcal{H}_1 \oplus \mathcal{H}_2$
and 
$I_i$ is the identity on $\mathcal{H}_i$ then 
\[
U = I_1 \oplus \left(I_2 - \frac{2}{\|\xi\|^2} \xi\otimes \bar \xi
\right) \qquad  \xi \in 
\mathcal{H}_2,  
\] 
will be called a Householder transformation. 
\end{definition}
A straightforward calculation shows that $S^* = S^{-1} = S$ and thus
also 
$U^* = U^{-1} = U.$
An important property of the operator $S$ is that if $\{e_j\}$ is 
an orthonormal basis for $\mathcal{H}$ and $\eta \in
\mathcal{H}$ then one can
choose $\xi \in \mathcal{H}$ such that 
\[
\langle S\eta, e_j\rangle =  \langle (I - \frac{2}{\|\xi\|^2}
\xi\otimes 
\bar \xi)\eta, e_j\rangle = 0, \qquad \forall j \neq 1.
\]
In other words, one can introduce zeros in the column below the diagonal entry. Indeed, if $\eta_1 = \langle \eta, e_1\rangle \neq 0$ one may
choose 
$\xi = \eta \pm \|\eta\|\zeta,$
where 
$\zeta =
\eta_1/|\eta_1|e_1$ and if $\eta_1 = 0$ choose $\xi = \eta \pm \|\eta\|e_1.$ The following theorem gives the existence of a QR decomposition, even in the case where the operator is not invertible.

\begin{theorem}[\cite{hansen2008}]\label{QRthrm}
Let $T$ be a bounded operator on a separable Hilbert space
$\mathcal{H}$ and let
$\{e_j\}_{j\in\mathbb{N}}$ be an orthonormal basis for
$\mathcal{H}\cong l^2(\mathbb{N}).$ Then there exist an isometry $Q$ such that $T
= QR$ where $R$ is upper triangular with respect to $\{e_j\}$. 
Moreover, 
$$
Q = \underset{n \rightarrow \infty}{\text{SOT-lim \,}}
V_n
$$ 
where $V_n = U_1 \cdots U_n$ are unitary and each $U_j$ is a Householder transformation. 
\end{theorem}

\subsection{The IQR algorithm}

Let $T \in \mathcal{B}(\mathcal{H})$ be invertible and let $\{e_j\}$
be an orthonormal basis for $\mathcal{H}$. By Theorem
\ref{QRthrm} we have $T = QR,$ where $Q$ is an isometry and $R$ is upper
triangular with respect to $\{e_j\}.$ Since $T$ is invertible, $Q$ is in fact unitary. Consider the following
construction of unitary operators $\{\hat Q_k\}$ and upper triangular
(w.r.t. $\{e_j\}$) operators $\{\hat R_k\}.$ Let $T = Q_1R_1$ be a QR
decomposition of $T$ and define $T_1 =
R_1Q_1.$ Then QR factorize $T_1 = Q_2R_2$ and define
$T_2 = 
R_2Q_2.$ The recursive procedure becomes 
\begin{equation}\label{rec}
T_{m-1} = Q_mR_m, \quad T_m = R_mQ_m.
\end{equation}
Now define 
\begin{equation}\label{qdef}
\hat Q_m = Q_1Q_2 \hdots Q_m, \quad \hat R_m = R_mR_{m-1}\hdots R_1. 
\end{equation}
This is known as the QR algorithm and is completely analogous to the finite dimensional case. Note also that we have $T_n = \hat{Q}^*_n T \hat{Q}_n.$ In the finite dimensional case and under favourable conditions $\hat{Q}_n^*T\hat{Q}_n$ converges to a diagonal operator and the columns of $\hat{Q}_n$ converge to the corresponding eigenvectors as $n\rightarrow\infty$ (see Theorem \ref{finite} below). We will see that the IQR algorithm behaves similarly for the extreme parts of the spectrum.

\begin{definition}
Let $T \in \mathcal{B}(\mathcal{H})$ be invertible and let $\{e_j\}$
be an orthonormal basis for $\mathcal{H}$. The sequences
$\{\hat Q_j\}$ and $\{\hat R_j\}$ constructed as in (\ref{rec}) and
(\ref{qdef}) will be 
called a $Q$-sequence and an $R$-sequence of $T$ with respect to
$\{e_j\}.$ 
\end{definition}

\begin{remark}
Note that since the Householder transformations used in the proof of Theorem \ref{QRthrm} are unique up to a $\pm$ sign, we will with some abuse of language refer to the QR decomposition constructed as {\it the} QR decomposition. In general for an invertible operator, the IQR algorithm is uniquely defined up to phase - see Section \ref{invert_comp}. This will not be a problem for our theorems or numerical examples.
\end{remark}

The following observation will be useful in the later
developments. From the construction in (\ref{rec}) and (\ref{qdef}) we
get 
\[
T = Q_1R_1 = \hat Q_1 \hat R_1,
\]
\[
T^2 = Q_1R_1Q_1R_1 = Q_1Q_2R_2R_1 = \hat Q_2 \hat R_2,
\]
\[
T^3 = Q_1R_1Q_1R_1Q_1R_1 = Q_1Q_2R_2Q_2R_2R_1 = Q_1Q_2Q_3R_3R_2R_1 =
\hat Q_3 \hat R_3.
\]
An easy induction gives us that 
\begin{equation}
\label{pwr_ref}
T^m = \hat Q_m \hat R_m.
\end{equation}
Note that $\hat R_m$ must be upper triangular with respect to
$\{e_j\}_{j\in\mathbb{N}}$ since $R_j, \, j \leq m$ is upper triangular with respect
to 
$\{e_j\}_{j\in\mathbb{N}}.$ Also, if $T$ is invertible then $\langle R
e_i,e_i\rangle \neq 0.$ From this it follows immediately that
\begin{equation}\label{usefulspan}
\mathrm{span}\{T^m e_j\}_{j=1}^J = \mathrm{span} \{\hat Q_me_j\}_{j=1}^J,
\quad J \in \mathbb{N}. 
\end{equation}
%In what follows we will drop the $\hat{\cdot}$ notation for convenience.

%\section{Main results}
 %Finally, we link our results to the existing theory of the Toda flow in infinite dimensions. 

\section{Convergence theorems}
\label{conv_main}

%\subsection{Normal operators}

In finite dimensions we have the following well known theorem:

\begin{theorem}[Finite dimensions]\label{finite}
Let $T \in \mathbb{C}^{N \times N}$ 
be a normal matrix with eigenvalues satisfying $|\lambda_1| >
\hdots > 
|\lambda_N|$. 
Let $\{Q_m\}$ be a $Q$-sequence of
unitary operators. Then (up to re-ordering of the basis)
$$
Q^*_m T Q_m \longrightarrow \bigoplus_{j=1}^N \lambda_j e_j \otimes e_j, \qquad \text{as }m \rightarrow \infty.
$$ 
\end{theorem}

In this section we will address the convergence of the IQR algorithm for normal operators under similar assumptions and prove an analogue of Theorem \ref{finite} in infinite dimensions (Theorem \ref{new_QR}). As well as this, and for more general operators $T$ that aren't necessarily normal, we address block convergence (Theorem \ref{non_normal_prop}), relevant when the eigenvalues do not have distinct moduli, and convergence to (dominant) invariant subspaces (Theorem \ref{non_normal_prop2}).

\subsection{Preliminary definitions and results}
\label{prelims}

To state and prove our theorems we need some preliminary results. The reader only interested in the results themselves is referred to Section \ref{main_res_sectIII}. If $T$ is a normal operator, we will use $\chi_S(T)$ to denote the indicator function of the set $S$ defined via the functional calculus. Without loss of generality, we deal with the Hilbert space $\mathcal{H}=l^2(\mathbb{N})$ and the canonical orthonormal basis $\{e_j\}_{j\in\mathbb{N}}$. Our first set of results concerns convergence of spanning sets under power iterations and is analogous to the finite dimensional case. The following proposition can be found in \cite{hansen2008} and together with Lemma \ref{lemma_span} below, these are the only results we will use from \cite{hansen2008}.
 
\begin{proposition}\label{subspace}
Suppose that $T\in\mathcal{B}(\mathcal{H})$ is normal, is invertible and that $\sigma(T)=\omega\cup\Psi$ is a disjoint union such that $\omega=\{\lambda_i\}_{i=1}^N$ consists of finitely many isolated eigenvalues of $T$ with $\left|\lambda_1\right|>\left|\lambda_2\right|>...>\left|\lambda_N\right|$. Suppose further that $\mathrm{sup}\{\left|z\right|:z\in\Psi\}<\left|\lambda_N\right|$. Let $l\in \mathbb{N}$ and suppose that $\{\xi_i\}_{i=1}^l$ are linearly independent vectors in $\mathcal{H}$ such that $\{\chi_{\omega}(T)\xi_i\}_{i=1}^l$ are also linearly independent. Then
\begin{enumerate}
	\item[(i)] The vectors $\{T^k\chi_{\omega}(T)\xi_i\}_{i=1}^l$ are linearly independent and there exists an $l$-dimensional subspace $B\subset\mathrm{ran}\chi_{\omega}(T)$ such that
	$$
	\mathrm{span}\{T^k\xi_i\}_{i=1}^l\rightarrow B,\quad \text{as }k\rightarrow\infty.
	$$
	\item[(ii)] If
	$$
	\mathrm{span}\{T^k\xi_i\}_{i=1}^{l-1}\rightarrow D\subset\mathcal{H},\quad \text{as }k\rightarrow\infty,
	$$
	where $D$ is an $(l-1)$-dimensional subspace, then
	$$
  \mathrm{span}\{T^k
  \xi_i\}_{i=1}^l \rightarrow D \oplus
  \mathrm{span}\{\xi\},  \quad \text{as }k\rightarrow \infty,
  $$ 
where $\xi \in \mathrm{ran}\chi_{\omega}(T)$ is an eigenvector of $T$. 
\end{enumerate}
\end{proposition}

In order to extend this proposition to describe rates of convergence and prove our main theorems, we need to describe the space $B$ in more detail. This is done inductively as follows. The first step is to choose $\nu_{1,1}\in \{\lambda_i\}_{i=1}^N$ of maximum modulus such that
$$
\mathrm{span}\{\chi_{\nu_{1,1}}(T)\xi_1\} \neq \{0\}.
$$
We then let $\xi_{1,1}$ be a linear multiple of $\xi_1$ such that $\chi_{\nu_{1,1}}(T)\xi_{1,1}$ has norm one. Now suppose that at the $m$-th stage we have constructed vectors $\{\xi_{m,i}\}_{i=1}^m$ with the same linear span as $\{\xi_{i}\}_{i=1}^m$ and such that there exist $\{\nu_{m,j}\}_{j=1}^{s_m}\subset \{\lambda_i\}_{i=1}^N$ with the following properties. After re-ordering the vectors $\{\xi_{m,i}\}_{i=1}^m$ if necessary, there exist integers $0=k_{m,0}<k_{m,1}<k_{m,2}<...<k_{m,s_m}=m$ such that
\begin{enumerate}
	\item[(1)] $\left|\nu_{m,s_m}\right|<\left|\nu_{m,s_m-1}\right|<...<\left|\nu_{m,1}\right|.$
	\item[(2)] $\chi_{\lambda}(T)\xi_{m,i}=0$ if $i>k_{m,j}$ and $\lambda\in\{\lambda_i\}_{i=1}^N$ has $\left|\lambda\right|>\left|\nu_{m,j+1}\right|$.
	\item[(3)] $\{\chi_{\nu_{m,j}}(T)\xi_{m,i}\}_{i=k_{m,j-1}+1}^{k_{m,j}}$ are orthonormal.
\end{enumerate}
We seek to add the space spanned by the vector $\xi_{m+1}$ whilst preserving these properties.

First we deal with (2). Let $\eta_{m+1}\in\{\lambda_i\}_{i=1}^N$ be of maximal modulus such that $\chi_{\{\lambda_1,...,\eta_{m+1}\}}(T)\xi_{m+1}\notin\mathrm{span}\{\chi_{\{\lambda_1,...,\eta_{m+1}\}}(T)\xi_{j}\}_{j=1}^{m}$. If $\left|\eta_{m+1}\right|<\left|\nu_{m,1}\right|$ then let $t({m+1})$ be maximal such that $\left|\eta_{m+1}\right|<\left|\nu_{m,t({m+1})}\right|$. We then choose complex numbers $\{a_{m,j}\}_{j=1}^{k_{m,t({m+1})}}$ such that writing
$$
\tilde \xi_{m+1,m+1}=\xi_{m+1}+\sum_{j=1}^{k_{m,t({m+1})}}a_{m,j}\xi_{m,j}
$$
we have that $\chi_{\lambda}(T)\tilde \xi_{m+1,m+1}=0$ if $\lambda\in\{\lambda_i\}_{i=1}^N$ has $\left|\lambda\right|>\left|\eta_{m+1}\right|$. Note that by (2), (3) and the definition of $\eta_{m+1}$, the coefficients $a_{m,j}$ are determined uniquely in terms of $\{\xi_{m,i}\}_{i=1}^{k_{m,t({m+1})}}$. If $\left|\eta_{m+1}\right|\geq\left|\nu_{m,1}\right|$ then let $t({m+1})=0$ and we set $\tilde \xi_{m+1,m+1}=\xi_{m+1}$. In this case we still have that $\chi_{\lambda}(T)\tilde \xi_{m+1,m+1}=0$ if $\lambda\in\{\lambda_i\}_{i=1}^N$ has $\left|\lambda\right|>\left|\eta_{m+1}\right|$.

We then define $\xi_{m+1,j}=\xi_{m,j}$ for $1\leq j\leq m$ and now deal with (3). If $\eta_{m+1}\notin\{\nu_{m,j}\}_{j=1}^{s_m}$ then let $\xi_{m+1,m+1}$ be a linear multiple of $\tilde \xi_{m+1,m+1}$ such that $\chi_{\eta_{m+1}}(T)\xi_{m+1,m+1}$ has norm 1 and we let $\{\nu_{m+1,j}\}_{j=1}^{s_m+1}$ be a re-ordering of $\{\nu_{m,j}\}_{j=1}^{s_m}\cup\{\eta_{m+1}\}$. Otherwise, we have $\eta_{m+1}=\nu_{m,t({m+1})+1}$ and we apply Gram-Schmidt to
$$
\{\chi_{\nu_{m,t({m+1})+1}}(T)\xi_{m+1,i}\}_{i=k_{m,t({m+1})}+1}^{k_{m,t({m+1})+1}}\cup\{\chi_{\nu_{m,t({m+1})+1}}(T)\tilde\xi_{m+1,m+1}\}
$$
(without changing $\{\xi_{m+1,i}\}_{i=k_{m,t({m+1})}+1}^{k_{m,t({m+1})+1}}$). Note that by (2) and the definition of $\eta_{m+1}$ these vectors are linearly independent. This gives $\xi_{m+1,m+1}$ such that
$$
\{\chi_{\nu_{m,t({m+1})+1}}(T)\xi_{m+1,i}\}_{i=k_{m,t({m+1})}+1}^{k_{m,t({m+1})+1}}\cup\{\chi_{\nu_{m,t({m+1})+1}}(T)\xi_{m+1,m+1}\}
$$
are orthonormal and $\chi_{\lambda}(T)\xi_{m+1,m+1}=0$ if $\lambda\in\{\lambda_i\}_{i=1}^N$ has $\left|\lambda\right|>\left|\nu_{m,t({m+1})+1}\right|.$ After re-ordering indices if necessary, we see that (1)-(3) now hold for $m+1$.

After $l$ steps the above process terminates giving a new basis $\{\tilde\xi_i\}_{i=1}^l=\{\xi_{l,i}\}_{i=1}^l$ for $\mathrm{span}\{\xi_{i}\}_{i=1}^l$ along with $\{\nu_{j}\}_{j=1}^n=\{\nu_{l,j}\}_{j=1}^n\subset\{\lambda_i\}_{i=1}^N$ and $0=k_0<k_1<k_2<...<k_n=l$ such that
\begin{enumerate}
  \item[(i)] $\left|\nu_{n}\right|<\left|\nu_{n-1}\right|<...<\left|\nu_{1}\right|.$
	\item[(ii)] $\chi_{\lambda}(T)\tilde\xi_{i}=0$ if $i>k_{j}$ and $\lambda\in\{\lambda_i\}_{i=1}^N$ has $\left|\lambda\right|>\left|\nu_{j+1}\right|$.
	\item[(iii)] $\{\chi_{\nu_{j}}(T)\tilde\xi_{i}\}_{i=k_{j-1}+1}^{k_{j}}$ are orthonormal.
\end{enumerate}

The subspace $B$ can then be described as
$$
B = \bigoplus_{j=1}^n \mathrm{span}\{\chi_{\nu_j}(T)
\tilde\xi_{i}\}_{i=k_{j-1}+1}^{k_j}. 
$$
\begin{definition}
With respect to the above construction we define the following:
\begin{equation}
\label{need_ejs}
E_j :=\mathrm{span}\{\chi_{\nu_j}(T)\tilde\xi_{i}\}_{i = k_{j-1}+1}^{k_j},\quad Z(T,\{\xi_j\}_{j=1}^l):=\Big(\sum_{i=1}^{l}(\|\tilde \xi_i\|^2-1)\Big)^{\frac{1}{2}}.
\end{equation}
\end{definition}

Since the Gram-Schmidt process is defined uniquely up to phases we see that $Z(T,\{\xi_j\}_{j=1}^l)$ is well-defined. The above construction also shows that if $\{\chi_{\omega}(T)\xi_i\}_{i=1}^{l+1}$ are linearly independent then
$$Z(T,\{\xi_j\}_{j=1}^{l+1})\geq Z(T,\{\xi_j\}_{j=1}^l).$$
We can now prove the following refinement of Proposition \ref{subspace}:

\begin{proposition}
\label{rate_extension}
Suppose the assumptions of Proposition \ref{subspace} hold. Let $J\leq N$ be minimal such that $\{\chi_{\{\lambda_1,...,\lambda_J\}}(T)\xi_i\}_{i=1}^l$ are linearly independent. Set
\begin{align*}
\rho&=\mathrm{sup}\{\left|z\right|:z\in\Psi\cup\{\lambda_{J+1},...,\lambda_N\}\},\\
r&=\max\{\left|\lambda_2/\lambda_1\right|,...,\left|\lambda_J/\lambda_{J-1}\right|,\rho/\left|\lambda_J\right|\}.
\end{align*}
Then $r<1$ and ${\delta}(B,\mathrm{span}\{T^k\xi_i\}_{i=1}^l)\leq Z(T,\{\xi_j\}_{j=1}^l)r^k$. Since the spaces are $l$-dimensional, it follows from (\ref{fd_property}) that we have the convergence rate
$$
\hat{\delta}(B,\mathrm{span}\{T^k\xi_i\}_{i=1}^l)\leq Z(T,\{\xi_j\}_{j=1}^l)l^{\frac{1}{2}}r^k.$$
\end{proposition}

\begin{proof}
Consider the subspaces
$$
E_j^k = \mathrm{span}\{T^k
\tilde \xi_{i}\}_{i=k_{j-1}+1}^{k_j}.
$$
Let $\zeta = \sum_{i=k_{j-1}+1}^{k_j} \alpha_i\chi_{\nu_j}(T)\tilde\xi_{i}\in E_j$ be a unit vector (hence $\sum_{i=k_{j-1}+1}^{k_j}\left|\alpha_i\right|^2=1$) and consider
$$
\eta_k = \sum_{i=k_{j-1}+1}^{k_j} \alpha_iT^k\tilde\xi_{i}/\nu^k_j\in E_j^k.
$$
By construction, we have for any such $\tilde\xi_{i}$ in the above sum that
$$
\tilde\xi_{i} = (\chi_{\nu_j}(T) + \chi_{\theta_j}(T))\tilde\xi_{i},
\qquad  
\theta_j = \{\lambda \in \sigma(T) : 
|\lambda| < |\nu_j|\}.
$$     
This gives $T^k\tilde\xi_{i} = \nu^k_j\chi_{\nu_j}(T)\tilde\xi_{j,i} 
 + T^k\chi_{\theta_j}(T)\tilde\xi_{i}.$ Now, by the assumption on $\sigma(T),$ we have 
\[
\rho_j = \sup\{|z|: z \in
  \theta_j\} < |\nu_j|.
\] Thus, since 
\[
\|T^k \chi_{\theta_j}(T)\tilde\xi_{i}\|/|\nu_j^k| < 
|\rho_j/\nu_j|^k \|\chi_{\theta_j}(T)\tilde\xi_{i}\|,
\]
we have 
$$
\left\|\zeta-\eta_k\right\|\leq |\rho_j/\nu_j|^k\sum_{i=k_{j-1}+1}^{k_j} \left|\alpha_i\right|\|\chi_{\theta_j}(T)\tilde\xi_{i}\|\leq \Big(\sum_{i=k_{j-1}+1}^{k_j}(\|\tilde\xi_i\|^2-1)\Big)^{\frac{1}{2}}r^k.
$$
Here we have used H{\"o}lder's inequality together with the fact that $\|\chi_{\theta_j}(T)\tilde\xi_{i}\|^2=\|\tilde\xi_{i}\|^2-1$ by orthonormality of $\{\chi_{\nu_j}(T)\tilde\xi_i\}_{i=k_{j-1}+1}^{k_j}$. The right hand side gives an upper bound for $\delta(E_j,E_j^k)$. Analogous rates of convergence hold for the other subspaces and from (\ref{sum_property}) we have
\begin{equation}
\label{perp_argument}
\delta(B,\mathrm{span}\{T^k\tilde\xi_i\}_{i=1}^l)\leq Z(T,\{\xi_j\}_{j=1}^l)r^k,
\end{equation}
since the spaces $E_j$ are orthogonal.
\end{proof}

For the rest of this section we shall assume the following:
\vspace{2mm}
\begin{tcolorbox}
\begin{itemize}
	\item[(A1)] $T \in \mathcal{B}(\mathcal{H})$ is an invertible normal operator and $\{e_j\}_{j\in\mathbb{N}}$ an orthonormal basis for $\mathcal{H}$. $\{Q_k\}$ and $\{R_k\}$ are $Q$- and $R$-sequences of $T$ with respect to the basis $\{e_j\}_{j\in\mathbb{N}}.$
	\item[(A2)] $\sigma(T) = \omega \cup \Psi$
such that $\omega \cap \Psi = \emptyset$ and $\omega =\{\lambda_i\}_{i=1}^N,$ where the $\lambda_i$s are isolated eigenvalues with (possibly infinite) multiplicity $m_i$. Let $M=m_1+...+m_N=\mathrm{dim}(\mathrm{ran}\chi_{\omega}(T))$ and suppose that $|\lambda_1| > \hdots > |\lambda_N|.$ Suppose further that $\sup\{|\theta|:\theta \in \Psi\} < |\lambda_N|.$ 
\end{itemize}
\end{tcolorbox}
\vspace{2mm}

To apply Propositions \ref{subspace} and \ref{rate_extension} to prove the main result Theorem \ref{new_QR}, we need to take care of the case that some of the $e_j$ may have $\chi_{\omega}(T)e_j=0$.

\begin{definition}
\label{define_the_basis}
Suppose that (A1) and (A2) hold and let $K\in\mathbb{N}\cup\{\infty\}$ be minimal with the property that $\dim(\mathrm{span}\{\chi_{\omega}(T)e_j\}_{j=1}^K) = M.$ Define 
\begin{align*}
&\Lambda_{\omega} = \{e_j:\chi_{\omega}(T)e_j \neq 0, j\leq K\}, \\
&\Lambda_{\Psi}= \{e_j:\chi_{\omega}(T)e_j = 0, j \leq K\},\\
&\tilde \Lambda_{\omega} = \{e_j \in
\Lambda_{\omega}:\chi_{\omega}(T)e_j \in
\mathrm{span}\{\chi_{\omega}(T)e_i\}_{i=1}^{j-1}\}.
\end{align*}
Define also the corresponding subset $\{\hat e_j\}_{j=1}^M\subset\{e_j\}_{j=1}^K$ such that $\{\hat e_j\}_{j=1}^M = \Lambda_{\omega}\setminus \tilde \Lambda_{\omega}$ and such that writing $\hat{e}_j=e_{p_j}$, the $p_j$ are increasing.
\end{definition}

Note that we have the following decomposition of $T$ into 
\[
T = \left(\sum_{j=1}^M \lambda_{c_j} \,\xi_j\otimes \bar \xi_j \right)\oplus
\chi_{\Psi}(T)T, \quad \lambda_{c_j} \in \omega, 
\]
where $\{\xi_j\}_{j=1}^M$ is an orthonormal set of eigenvectors of $T$. The following simple lemma extends Lemma 39 in \cite{hansen2008} to infinite $M$ but the proof is verbatim so omitted.

\begin{lemma}\label{lemma_span}
If $e_m \in
  \Lambda_{\Psi} \cup
  \tilde \Lambda_{\omega},$ then
\[
\mathrm{span}\{\chi_{\omega}(T) q_{k,j}\}_{j=1}^m =
  \mathrm{span}\{\chi_{\omega}(T) \hat q_{k,j}\}_{j=1}^{s(m)}, \quad 
 q_{k,j} = Q_k  e_j, \quad \hat q_{k,j} = Q_k \hat e_j,
\] 
where $s(m)$ 
is the largest integer such that $\{\hat e_j\}_{j=1}^{s(m)}
  \subset \{e_j\}_{j=1}^m.$ 
\end{lemma}

The following theorem is the key step of the proof of Theorem \ref{new_QR} and concerns convergence to the eigenvectors of $T$.

\begin{theorem}\label{QR-theorem}
Assume (A1) and (A2) and define
$$
\rho=\mathrm{sup}\{\left|z\right|:z\in\Psi\},\quad r=\max\{\left|\lambda_2/\lambda_1\right|,...,\left|\lambda_N/\lambda_{N-1}\right|,\rho/\left|\lambda_N\right|\}.
$$
Then there exists a collection of orthonormal eigenvectors $\{\hat q_j\}_{j=1}^M\subset\mathrm{ran}\chi_{\omega}(T)$ of $\ T$ and collections of constants $A(m)$, $B(j)$ and $C(\mu)$ such that
\begin{enumerate}
	\item[(a)] If $e_m\in\Lambda_{\Psi}\cup\tilde \Lambda_{\omega}$ and $\mu$ is maximal with $p_\mu<m$ (recall that $\hat e_j=e_{p_j}$), then we have
	\begin{equation}
	\label{ABOUND}
	\left\|\chi_\omega(T)q_{k,m}\right\|\leq A(m) Z(T,\{\hat e_j\}_{j=1}^\mu)r^k.
	\end{equation}
	In the case that $m<p_1$, we interpret this as $\left\|\chi_\omega(T)q_{k,m}\right\|=0$ which holds from Lemma \ref{lemma_span}.
	\item[(b)] For any $j<M+1$,
	\begin{equation}
	\label{BBOUND}
	\hat\delta(\mathrm{span}\{\hat q_j\},\mathrm{span}\{\hat q_{k,j}\})\leq B(j)Z(T,\{\hat e_i\}_{i=1}^j) r^k.\end{equation}
	\item[(c)] For any $\mu<M+1$,
	\begin{equation}
	\label{CBOUND}
	\delta(\mathrm{span}\{\hat q_{j,k}\}_{j=1}^{\mu},\mathrm{span}\{\hat q_{j}\}_{j=1}^{\mu})\leq C(\mu)Z(T,\{\hat e_j\}_{j=1}^\mu) r^k
	\end{equation}
	and hence
	\begin{equation}
	\label{CBOUND2}
	\hat\delta(\mathrm{span}\{\hat q_{j,k}\}_{j=1}^{\mu},\mathrm{span}\{\hat q_{j}\}_{j=1}^{\mu})\leq \mu^{\frac{1}{2}}C(\mu)Z(T,\{\hat e_j\}_{j=1}^\mu) r^k.
	\end{equation}
\end{enumerate}
Here, as in Lemma \ref{lemma_span}, $q_{k,j} = Q_k  e_j$ and $\hat q_{k,j} = Q_k \hat e_j$. Finally, if $M$ is finite then we must have $\mathrm{span}\{\hat q_j\}_{j=1}^M = \mathrm{ran}\chi_{\omega}(T).$
\end{theorem}

We will provide an inductive proof of Theorem \ref{QR-theorem} which requires the following for the inductive step of part (a).

\begin{lemma}\label{noidea} Assume the conditions in the statement of Theorem \ref{QR-theorem}. Suppose also that (b) in Theorem \ref{QR-theorem} holds for $j=1,...,\mu$ and that (c) holds for a given $\mu<M$. Let $e_{p_{\mu+1}} = \hat  e_{\mu+1}$, then if $e_m \in
  \Lambda_{\Psi} \cup
  \tilde \Lambda_{\omega},$ where $m < p_{\mu+1}$, (\ref{ABOUND}) also holds with
	$$
	A(m)=\Big\{\sum_{j=1}^{\mu}\big[C(\mu)+B(j)\big]^2\Big\}^{\frac{1}{2}}+C(\mu).
	$$
\end{lemma}

\begin{proof}
First note that from (\ref{usefulspan}), invertibility of $T$ and the fact that $\{\chi_{\omega}(T)\hat e_j\}_{j=1}^\mu$ are linearly independent, it must hold that $\{\chi_{\omega}(T)\hat q_{k,j}\}_{j=1}^{\mu}$ are linearly independent also. Then by using the assumptions stated and the fact that $\chi_\omega(T)\hat q_j=\hat q_j$ we have
\begin{align*}
 \delta(\mathrm{span}\{\chi_{\omega}(T)\hat q_{k,j}\}_{j=1}^{\mu}, \mathrm{span}\{ \hat
 q_j\}_{j=1}^{\mu})&\leq\delta(\mathrm{span}\{\hat q_{k,j}\}_{j=1}^{\mu}, \mathrm{span}\{ \hat
 q_j\}_{j=1}^{\mu})\\
&\leq C(\mu)Z(T,\{\hat e_j\}_{j=1}^\mu)r^k.
% \mu^{\frac{1}{2}}\delta(\mathrm{span}\{ \hat q_j\}_{j=1}^{\mu},\mathrm{span}\{\chi_{\omega}(T)\hat q_{k,j}\}_{j=1}^{\mu} )\\
%&\leq\mu^{\frac{1}{2}}\delta(\mathrm{span}\{\hat q_{j}\}_{j=1}^{\mu}, \mathrm{span}\{ \hat
% q_{k,j}\}_{j=1}^{\mu})\\
%&\leq \mu C(\mu)Z(T,\{\hat e_j\}_{j=1}^\mu)r^k.
\end{align*}
Also, we have that $s(m) \leq \mu$ and Lemma \ref{lemma_span} implies
$$
\mathrm{span}\{\chi_{\omega}(T) q_{k,j}\}_{j=1}^m =
  \mathrm{span}\{\chi_{\omega}(T) \hat q_{k,j}\}_{j=1}^{s(m)}\subset \mathrm{span}\{\chi_{\omega}(T) \hat q_{k,j}\}_{j=1}^{\mu}.
$$
Using the fact that $\left\|\chi_{\omega}(T)q_{k,m}\right\|\leq 1$ and the definition of $\delta$ (along with the fact that $\mathrm{span}\{ \hat q_j\}_{j=1}^{\mu}$ is finite dimensional), it follows that there exists some $v_k=\sum_{j=1}^\mu\beta_{j,k}{\hat q_j}\in\mathrm{span}\{\hat
 q_j\}_{j=1}^{\mu}$ with $\left\|v_k\right\|\leq 1$ and
\begin{equation}
\label{squ}
\left\|\chi_{\omega}(T)q_{k,m} -v_k\right\|\leq  C(\mu)Z(T,\{\hat e_j\}_{j=1}^\mu)r^k.
\end{equation}
%But then for $j\leq\mu$ we have
%$$
%\left|\langle v_k,\hat q_j\rangle\right|\leq D(\mu)Cr^k+\left|\langle \chi_{\omega}(T)q_{k,m},\hat q_j\rangle\right|.
%$$
We also have from assumption (b) that
\begin{equation}
\label{sq2}
\left|\langle \chi_{\omega}(T)q_{k,m},\hat q_j\rangle\right|=\left|\langle q_{k,m},\hat q_j\rangle\right|\leq B(j)Z(T,\{\hat e_i\}_{i=1}^j)r^k + \left|\langle q_{k,m},\hat q_{k,j}\rangle\right|=B(j)Z(T,\{\hat e_i\}_{i=1}^j)r^k,
\end{equation}
since $q_{k,m}$ is orthogonal to $\hat q_{k,j}$. This together with (\ref{squ}) gives that $\left|\beta_{j,k}\right|\leq\big[C(\mu)+B(j)\big]Z(T,\{\hat e_j\}_{j=1}^\mu)r^k$. Hence we must have
$$
\left\|v_k\right\|\leq \Big\{\sum_{j=1}^{\mu}\big[ C(\mu)+B(j)\big]^2\Big\}^{\frac{1}{2}}Z(T,\{\hat e_j\}_{j=1}^\mu)r^k.
$$
Using (\ref{squ}) again then gives the result. Note that we have used orthonormality of $\{\hat q_j\}_{j=1}^{\mu}$ which will be proven as part of the induction.
\end{proof}

\begin{proof}[\textbf{Proof of Theorem \ref{QR-theorem}}:]

We begin with the initial step of the induction for (b) and (c). Note that (a) trivially holds by construction with $A(m)=0$ for any $m<p_1$ where $e_{p_1}=\hat e_{1}$ and this provides the initial step for (a).

By Propositions \ref{subspace} and \ref{rate_extension}, there exists a unit eigenvector $\hat q_1\in\mathrm{ran}\chi_{\omega}(T)$ such that
$$
\delta(\mathrm{span}\{\hat q_1\},\mathrm{span}\{T^k \hat e_1\})\leq Z(T,\{\hat e_1\})r^k.
$$
Since $\mathrm{span}\{T^k\hat  e_1\}\subset  \mathrm{span}\{T^ke_i\}_{i=1}^{p_1}$, this implies that
$$
\delta( \mathrm{span}\{\hat q_1\}, \mathrm{span}\{T^ke_i\}_{i=1}^{p_1}) \leq Z(T,\{\hat e_1\})r^k.
$$
Thus, it follows that 
\begin{equation}\label{n2}
\delta(\mathrm{span}\{\hat q_1\},\mathrm{span}\{q_{k,i}\}_{i=1}^{p_1}) =
\delta( \mathrm{span}\{\hat q_1\},\mathrm{span}\{T^ke_i\}_{i=1}^{p_1}) \leq Z(T,\{\hat e_1\})r^k,
\end{equation}  
from (\ref{usefulspan}). Note that $\{q_{k,i}\}_{i=1}^{p_1}$ are orthonormal (recall that $Q_k$ is unitary) and hence by (\ref{n2}) there exists some coefficients $\alpha_{k,i}$ with $\sum_{i=1}^{p_1}|\alpha_{k,i}|^2 \leq 1$ such that defining $\tilde \eta_k =  \sum_{i=1}^{p_1}\alpha_{k,i}q_{k,i}$ we have
\begin{equation}
\label{eta_approx2}
\left\|\hat q_1-\tilde \eta_k\right\|\leq Z(T,\{\hat e_1\})r^k.
\end{equation}
If $e_m \in \Lambda_{\Psi} \cup  \tilde \Lambda_{\omega},$ where $m < {p_1}$ then by Lemma \ref{lemma_span} $\langle q_{k,m},\hat q_1\rangle=0$. It follows that we must have
$$
\delta(\mathrm{span}\{\hat q_1\},\mathrm{span}\{\hat q_{k,1}\})\leq\left\|\hat q_1-\alpha_{k,p_1}\hat q_{k,1}\right\|\leq Z(T,\{\hat e_1\})r^k.
$$
Hence we can take $B(1)=1$ and $C(1)=1$ in (b) and (c) respectively which completes the initial step.

For the induction step we will argue simultaneously for (a), (b) and (c) using induction on $\mu$. Suppose that (a) holds for $m<p_{\mu}$ with $e_{p_\mu}=\hat e_\mu$ together with (b) and (c) for $j\leq \mu$ and some $\mu<M$. Let $e_{p_{\mu+1}}=\hat e_{\mu+1}$ then we can use Lemma \ref{noidea} to extend (a) to all $m<p_{\mu+1}$ and this provides the step for (a). For (b), we note that Propositions \ref{subspace} and \ref{rate_extension} imply that
\begin{equation}\label{a}
\delta(\mathrm{span}\{\hat q_i\}_{i=1}^{\mu} 
  \oplus \mathrm{span}\{\xi\},\mathrm{span}\{T^k\hat
  e_i\}_{i = 1}^{\mu+1}, )\leq Z(T,\{\hat e_j\}_{j=1}^{\mu+1})r^k, \quad \xi \in \mathrm{ran}\chi_{\omega}(T),
\end{equation}
where $\xi$ is a unit eigenvector of $T.$ We may also assume without loss of generality that $\xi$ is orthogonal to $\hat q_j$ for $j=1,...,\mu$. As before, since $\mathrm{span}\{T^k\hat
  e_i\}_{i=1}^{\mu+1} \subset  \mathrm{span}\{T^ke_i\}_{i=1}^{p_{\mu+1}}
$ we have
$$
\delta(
 \mathrm{span}\{\hat q_i\}_{i=1}^{\mu} \oplus \mathrm{span}\{\xi\},
 \mathrm{span}\{T^ke_i\}_{i=1}^{p_{\mu+1}}) \leq Z(T,\{\hat e_j\}_{j=1}^{\mu+1})r^k,
$$
and hence by invertibility of $T$
\begin{equation}\label{n}
\begin{split}
\delta(\mathrm{span}\{\hat q_i\}_{i=1}^{\mu} \oplus \mathrm{span}\{\xi\},\mathrm{span}\{q_{k,i}\}_{i=1}^{p_{\mu+1}})&=\delta(\mathrm{span}\{\hat q_i\}_{i=1}^{\mu} \oplus \mathrm{span}\{\xi\}, \mathrm{span}\{T^ke_i\}_{i=1}^{p_{\mu+1}})\\
&\leq Z(T,\{\hat e_j\}_{j=1}^{\mu+1})r^k.
\end{split}
\end{equation}  
Again, using that$\{q_{k,i}\}_{i=1}^{p_{\mu+1}}$ are orthonormal, there exists some coefficients $\alpha_{k,i}$ with $\sum_{i=1}^{p_{\mu+1}}|\alpha_{k,i}|^2 \leq 1$ such that defining $\tilde \eta_k =  \sum_{i=1}^{p_{\mu+1}}\alpha_{k,i}q_{k,i}$ we have
\begin{equation}
\label{eta_approx}
\left\|\xi-\tilde \eta_k\right\|\leq Z(T,\{\hat e_j\}_{j=1}^{\mu+1})r^k.
\end{equation}
If $e_m \in \Lambda_{\Psi} \cup  \tilde \Lambda_{\omega},$ where $m < {p_{\mu+1}}$ then as shown above we have
$$
\left|\langle q_{k,m},\xi\rangle\right|=\left|\langle \chi_{\omega}(T)q_{k,m},\xi\rangle\right|\leq A(m)Z(T,\{\hat e_j\}_{j=1}^{\mu})r^k\leq A(m)Z(T,\{\hat e_j\}_{j=1}^{\mu+1})r^k.
$$
Taking the inner product of $\xi-\tilde \eta_k$ with $q_{k,m}$ and using (\ref{eta_approx}) together with the orthonormality of the $q_{k,j}$s, it follows that $\left|\alpha_{k,m}\right|\leq \big(A(m)+1\big)Z(T,\{\hat e_j\}_{j=1}^{\mu+1})r^k$. Similarly, if $j\leq \mu$ then for any $c\in\mathbb{C}$
$$
\left|\langle \hat q_{k,j},\xi\rangle\right|\leq\left|\langle c\hat q_{j},\xi\rangle\right|+\left|c\hat q_{j}-\hat q_{k,j}\right|=\left|c\hat q_{j}-\hat q_{k,j}\right|,
$$
since $\xi$ is orthogonal to $\hat q_j$. Minimising over $c$, we can bound this by $B(j)Z(T,\{\hat e_j\}_{j=1}^{\mu}) r^k$. In the same way, it then follows that $|\alpha_{k,p_{j}}|\leq\big(B(j)+1\big)Z(T,\{\hat e_j\}_{j=1}^{\mu+1})r^k$ where $\hat e_j=e_{p_j}$. Together, these imply that
$$
\left\|\xi-\alpha_{k,p_{\mu+1}}\hat q_{k,\mu+1}\right\|\leq \Big[1+\Big\{\sum_{m=1,e_m\in\Lambda_{\Psi}\cup\tilde \Lambda_{\omega}}^{p_{\mu+1}} \big[A(m)+1\big]^2 +\sum_{j=1}^{\mu} \big[B(j)+1\big]^2\Big\}^{\frac{1}{2}}\Big]Z(T,\{\hat e_j\}_{j=1}^{\mu+1})r^{k}.
$$
To finish the inductive step, we define $\hat q_{\mu+1}=\xi$. Recall that $\xi$ is orthogonal to any $\hat q_{l}$ with $l\leq \mu$. Hence it follows that $\{\hat q_i\}_{i=1}^{\mu+1}$ are orthonormal and we can take
$$
B(\mu+1)=1+\Big\{\sum_{m=1,e_m\in\Lambda_{\Psi}\cup\tilde \Lambda_{\omega}}^{p_{\mu+1}} \big[A(m)+1\big]^2 +\sum_{j=1}^{\mu} \big[B(j)+1\big]^2\Big\}^{\frac{1}{2}}
$$
in (b). For the induction step for (c), the fact that $\{\hat q_{k,i}\}_{i=1}^{\mu+1}$ are orthonormal and (\ref{sum_property}) imply we can take
$$
C(\mu+1)=\Big(\sum_{j=1}^{\mu+1}B(j)^2\Big)^{\frac{1}{2}}.
$$

Finally, if $M$ is finite we demonstrate that $\mathrm{span}\{\hat q_j\}_{j = 1}^M = \mathrm{span}\{\xi_j\}_{j=1}^M.$ Since the $\{\hat q_i\}_{i=1}^M$ are orthogonal and are eigenvectors of $\sum_{j=1}^M \lambda_{c_j} \,\xi_j\otimes\bar \xi_j$ it follows that $\mathrm{span}\{\hat q_j\}_{j = 1}^M =
\mathrm{span}\{\xi_j\}_{j=1}^M = \mathrm{ran}\chi_{\omega}(T).$
\end{proof}

\subsection{Main Results}
\label{main_res_sectIII}

Our first result generalises Theorem \ref{finite} to infinite dimensions and relies on Theorem \ref{QR-theorem} (which concerns convergence to eigenvectors).

\begin{theorem}[Convergence theorem for normal operators in infinite dimensions]\label{new_QR} Let $T \in \mathcal{B}(l^2(\mathbb{N}))$ be an invertible normal operator with $\sigma(T) = \omega \cup \Psi$ and $\omega = \{\lambda_i\}_{i=1}^N,$ where the $\lambda_i$'s are isolated eigenvalues with (possibly infinite) multiplicity $m_i$ satisfying $|\lambda_1| > \hdots > |\lambda_N|.$ Suppose further that $\sup\{|\theta|:\theta \in \Psi\} < |\lambda_N|,$ and let $\{e_j\}_{j\in\mathbb{N}}$ be the canonical orthonormal basis. Let $\{Q_n\}_{n\in\mathbb{N}}$ and $\{R_n\}_{n\in\mathbb{N}}$ be $Q$- and $R$-sequences of $T$ with respect to $\{e_j\}_{j\in\mathbb{N}}.$ Let $\{\hat e_j\}_{j=1}^M \subset \{e_j\}_{j\in\mathbb{N}}$, where $M = m_1 + \hdots +m_N$, be the subset described in Definition \ref{define_the_basis} and Theorem \ref{QR-theorem}, i.e. $\mathrm{span}\{Q_k \hat e_j\} \rightarrow\mathrm{span}\{\hat q_j\}$ where $\{\hat q_j\}_{j=1}^M\subset\mathrm{ran}\chi_{\omega}(T)$ is a collection of orthonormal eigenvectors of $T$ and if $e_j \notin \{\hat e_j\}_{j=1}^M,$ then $\chi_{\omega}(T)Q_ke_j \rightarrow 0.$ Then: 
\begin{itemize}
\item[(i)] Every subsequence of 
$\{Q_n^*TQ_n\}_{n\in\mathbb{N}}$ has a convergent subsequence 
$\{Q_{n_k}^*TQ_{n_k}\}_{k\in\mathbb{N}}$ such that 
\begin{equation*}
Q_{n_k}^*TQ_{n_k} \stackrel{\text{WOT}}{\longrightarrow} \left(\bigoplus_{j=1}^{M} \langle T\hat q_j,\hat q_j\rangle\hat e_j \otimes \hat e_j \right) \bigoplus \sum_{j \in \Theta} \xi_j \otimes e_j, 
\end{equation*}
as $k \rightarrow \infty,$ where
$$
\Theta = \{j: e_j \notin \{\hat e_l\}_{l=1}^M\}, \quad \xi_j \in \overline{\mathrm{span}\{e_i\}_{i\in\Theta}}$$
and only $\sum_{j \in \Theta} \xi_j \otimes e_j$ depends on the choice of subsequence. Furthermore, if $T$ has only finitely many non-zero entries in each column then we can replace $WOT$ convergence by $SOT$ convergence.
\item[(ii)] We have the following convergence of sections:
$$
\widehat P_{M} Q_n^*TQ_n\widehat P_{M}  \stackrel{\text{SOT}}{\longrightarrow} \bigoplus_{j=1}^{M} \langle T\hat q_j,\hat q_j\rangle\hat e_j \otimes \hat e_j ,\qquad \text{as }n \rightarrow \infty,
$$
where $\widehat P_M$ denotes the orthogonal projection onto $\overline{\mathrm{span}\{\hat e_j\}_{j=1}^M}$. Furthermore, if we define
$$
\rho=\mathrm{sup}\{\left|z\right|:z\in\Psi\},\quad r=\max\{\left|\lambda_2/\lambda_1\right|,...,\left|\lambda_N/\lambda_{N-1}\right|,\rho/\left|\lambda_N\right|\}
$$
then $r<1$ and for any fixed $x\in \mathrm{span}\{\hat e_j\}_{j=1}^M$ we have the following rate of convergence
\begin{equation}
\label{when_M_finite2}
\left\|\widehat P_{M} Q_n^*TQ_n \widehat P_{M}x-\left(\bigoplus_{j=1}^{M} \langle T\hat q_j,\hat q_j\rangle\hat e_j \otimes \hat e_j \right)x\right\|= O(r^n),\quad \text{as }n\rightarrow\infty.
\end{equation}
\end{itemize}
If $M$ is finite then we can write (after possibly re-ordering)
\begin{equation}
\label{when_M_finite}
\bigoplus_{j=1}^{M} \langle T\hat q_j,\hat q_j\rangle\hat e_j \otimes \hat e_j =\bigoplus_{k=1}^N\left(\lambda_k\bigoplus_{j=1+\sum_{l<k}m_l}^{\sum_{l\leq k}m_l} \hat e_j \otimes \hat e_j \right),
\end{equation}
and in part (ii) we have the rate of convergence
\begin{equation}\label{eq:big_O}
\left\|\widehat P_{M} Q_n^*TQ_n\widehat P_{M}-\bigoplus_{j=1}^{M} \langle T\hat q_j,\hat q_j\rangle\hat e_j \otimes \hat e_j \right\|= O(r^n),\quad \text{as }n\rightarrow\infty.
\end{equation}
If $\{\chi_{\omega}(T)e_l\}_{l=1}^M$ are linearly independent, then we can take $\hat e_j=e_j$.
\end{theorem}

\begin{remark}
What Theorem \ref{new_QR} essentially says is that if we take the $n$-th iteration of the IQR algorithm and truncate to an $m\times m$ matrix (i.e. $P_mQ_n^*TQ_nP_m$) then, as $n$ grows, the eigenvalues of this matrix will converge to the \textit{extremal parts} of the spectrum of $T$. In particular, the theorem suggests that the IQR algorithm can locate the \textit{extremal parts} of the spectrum.
\end{remark}

\begin{proof}[\textbf{Proof of Theorem \ref{new_QR}}:]
To prove (i), since a closed ball in $\mathcal{B}(l^2(\mathbb{N}))$ is weakly sequentially compact, it follows that that any subsequence of $\{Q_n^*TQ_n\}_{n\in\mathbb{N}}$ must have a weakly convergent subsequence $\{Q_{n_k}^*TQ_{n_k}\}_{k\in\mathbb{N}}$. In particular, there exists a $W \in \mathcal{B}(l^2(\mathbb{N}))$ such that
$$
 Q_{n_k}^*TQ_{n_k} \stackrel{\text{\scriptsize WOT}}{\longrightarrow} W, \qquad k
\rightarrow \infty.
$$
Let $\widehat P_M$ denote the projection onto $\overline{\mathrm{span}\{\hat e_j\}_{j=1}^M}$. Note that part (i) of the theorem will follow if we can show that 
\begin{equation}\label{tjall}
\widehat P_M W \widehat P_M = \bigoplus_{j=1}^{M} \langle T\hat q_j,\hat q_j\rangle\hat e_j \otimes \hat e_j ,
\end{equation}
and 
$$
\widehat P_M^{\perp} W \widehat P_M = 0, \quad \widehat P_MW \widehat P_M^{\perp} = 0.
$$
We will indeed show this, and we start by observing that, due to the weak convergence and the standard functional calculus, we have that 
\begin{equation}\label{W}
\langle W \hat e_j, e_i \rangle  = \lim_{k\rightarrow \infty}\langle TQ_{n_k} \hat e_j, \chi_{\omega}(T)Q_{n_k} e_i \rangle + \lim_{k\rightarrow \infty}\langle TQ_{n_k} \hat e_j, \chi_{\Psi}(T) Q_{n_k}e_i \rangle, 
\end{equation}
\begin{equation}\label{W_2}
\langle W e_i,\hat e_j \rangle  = \lim_{k\rightarrow \infty}\langle \chi_{\omega}(T)Q_{n_k} e_i, T^*Q_{n_k} \hat e_j \rangle + \lim_{k\rightarrow \infty}\langle T Q_{n_k} e_i, \chi_{\Psi}(T)Q_{n_k} \hat e_j \rangle.
\end{equation}

We then have the following 
\begin{equation}\label{lim1}
\begin{split}
&\chi_{\omega}(T)Q_ne_i\rightarrow 0,  \quad n \rightarrow \infty, \quad i \in \Theta \\
&\qquad\qquad\Longrightarrow \quad 
\begin{cases}
\lim_{k\rightarrow \infty}\langle TQ_{n_k} \hat e_j, \chi_{\omega}(T)Q_{n_k} e_i \rangle = 0, & i \in \Theta,\\
\lim_{k\rightarrow \infty}\langle \chi_{\omega}(T)Q_{n_k} e_i, T^*Q_{n_k} \hat e_j \rangle  = 0, & i \in \Theta,
\end{cases}
\end{split}
\end{equation}
\begin{equation}\label{ol}
\begin{split}
&\mathrm{span}\{Q_n \hat e_j\} \rightarrow
\mathrm{span}\{\hat q_j\}, \quad n \rightarrow \infty, \quad T\hat q_j = \lambda \hat q_j, \, \lambda \in \omega, \\
&\qquad\qquad\Longrightarrow \quad 
\begin{cases}
\lim_{k\rightarrow \infty}\langle TQ_{n_k} \hat e_j, \chi_{\Psi}(T) Q_{n_k}e_i \rangle = 0, & i \in \mathbb{N},\\
\lim_{k\rightarrow \infty}\langle T Q_{n_k} e_i, \chi_{\Psi}(T)Q_{n_k} \hat e_j \rangle = 0, & i \in \mathbb{N},\\
 \lim_{k\rightarrow \infty}\langle TQ_{n_k} \hat e_j, \chi_{\omega}(T)Q_{n_k} \hat e_l \rangle = \delta_{j,l} \lambda.
\end{cases}
\end{split}
\end{equation}
Thus, by (\ref{W}), (\ref{lim1}), (\ref{ol}) and Theorem \ref{QR-theorem} we get (\ref{tjall}) and also that $\widehat P_M^{\perp}W \widehat P_M = 0$. Also, by (\ref{W_2}), (\ref{lim1}), (\ref{ol}) and Theorem \ref{QR-theorem} we get that $\widehat P_MW \widehat P_M^{\perp} = 0$. Note that in all of these cases, Theorem \ref{QR-theorem} implies that the rate of convergence is such that the difference between $\langle W \hat e_j, e_i \rangle$, $\langle W e_i,\hat e_j \rangle$ and their limiting values is $O(r^{n_k})$ (however, not necessarily uniformly over the indices). Now suppose that $T$ has finitely many non-zero entries in each column. This can be described by a function $f:\mathbb{N}\rightarrow\mathbb{N}$ non-decreasing with $f(n)\geq n$ such that $\langle Te_j,e_i\rangle = 0$ when $i > f(j)$ as in Definition \ref{quasi_f}. Proposition \ref{subdiagonal} shows that this is preserved under the iteration in the IQR algorithm, i.e. $Q_{n_k}^*TQ_{n_k}$ also has this property. So let $x\in l^2(\mathbb{N})$ and $\epsilon>0$. Choose $y$ of finite support such that $\|x-y\|\leq \epsilon$. It is then clear that $\|Q_{n_k}^*TQ_{n_k}y-Wy\|\rightarrow 0$ as $n_k\rightarrow\infty$ (since we only require convergence in finitely many entries). Hence
$$
\limsup_{n_k\rightarrow\infty}\|Q_{n_k}^*TQ_{n_k}x-Wx\|\leq (\|T\|+\|W\|)\epsilon.
$$
Since $\epsilon>0$ and $x$ were arbitrary we have $ Q_{n_k}^*TQ_{n_k} \stackrel{\text{\scriptsize SOT}}{\longrightarrow} W$.

To prove (ii), suppose that $x\in\mathrm{span}\{\hat e_j\}_{j=1}^M$, then $x$ can be written as 
$$
x=\sum_{j=1}^Mx_j\hat e_j,
$$
with at most finitely many $x_j$ non-zero. We have that $\hat\delta(\mathrm{span}\{Q_n\hat e_j\},\mathrm{span}\{\hat q_j\})= O(r^n)$ and hence there exists some $a_{n,j}$ of unit modulus such that $\left\|Q_n\hat e_j-a_{n,j}\hat q_j\right\|= O(r^n)$. Since $Q_n$ is unitary we then have
\begin{align*}
\left\|\widehat P_{M} Q_n^*TQ_n\widehat P_{M}x-\left(\bigoplus_{j=1}^{M} \langle T\hat q_j,\hat q_j\rangle\hat e_j \otimes \hat e_j \right)x\right\|&\leq\left\|Q_n^*TQ_n\widehat P_{M}x-\left(\bigoplus_{j=1}^{M} \langle T\hat q_j,\hat q_j\rangle\hat e_j \otimes \hat e_j \right)Q_n^*Q_nx\right\|\\
&=\left\|\sum_{j=1}^Mx_j(T-\langle T\hat q_j,\hat q_j\rangle I)Q_n\hat e_j\right\|= O(r^n),
\end{align*}
where we have used the fact that $T$ is bounded in the last line. We therefore have convergence on $\mathrm{span}\{\hat e_j\}_{j=1}^M$, and, since the operators are uniformly bounded, we must have convergence on $\overline{\mathrm{span}\{\hat e_j\}_{j=1}^M}$ which implies that
$$
\widehat P_{M} Q_n^*TQ_n\widehat P_{M}  \stackrel{\text{SOT}}{\longrightarrow} \bigoplus_{j=1}^{M} \langle T\hat q_j,\hat q_j\rangle\hat e_j \otimes \hat e_j ,\qquad \text{as }n \rightarrow \infty.
$$

For the last parts, suppose that $M$ is finite. Theorem \ref{QR-theorem} then implies (\ref{when_M_finite}) after a possible re-ordering. The rate of convergence in (\ref{when_M_finite2}) also implies that
$$
\left\|\widehat P_{M} Q_n^*TQ_n\widehat P_{M}-\bigoplus_{j=1}^{M} \langle T\hat q_j,\hat q_j\rangle\hat e_j \otimes \hat e_j \right\|= O(r^n).
$$
More generally, let $K\in\mathbb{N}\cup\{\infty\}$ be minimal such that $\dim(\mathrm{span}\{\chi_{\omega}(T)e_j\}_{j=1}^K) = M.$ Recall that we defined
\begin{align*}
&\Lambda_{\omega} = \{e_j:\chi_{\omega}(T)e_j \neq 0, j\leq K\}, \quad
\Lambda_{\Psi}= \{e_j:\chi_{\omega}(T)e_j = 0, j \leq K\}\\
&\text{and    } \tilde \Lambda_{\omega} = \{e_j \in
\Lambda_{\omega}:\chi_{\omega}(T)e_j \in
\mathrm{span}\{\chi_{\omega}(T)e_i\}_{i=1}^{j-1}\}.
\end{align*}
Recall also from the proof of Theorem \ref{QR-theorem} that $\{\hat e_j\}_{j=1}^{M} = \Lambda_{\omega} \setminus \tilde \Lambda_{\omega}.$ If $\{\chi_{\omega}(T)e_j\}_{j=1}^M$ are linearly independent then $\tilde \Lambda_{\omega} = \emptyset$, and therefore $\{\hat e_j\}_{j=1}^{M} = \{ e_j\}_{j=1}^{M},$ which yields that the projection $\widehat P_M$ in (\ref{tjall}) is the projection onto $\overline{\mathrm{span}\{e_j\}_{j=1}^M}$.
\end{proof}

Theorems \ref{new_QR} and \ref{QR-theorem} also give us convergence to the eigenvectors. With the use of (possibly countably many) shifts and rotations, the above theorem allows us to find all eigenvalues, their multiplicities and eigenspaces outside the convex hull of the essential spectrum, i.e. outside the essential numerical range.

\begin{example}
\label{lose_spec_exam}
It is possible in the case of infinite $M$ that the $\hat q_j$ do not form an orthonormal basis of $\mathrm{ran}\chi_{\omega}(T)$ and we can even loose part of $\omega$ in the convergence of $\widehat P_{M} Q_n^*TQ_n \widehat P_{M}$ to a diagonal operator. This is to be contrasted to the finite dimensional case. For example, suppose that with respect to an initial orthonormal basis $\{v_j\}_{j\in\mathbb{N}}$, $T$ is given by the diagonal matrix $\mathrm{Diag}(1/2,1,1,...)$. Now define $f_j=v_1+(1/j)v_{j+1}$ and apply Gram-Schmidt to the sequence $\{f_j\}_{j\in\mathbb{N}}$ to generate orthonormal vectors $\{e_j\}_{j\in\mathbb{N}}$. It is easy to see that any $v_j$ can be approximated to arbitrary accuracy using finite linear combinations of $e_j$ and hence $\{e_j\}_{j\in\mathbb{N}}$ is an orthonormal basis of our Hilbert space. We also have that the $\chi_{1}(T)(f_j)=(1/j)v_{j+1}$ are linearly independent and hence so are $\chi_{1}(T)(e_j)$. It follows that the IQR iterates converge in the strong operator topology to the identity operator. However, we could equally take $\omega=\{1,1/2\}$ in Theorem \ref{new_QR}. Hence we have the curious case that $\overline{\mathrm{span}\{\hat q_j\}_{j\in\mathbb{N}}}\subset\overline{\mathrm{span}\{\hat v_j\}_{j>1}}$ and we loose the eigenvalue $1/2$.
\end{example}

The following corollary is entirely analogous to the finite dimensional case.

\begin{corollary}
\label{new_QRCOR}
Suppose that the conditions of Theorem \ref{new_QR} hold with $M$ finite. Suppose also that for $j=1,...,N$ the vectors $\{\chi_{\{\lambda_1,...,\lambda_j\}}(T)e_i\}_{i=1}^{\sum_{l\leq j}m_l}$ are linearly independent. In the notation of Theorem \ref{new_QR}, let $\rho=\mathrm{sup}\{\left|z\right|:z\in\Psi\}$. For $j<N$ define $r_j=\max\{|\lambda_{k+1}/\lambda_k|:k\leq j\}$ and for $j=N$ define $r_N=\max\{|\lambda_{k+1}/\lambda_k|,|\lambda_N/\rho|:k\leq j\}$. We then have the following rates of convergence to the diagonal operator for $i,j\leq M$:
\begin{enumerate}
	\item $\left|\langle Q_n^*TQ_n e_j, e_i\rangle\right|= O(r_k^n)$ as $n\rightarrow\infty$ if $i>j$ and $k$ is minimal such that $i\leq \sum_{l\leq k}m_l$,
	\item $\left|\langle Q_n^*TQ_n e_i, e_i\rangle-\lambda_k\right|= O(r_k^n)$ as $n\rightarrow\infty$ if $k$ is minimal such that $i\leq \sum_{l\leq k}m_l$.
\end{enumerate}
\end{corollary}

\begin{proof}
The result follows from Theorem \ref{new_QR} applied successively to $\omega_1,\omega_2,...,\omega_N$ where $\omega_j=\{\lambda_k:k\leq j\}$. In general analogous results follows from Theorem \ref{new_QR} when $M$ is infinite and with other linear independence conditions on $\chi_{\omega'}(T)e_i$ with $\omega'\subset\omega$ but the statements become less succinct.
\end{proof}

In the finite dimensional case and the case of distinct eigenvalues of the \textit{same magnitude} the QR algorithm applied to a normal matrix will `converge' to a block diagonal matrix (without necessarily converging in each block). This can be extended to infinite dimensions by inductively using the following theorem which also extends to \textit{non-normal} operators.

\begin{theorem}[Block convergence theorem in infinite dimensions]
\label{non_normal_prop}
Let $T \in \mathcal{B}(l^2(\mathbb{N}))$ be an invertible operator (not necessarily normal) and suppose that there exists an orthogonal projection $P$ of rank $M$ (possibly infinite) such that both the ranges of $P$ and of $I-P$ are invariant under $T$. Suppose also that there exists $\alpha>\beta>0$ such that
\begin{itemize}
	\item $\|Tx\|\geq \alpha \|x\|\quad \forall x\in \mathrm{ran}(P)$,
	\item $\|Tx\|\leq \beta \|x\| \quad \forall x\in \mathrm{ran}(I-P)$.
\end{itemize}
Let $\{Q_n\}_{n\in\mathbb{N}}$ and $\{R_n\}_{n\in\mathbb{N}}$ be $Q$- and $R$-sequences of $T$ with respect to $\{e_i\}.$ Then there exists a subset $\{\hat e_j\}_{j=1}^M\subset \{e_i\}_{i\in\mathbb{N}}$ such that
\begin{itemize}
	\item[(i)] For any finite $\mu\leq M$ we have $\delta (\mathrm{span}\{Q_n\hat e_j\}_{j=1}^\mu,\mathrm{ran}(P))= O(\beta^n/\alpha^n)$ as $n\rightarrow\infty$. If $M$ is finite this implies full convergence $\hat\delta (\mathrm{span}\{Q_n\hat e_j\}_{j=1}^M,\mathrm{ran}(P))= O(\beta^n/\alpha^n)$ as $n\rightarrow\infty$.
	\item[(ii)] Every subsequence of $\{Q_n^*TQ_n\}_{n\in\mathbb{N}}$ has a convergent subsequence 
$\{Q_{n_k}^*TQ_{n_k}\}_{k\in\mathbb{N}}$ such that 
\begin{equation*}
Q_{n_k}^*TQ_{n_k} \stackrel{\text{WOT}}{\longrightarrow} \sum_{j=1}^M \xi_j \otimes \hat e_j \bigoplus \sum_{i \in \Theta} \zeta_i \otimes e_i, 
\end{equation*}
as $k \rightarrow \infty,$ where
$$
\Theta = \{j: e_j \notin \{\hat e_l\}_{l=1}^M\}, \quad \xi_j \in \overline{\mathrm{span}\{\hat e_l\}_{l=1}^M}, \quad \zeta_i \in \overline{\mathrm{span}\{e_l\}_{l\in\Theta}}.$$
\end{itemize}
If $\{Pe_l\}_{l=1}^M$ are linearly independent then we can take $\hat e_j=e_j$. Furthermore, if $T$ has only finitely many non-zero entries in each column then we can replace $WOT$ convergence by $SOT$ convergence.
\end{theorem}

\begin{remark}
Theorem \ref{non_normal_prop} essentially says that the IQR algorithm can compute the invariant subspace $\mathrm{ran}(P)$ of such an operator if there is enough separation between $T$ restricted to $\mathrm{ran}(P)$ and $\mathrm{ran}(I-P)$. In other words, provided the existence of a \textit{dominant} invariant subspace.
\end{remark}

\begin{proof}[\textbf{Proof of Theorem \ref{non_normal_prop}}:]
The main ideas of the proof of Theorem \ref{non_normal_prop} have already been presented so we sketch the proof. We first define the vectors $\{\hat e_j\}_{j=1}^M$ in a similar way to Definition \ref{define_the_basis} inductively by $\hat e_j=e_{p_j}$ where
$$
p_j=\min\{i:Pe_i\notin\mathrm{span}\{P\hat e_{k}\}_{k=1}^{j-1}\}.
$$
Let $r=\beta/\alpha<1$. We will prove inductively that
\begin{itemize}
	\item[(a)] $\hat\delta(\mathrm{span}\{Q_n\hat e_j\}_{j=1}^\mu,\mathrm{span}\{PQ_n\hat e_j\}_{j=1}^\mu)\leq C_1(\mu)r^n$ for any finite $\mu\leq M$,
	\item[(b)] $\|PQ_ne_j\|\leq C_2(j)r^n$ for any $j\in\Theta$,
\end{itemize}
for some constants $C_1(\mu)$ and $C_2(j)$. Suppose that this has been done. Part (i) of Theorem \ref{non_normal_prop} now follows since $\mathrm{span}\{PQ_n\hat e_j\}_{j=1}^\mu\subset\mathrm{ran}(P)$. We then argue as in the proof of Theorem \ref{new_QR} to gain
$$
 Q_{n_k}^*TQ_{n_k} \stackrel{\text{\scriptsize WOT}}{\longrightarrow} W, \qquad k
\rightarrow \infty.
$$
Then by studying the inner products $\langle TQ_{n_k}e_j,Q_{n_k}e_i\rangle$ using the invariance of $\mathrm{ran}(P)$, $\mathrm{ran}(I-P)$ under $T$ and from (b), part (ii) of Theorem \ref{non_normal_prop} easily follows (note that (a) implies that $\|(I-P)Q_n\hat e_j\|\leq C_1(j)r^n$). The final part of the theorem then follows from the same arguments in the proof of Theorem \ref{new_QR}. Hence we only need to prove (a) and (b).

We first claim that
\begin{equation}
\label{non_norm_delta}
\delta(\mathrm{span}\{PT^n\hat e_j\}_{j=1}^\mu,\mathrm{span}\{T^n\hat e_j\}_{j=1}^\mu)\leq C_3(\mu)r^n.
\end{equation}
$P$ commutes with $T$ which is invertible and hence both of these spaces have dimension $\mu$ by the construction of the $\hat e_j$. It follows that (\ref{non_norm_delta}) implies
\begin{equation}
\label{non_norm_delta2}
\hat\delta(\mathrm{span}\{PT^n\hat e_j\}_{j=1}^\mu,\mathrm{span}\{T^n\hat e_j\}_{j=1}^\mu)\leq \mu^{\frac{1}{2}}C_3(\mu)r^n=C_4(\mu)r^n.
\end{equation}
To show (\ref{non_norm_delta}), let $x_1^n,...,x_\mu^n$ be an orthonormal basis for $\mathrm{span}\{PT^n\hat e_j\}_{j=1}^\mu$ and let $\xi=\sum_{j=1}^\mu \alpha_j x_j^n$ have norm at most $1$. Now, we may choose coefficients $\beta_{j,n}$ such that $
T^n\sum_{j=1}^\mu \beta_{j,n} x_j^n = \xi$ since $T|_{\mathrm{ran}(P)}$ is invertible when viewed as an operator acting on $\mathrm{ran}(P)$. By the assumptions on $T$ we must have that
$$
\big(\sum_{j=1}^m\left|\beta_{j,n}\right|^2\big)^{1/2} \leq \frac{1}{\alpha^n}.
$$
We may change basis from $\{\hat e_j\}_{j=1}^\mu$ to $\{\tilde e_j\}_{j=1}^\mu$ such that $P\tilde e_j=x_j^n$. Form the vector
$$
\eta_n= T^n\big(\sum_{j=1}^\mu\beta_{j,n} \tilde e_j\big)\in \mathrm{span}\{T^n\hat e_j\}_{j=1}^\mu.
$$
Then clearly by H\"older's inequality
$$
\|\xi-\eta_n\|\leq \frac{\big(\sum_{j=1}^\mu\|T^n(I-P)\tilde e_j\|^2\big)^{1/2}}{\alpha^n}\leq C_3(\mu)\frac{\beta^n}{\alpha^n},
$$
proving (\ref{non_norm_delta}) and hence (\ref{non_norm_delta2}).

Note that the proof of Lemma \ref{lemma_span} carries over (replacing the projection $\chi_{\omega}(T)$ by $P$) to prove that
\begin{equation}
\label{ext_lemma}
\mathrm{span}\{PQ_ne_j\}_{j=1}^m=\mathrm{span}\{PQ_n\hat e_j\}_{j=1}^{s(m)}
\end{equation}
where $s(m)$ is maximal with $\{\hat e_j\}_{j=1}^{s(m)}\subset\{e_j\}_{j=1}^m$. It follows that
\begin{align*}
\delta(\mathrm{span}\{T^n\hat e_j\}_{j=1}^\mu,\mathrm{span}\{PQ_n\hat e_j\}_{j=1}^\mu)&=\delta(\mathrm{span}\{T^n\hat e_j\}_{j=1}^\mu,\mathrm{span}\{PQ_n e_j\}_{j=1}^{p_\mu})\\
&=\delta(\mathrm{span}\{T^n\hat e_j\}_{j=1}^\mu,\mathrm{span}\{PT^n e_j\}_{j=1}^{p_\mu})\\
&\leq\delta(\mathrm{span}\{T^n\hat e_j\}_{j=1}^\mu,\mathrm{span}\{PT^n \hat e_j\}_{j=1}^{\mu})\leq C_4(\mu)r^n,
\end{align*}
where we have used (\ref{usefulspan}) to reach the second line and the fact that $\mathrm{span}\{PT^n \hat e_j\}_{j=1}^{\mu}\subset\mathrm{span}\{PT^n e_j\}_{j=1}^{p_\mu}$ to reach the third line. Again, both spaces have dimension $\mu$ so we have
\begin{equation}
\label{what_we_wnt}
\begin{split}
\delta(\mathrm{span}\{PQ_n\hat e_j\}_{j=1}^\mu,\mathrm{span}\{Q_n e_j\}_{j=1}^{p_\mu})&=\delta(\mathrm{span}\{PQ_n\hat e_j\}_{j=1}^\mu,\mathrm{span}\{T^n e_j\}_{j=1}^{p_\mu})\\
&\leq \delta(\mathrm{span}\{PQ_n\hat e_j\}_{j=1}^\mu,\mathrm{span}\{T^n \hat e_j\}_{j=1}^{\mu})\leq C_5(\mu)r^n.
\end{split}
\end{equation}

With these arguments out of the way (these are the analogue of Proposition \ref{rate_extension}) we can now form our inductive argument, similar to the proof of Theorem \ref{QR-theorem}. Suppose first that (a) holds for $\mu$ (allowing $\mu=0$ for the initial step) and let $j\in\Theta$ have $j<p_{\mu+1}$ (where $p_{\mu+1}=\infty$ if $\mu=M$). From (a) for $\mu$ and (\ref{ext_lemma}) we have that
$$
PQ_ne_j=v_{n}+\sum_{i=1}^\mu a_{n,i}Q_n\hat e_i
$$
for some $v_{n}$ with $\|v_{n}\|\leq C_1(\mu)r^n$. Then we must have
$$
a_{n,i}+\langle v_{n}, Q_n\hat e_i\rangle = \langle PQ_n e_j,Q_n \hat e_i\rangle=\langle Q_n e_j,PQ_n \hat e_i\rangle.
$$
Using (a) again, along with the fact that $Q_n e_j$ is orthogonal to $\{Q_n\hat e_i\}_{i=1}^\mu$, we must have $\left|a_{n,i}\right|\leq 2C_1(\mu)r^n$. It follows that we can take $C_2(j)=(2\sqrt{\mu}+1)C_1(\mu)$ for $j\in[p_{\mu}+1,...,p_{\mu+1})$ in (b). Now we use (\ref{what_we_wnt}). Let $\xi\in\mathrm{span}\{PQ_n\hat e_j\}_{j=1}^{\mu+1}$ have unit norm and assume that $p_{\mu+1}<\infty$ (else there is nothing to prove since then $\mu=M$). Then there exists $b_{n,j}$ and $w_n$ such that
$$
\xi = \sum_{j=1}^{p_{\mu+1}} b_{n,j}Q_n e_j +w_n
$$
and $\|w_n\|\leq C_5(\mu+1)r^n$. Now let $j\in\Theta$ with $j<p_{\mu+1}$ then we must have
$$
\langle \xi, PQ_ne_j\rangle=\langle \xi, Q_ne_j\rangle=b_{n,j}+\langle w_n,Q_ne_j\rangle.
$$
We have proven (b) for such $j$ and hence we have $\left|b_{n,j}\right|\leq \big(C_2(j)+C_5(\mu+1)\big)r^n$. It follows that we can take
$$
C_1(\mu+1)=\mu^{\frac{1}{2}}\Big[C_5(\mu+1)+\Big\{\sum_{j=1,j\in\Theta}^{p_{\mu+1}}[C_2(j)+C_5(\mu+1)]^2\Big\}^{\frac{1}{2}}\Big],
$$
where the square root factor appears since the relevant spaces are $\mu$-dimensional. This completes the inductive step (the initial step is identical) and hence the proof of the theorem.
\end{proof}

Theorem \ref{non_normal_prop} can be made sharper (under a slightly stricter assumption on the linear independence of $\{e_j\}_{j=1}^M$) with the following theorem which includes the case that $\mathrm{ran}(I-P)$ is not necessarily invariant.

\begin{theorem}[Convergence to invariant subspace in infinite dimensions]
\label{non_normal_prop2}
Let $T \in \mathcal{B}(l^2(\mathbb{N}))$ be an invertible operator (not necessarily normal) and suppose that there exists an orthogonal projection $P$ of finite rank $M$ such that the range of $P$ is invariant under $T$. Suppose also that there exists $\alpha>\beta>0$ such that
\begin{itemize}
	\item $\|Tx\|\geq \alpha\|x\| \quad \forall x\in \mathrm{ran}(P)$,
	\item $\|(I-P)T(I-P)\|\leq \beta$.
\end{itemize}
Under these conditions, there exists a canonical $M$ dimensional $T^*-$invariant subspace $S$ and we let $\tilde P$ denote the orthogonal projection onto $S$ (in the special case that $\mathrm{ran}(I-P)$ is also $T$-invariant such as in Theorems \ref{new_QR} and \ref{non_normal_prop}, then $S=\mathrm{ran}(P)$). Suppose also that $\{\tilde Pe_j\}_{j=1}^M$ are linearly independent. Let $\{Q_n\}_{n\in\mathbb{N}}$ and $\{R_n\}_{n\in\mathbb{N}}$ be $Q$- and $R$-sequences of $T$ with respect to $\{e_i\}.$ Then
\begin{itemize}
	\item[(i)] The subspace angle $\phi(\mathrm{span}\{e_j\}_{j=1}^M,S)<\pi/2$ and we have
	\begin{equation}
	\label{beautiful}
	\hat\delta (\mathrm{span}\{Q_n e_j\}_{j=1}^M,\mathrm{ran}(P))\leq \frac{\sin\big(\phi(\mathrm{span}\{e_j\}_{j=1}^M,\mathrm{ran}(P))\big)}{\cos\big(\phi(\mathrm{span}\{e_j\}_{j=1}^M,S)\big)}\frac{\beta^n}{\alpha^n}\Big(1+\frac{\|PT(I-P)\|}{\alpha-\beta}\Big),
	\end{equation}
	\item[(ii)] Every subsequence of $\{Q_n^*TQ_n\}_{n\in\mathbb{N}}$ has a convergent subsequence 
$\{Q_{n_k}^*TQ_{n_k}\}_{k\in\mathbb{N}}$ such that 
\begin{equation*}
Q_{n_k}^*TQ_{n_k} \stackrel{\text{WOT}}{\longrightarrow} \sum_{j=1}^M \xi_j \otimes e_j \bigoplus \sum_{i=M+1}^{\infty} \zeta_i \otimes e_i, 
\end{equation*}
as $k \rightarrow \infty,$ where
$$
\xi_j \in \overline{\mathrm{span}\{e_l\}_{l=1}^M}, \quad \zeta_i\in \mathcal{H}.$$
\end{itemize}
Furthermore, if $T$ has only finitely many non-zero entries in each column then we can replace $WOT$ convergence by $SOT$ convergence.
\end{theorem}

\begin{remark}
Theorem \ref{non_normal_prop2} says that the IQR algorithm can be used to approximate dominant invariant subspaces. In particular, we shall use the bound (\ref{beautiful}) to build a $\Delta_1$ algorithm in Section \ref{class_thms}. Note in the normal case that Theorem \ref{new_QR} is more precise, both in giving convergence of individual vectors to eigenvectors and in the less restrictive assumptions on spanning sets and $M$. In the normal case (and that of Theorem \ref{non_normal_prop}) we also have that the limit operator has a block diagonal form.
\end{remark}

\subsection{Proof of Theorem \ref{non_normal_prop2}}

In this section we will prove Theorem \ref{non_normal_prop2}. The proof technique is different to those used above and hence we have given it a separate section. Throughout, we will denote the ratio $\beta/\alpha$ by $r$. Note that since $M$ is finite, the bound $\alpha$ implies that $T|_{\mathrm{ran}(P)}:\mathrm{ran}(P)\rightarrow\mathrm{ran}(P)$ is invertible with $\|T|_{\mathrm{ran}(P)}^{-1}\|\leq 1/\alpha$. First, let $Q$ denote a unitary change of basis matrix from $\{e_j\}$ to $\{\tilde e_j\}$ where $\{\tilde e_j\}_{j=1}^M$ is a basis for $\mathrm{ran}(P)$. Then as matrices with respect to the original basis we can write
$$
Q=[P_1,P_2],\quad Q^* TQ=
\left(\begin{matrix}
T_{11}      & T_{12} \\
0  & T_{22}
\end{matrix}\right),
$$ 
where $T_{11}\in\mathbb{C}^{M\times M}$ and $T_{12}$ has $M$ rows. Our assumptions imply that $\|T_{11}^{-1}\|\leq 1/\alpha$ and $\|T_{22}\|\leq \beta$. The next lemma shows that we can change the basis further to eliminate the sub-block $T_{12}$. This is needed to apply a power iteration type argument.

\begin{lemma}%[Infinite Dimensional Block Diagonal Decomposition]
\label{schur}
Define the linear function $F:\mathcal{B}(l^2(\mathbb{N}),\mathbb{C}^M)\rightarrow\mathcal{B}(l^2(\mathbb{N}),\mathbb{C}^M)$ by
$$
F(A)=T_{11}^{-1}AT_{22},
$$
where we identify elements of $\mathcal{B}(l^2(\mathbb{N}),\mathbb{C}^M)$ as matrices. Then we can define $A\in\mathcal{B}(l^2(\mathbb{N}),\mathbb{C}^M)$ by
$A-F(A)=-T_{11}^{-1}T_{12}$. Furthermore, if we define
$$
B(A)=\left(\begin{matrix}
I      & A \\
0  & I
\end{matrix}\right),
$$
then $B(A)$ has inverse $B(-A)$ and
\begin{equation}
\label{schur_decomp}
B(-A)\left(\begin{matrix}
T_{11}      & T_{12} \\
0  & T_{22}
\end{matrix}\right)B(A)=\left(\begin{matrix}
T_{11}      & 0 \\
0  & T_{22}
\end{matrix}\right).
\end{equation}
\end{lemma}
\begin{proof}
Our assumptions on $T$ ensure that $F$ is a contraction with $\left\|F\right\|\leq r<1$. Hence we can define $A$ via the series
$$
A=\sum_{k=0}^{\infty}F^{k}(-T_{11}^{-1}T_{12}).
$$
It is then straightforward to check $A-F(A)=-T_{11}^{-1}T_{12}$, $B(A)B(-A)=B(-A)B(A)=I$ and the identity (\ref{schur_decomp}).
\end{proof}

Let
$$
Y=Q\left(\begin{matrix}
I      & 0 \\
-A^*  & I
\end{matrix}\right)
$$
then we have the matrix identity
$$
Y^{-1}T^*Y=\left(\begin{matrix}
T_{11}^*      & 0 \\
0  & T_{22}^*
\end{matrix}\right).
$$
The canonical $T^*-$invariant subspace alluded to in Theorem \ref{non_normal_prop2} is then simply $S=\mathrm{span}\{Ye_j\}_{j=1}^M$. The space is canonical since it is easily seen that it is unchanged if we use a different basis for $\mathrm{ran}(P_1)$ and $\mathrm{ran}(P_2)$ in the definition of $Q$.

Now let $P_0=\left(
\begin{matrix}
 e_1 & e_2 & \hdots & e_M
\end{matrix}
\right)\in\mathcal{B}(\mathbb{C}^M,l^2(\mathbb{N}))$ denote the matrix who's columns are the first $M$ basis elements $\{e_j\}_{j=1}^M$. Since the $\{R_i\}$ are upper triangular, it is easy to see that
$$
T^nP_0=Q_nR_nP_0=Q_nP_0P_0^*R_nP_0.
$$
We will denote the (invertible) matrix $P_0^*R_nP_0\in\mathbb{C}^{M\times M}$ by $Z_n$. Now define
$$
V^1_n=P_1^*Q_nP_0\in\mathcal{B}(\mathbb{C}^M), \quad V^2_n=P_2^*Q_nP_0\in\mathcal{B}(\mathbb{C}^M,l^2(\mathbb{N})),
$$
then we have the relation
$$
\left(\begin{matrix}
T_{11}      & T_{12} \\
0  & T_{22}
\end{matrix}\right)^n\left(\begin{matrix}
V^1_0       \\
V^2_0  
\end{matrix}\right)=\left(\begin{matrix}
V^1_n       \\
V^2_n  
\end{matrix}\right)Z_n.
$$
But by Lemma \ref{schur} we have
$$
\left(\begin{matrix}
T_{11}      & T_{12} \\
0  & T_{22}
\end{matrix}\right)^n=B(A)\left(\begin{matrix}
T_{11}^n      & 0 \\
0  & T_{22}^n
\end{matrix}\right)B(-A).
$$
Unwinding the definitions, this implies the matrix identities
\begin{align}
T_{11}^n(V_0^1-AV_0^2)&=(V_n^1-AV_n^2)Z_n,\label{mat1}\\
T_{22}^nV_0^2&=V_n^2Z_n\label{mat2}.
\end{align}

\begin{lemma}
\label{tidy1}
The following identity holds
\begin{equation}
\label{id1}
\hat\delta (\mathrm{span}\{Q_n e_j\}_{j=1}^M,\mathrm{ran}(P))=\|V_n^2\|.
\end{equation}
\end{lemma}
\begin{proof}
Note that $\mathrm{span}\{Q_n e_j\}_{j=1}^M=\mathrm{ran}(Q_nP_0)$ and $\mathrm{ran}(P)=\mathrm{ran}(P_1)$. Since $P_1P_1^*$ and $Q_nP_0P_0^*Q_n^*$ are orthogonal projections, it follows that
\begin{align*}
\hat\delta (\mathrm{span}\{Q_n e_j\}_{j=1}^M,\mathrm{ran}(P))&=\| Q_nP_0P_0^*Q_n^*-P_1P_1^*\|\\
&=\| Q_n^*(Q_nP_0P_0^*Q_n^*-P_1P_1^*)Q\|\\
&=\left\|\left(\begin{matrix}
0      & P_0^*Q_n^*P_2 \\
-(I-P_0)^*Q_n^*P_1  & 0
\end{matrix}\right)\right\|.
\end{align*}
But we have that $\|P_0^*Q_n^*P_2\|=\|V_n^2\|$ and hence we are done if we can show $\|P_0^*Q_n^*P_2\|=\|(I-P_0)^*Q_n^*P_1\|$. Consider the unitary matrix
$$
U:=Q_n^*Q=\left(\begin{matrix}
P_0^*Q_n^*P_1      & P_0^*Q_n^*P_2 \\
(I-P_0)^*Q_n^*P_1  & (I-P_0)^*Q_n^*P_2
\end{matrix}\right)=\left(\begin{matrix}
U_{11}      & U_{12} \\
U_{21}  & U_{22}
\end{matrix}\right).
$$
Now let $x\in\mathbb{C}^{M}$ be of unit norm, then $\|U_{11}x\|^2+\|U_{21}x\|^2=1$. It follows that
$\|U_{21}\|^2=1-\sigma_0(U_{11})^2$, where $\sigma_0$ denotes the smallest singular value. Applying the same argument to $U^*$ we see that $\|U_{12}\|^2=1-\sigma_0(U_{11})^2=\|U_{21}\|^2$, completing the proof.
\end{proof}

\begin{lemma}
\label{tidy2}
The matrix $(V_0^1-AV_0^2)$ is invertible with
\begin{equation}
\label{id2}
\|(V_0^1-AV_0^2)^{-1}\|\leq\frac{1}{\cos\big(\phi(\mathrm{span}\{e_j\}_{j=1}^M,S)\big)}.
\end{equation}
\end{lemma}
\begin{proof}
First note that since $\{\tilde Pe_j\}_{j=1}^M$ are linearly independent, we must have $\phi(\mathrm{span}\{e_j\}_{j=1}^M,S)<\pi/2$ and hence the bound in (\ref{id2}) is finite. Let $W=(P_1-P_2A^*)(I+AA^*)^{-1/2}\in\mathcal{B}(\mathbb{C}^M,l^2(\mathbb{N}))$. By considering $W^*W=I\in\mathbb{C}^{M\times M}$, we see that the columns of $W$ are orthonormal. In fact, expanding $Y$ we have
$$
Y=[P_1-P_2A^* \quad P_2]
$$
and hence the columns of $W$ are a basis for the subspace $S$. Arguing as in the proof of Lemma \ref{tidy1}, we have that
$$
\hat\delta(\mathrm{span}\{e_j\}_{j=1}^M,S)=\sqrt{1-\sigma_0(W^*P_0)^2}<1.
$$
This implies that $W^*P_0$ is invertible with
$$
\sigma_0(W^*P_0)=\cos\big(\phi(\mathrm{span}\{e_j\}_{j=1}^M,S\big)\big)>0.
$$
We also have the identity
$$
V_0^1-AV_0^2=(I+AA^*)^{1/2}(W^*P_0).
$$
Since $(I+AA^*)^{-1/2}$ has norm at most $1$, we see that $(V_0^1-AV_0^2)$ is invertible and (\ref{id2}) holds.
\end{proof}

\begin{proof}[\textbf{Proof of Theorem \ref{non_normal_prop2}}:]
Using Lemma \ref{tidy2} and the matrix identities (\ref{mat1}) and (\ref{mat2}), we can write
$$
V_n^2=T_{22}^nV_0^2(V_0^1-AV_0^2)^{-1}T_{11}^{-n}(V_n^1-AV_n^2).
$$
Using (\ref{id1}) and (\ref{id2}), this implies
\begin{equation}
\label{bd1}
\begin{split}
\hat\delta (\mathrm{span}\{Q_n e_j\}_{j=1}^M,\mathrm{ran}(P))&\leq \frac{\|V_0^2\|\|V_n^1-AV_n^2\|r^n}{\cos\big(\phi(\mathrm{span}\{e_j\}_{j=1}^M,S)\big)}\\
&=\frac{\sin\big(\phi(\mathrm{span}\{e_j\}_{j=1}^M,\mathrm{ran}(P))\big)\|V_n^1-AV_n^2\|r^n}{\cos\big(\phi(\mathrm{span}\{e_j\}_{j=1}^M,S)\big)}.
\end{split}
\end{equation}
It is clear by summing a geometric series that
$$
\|A\|\leq \frac{\|T_{12}\|}{\alpha(1-r)}=\frac{\|PT(I-P)\|}{\alpha-\beta}.
$$
It follows that $\|V_n^1-AV_n^2\|\leq 1+\|PT(I-P)\|/(\alpha-\beta)$. Substituting this into (\ref{bd1}) proves part (i) of the theorem.

Next we argue that if $i>M$ then $\|PQ_ne_i\|\rightarrow 0$ as $n\rightarrow\infty$. We have that
$$
PQ_ne_i=\sum_{j=1}^M\alpha_{j,n}Q_ne_j+v_n
$$
with $\|v_n\|\rightarrow 0$ by part (i). Note that we then have
$$
\alpha_{j,n}=\langle PQ_ne_i, Q_ne_j \rangle +\epsilon_{j,n}=\langle Q_ne_i, PQ_ne_j \rangle +\epsilon_{j,n}
$$
with $\{\epsilon_{j,n}\}$ null. But again by (i) we have that $PQ_ne_j$ approaches $\mathrm{span}\{Q_ne_k\}_{k=1}^M$ which is orthogonal to $Q_ne_i$ and hence $\{\alpha_{j,n}\}$ is null. The proof of part (ii) now follows the same argument as in the proof of part (i) of Theorem \ref{new_QR} and of the final part of Theorem \ref{non_normal_prop}. The key property being that if $j\leq M$ and $i>M$ then $\langle Q_n^*TQ_ne_j,e_i\rangle\rightarrow 0$ due to the invariance of $\mathrm{ran}(P)$ under $T$. Note that it does not necessarily follow (as is easily seen by considering upper triangular $T$) that $\langle Q_n^*TQ_ne_i,e_j\rangle\rightarrow 0$ for such $i,j$.
\end{proof}

\section{The IQR algorithm can be computed}\label{implement}

The previous section gives a theoretical justification for why the IQR algorithm may work, but we are faced with the possibly unpleasant problem of how to compute with infinite data structures on a computer. Fortunately there is a way to overcome such a problem. The key is to impose some structural requirements on the infinite matrix.

\subsection{Quasi-banded subdiagonals}
\label{Quasi_banded_Subdiagonals}

\begin{definition}
\label{quasi_f}
Let $T$ be an infinite matrix acting as a bounded operator on $l^2(\mathbb{N})$
with basis $\{e_j\}_{j\in\mathbb{N}}$. For $f:\mathbb{N}\rightarrow\mathbb{N}$ non-decreasing with $f(n)\geq n$ we say that $T$ has quasi-banded subdiagonals with respect to $f$ if $\langle Te_j,e_i\rangle = 0$
when $i > f(j).$ 
\end{definition} 

This is the class of infinite matrices with a finite number of non-zero elements in each column (and not necessarily in each row) which is captured by the function $f$. It is for this class that the computation of the IQR algorithm is feasible on a finite machine. For this class of operators one can actually compute (without any approximation or any extra discretisation) the matrix elements of the $n$-th iteration of the IQR algorithm as if it was done on an infinite computer (meaning the computation collapses to a finite one). The following result of independent interest is needed in the proof and generalises the well known fact in finite dimensions that the QR algorithm preserves bandwidth (see \cite{parlett1998symmetric} for a good discussion of the tridiagonal case).

\begin{proposition}\label{subdiagonal}
Let $T \in \mathcal{B}(l^2(\mathbb{N}))$ and let $T_n$ be the $n$-th element in the IQR iteration, such that
$
T_n = Q^*_n \cdots Q^*_1 T Q_1 \cdots Q_n,
$
where 
$$
Q_j = \underset{l \rightarrow \infty}{\text{SOT-lim \,}} U^j_1\cdots U^j_l
$$ and $U^j_l$ is a Householder transformation. 
If T has quasi-banded subdiagonals with respect to $f$ then so does $T_n$.
\end{proposition}

\begin{proof}
By induction, it is enough to prove the result for $n=1$. From the construction of the Householder reflections $U_m^1=P_{m-1}\oplus S_m$, the chosen $\eta_m$ (see Theorem \ref{QRthrm}) have
\begin{equation}
\langle \eta_m, e_j\rangle = 0, \quad j>f(m).
\end{equation}
Using the fact that $f$ is increasing, it follows that each $U^1_m$ has quasi-banded subdiagonals with respect to $f$, as does the product $U^1_1\cdots U^1_m$. It follows that $Q_1$ must have quasi-banded subdiagonals with respect to $f$ and hence so does $T_1=R_1Q_1$ since $R_1$ is upper triangular.
\end{proof}

\begin{theorem}\label{QR-stuff}
Let $T \in \mathcal{B}(l^2(\mathbb{N}))$ have quasi-banded subdiagonals with respect to $f$ and let 
 $T_n$ be the $n$-th element in the IQR iteration, i.e.
$
T_n = Q^*_n \cdots Q^*_1 T Q_1 \cdots Q_n,
$
where 
$$
Q_j = \underset{l \rightarrow \infty}{\text{SOT-lim \,}} U^j_1\cdots U^j_l
$$ and $U^j_l$ is a Householder transformation (the superscript is not a power, but an index). 
Let $P_m$ be the usual projection onto $\mathrm{span}\{e_j\}_{j=1}^m$ and denote the $a$-fold iteration of $f$ by $\underbrace{f\circ{}f\circ{}...\circ{}f}_{a\text{ times}}=f_a$.
Then
\begin{equation}\label{second_appl}
\begin{split}
P_mT_nP_m &= P_m
U^n_m\cdots U^n_1U^{n-1}_{f_1(m)}\cdots U^{n-1}_1\cdots U^2_{f_{(n-2)}(m)}\cdots U^2_1 U^1_{f_{(n-1)}(m)}\cdots
  U^1_1 \\ 
  & \quad \cdot P_{f_{n}(m)} T  P_{f_{n}(m)} \\
  & \quad \cdot U^1_1 \cdots  U^1_{f_{(n-1)}(m)} U^2_1  \cdots U^2_{f_{(n-2)}(m)}
  \cdots U^{n-1}_1\cdots U^{n-1}_{f_1(m)} U^n_1 \cdots U^n_m P_m.
\end{split}
\end{equation}
\end{theorem}

\begin{remark}
What Theorem \ref{QR-stuff} says is that to compute the finite section of size $m$ of the $n$-th iteration of the IQR algorithm (i.e. $P_mT_nP_m$), one only needs information from the finite section of size $f_{n}(m)$ (i.e. $P_{f_{n}(m)} T  P_{f_{n}(m)}$) since the relevant Householder reflections can also be computed from this information. In other words, the IQR algorithm can be computed.
\end{remark}

\begin{proof}[\textbf{Proof of Theorem \ref{QR-stuff}}:]
By induction it is enough to prove that
\begin{equation}
\label{obvious}
P_mT_nP_m=P_mU_m^n...U_1^nP_{f(m)}T_{n-1}P_{f(m)}U_1^n...U_1^mP_m
\end{equation}
To see why this is true, note that by the assumption that $T$ has quasi-banded subdiagonals with respect to $f$, Proposition \ref{subdiagonal} shows that $T_n$ has quasi-banded subdiagonals with respect to $f$ for all $n \in \mathbb{N}.$ Thus, it follows from the construction in the proof of Theorem \ref{QRthrm} that each $U^j_l$ is of the form 
$$
U^j_l = I_{l,j,1} \oplus \left(I_{l,j,2} - \frac{2}{\|\xi_{l,j}\|^2} \xi_{l,j}\otimes \bar \xi_{l,j}
\right)\oplus I_{l,j,3},
$$ 
where $I_{l,j,1}$ denotes the identity on $P_{l-1}\mathcal{H}$, $I_{l,j,2}$ denotes the identity on $\mathrm{span}\{e_k:l\leq k\leq f(l)\}$, $I_{l,j,3}$ denotes the identity on $P_{f(l)}^{\perp}\mathcal{H}$ and $\xi_{l,j}\in\mathrm{span}\{e_k:l\leq k\leq f(l)\}$. Since $P_m$ is compact, it then follows that
\begin{equation}\label{11}
\begin{split}
P_mT_nP_m &= (\underset{l \rightarrow \infty}{\text{SOT-lim}}\ P_m U_l^n...U_1^n)P_{f(m)}T_{n-1}P_{f(m)}(\underset{l \rightarrow \infty}{\text{SOT-lim}}\  U_1^n...U_l^nP_m)\\
&=P_mU_m^n...U_1^nP_{f(m)}T_{n-1}P_{f(m)}U_1^n...U_1^mP_m.
\end{split}
\end{equation}
\end{proof}

\begin{remark}
This result allows us to implement the IQR algorithm because each $U^j_l$ only affects finitely many columns or rows of $A$ if multiplied either on the left or the right. In computer science
it is often referred to as ``Lazy evaluation'' when one computes with infinite data structures, but defers the use of the information until needed. A simple implementation is shown in the appendix for the case that the matrix has $k$ subdiagonals (i.e. we have $f(n)=n+k$).
\end{remark}

%\subsubsection{Which operators fit this framework?}\label{restrict}

The next question is how restrictive is the assumption in Definition \ref{quasi_f}? In particular, suppose that $T \in \mathcal{B}(\mathcal{H})$ and that $\xi \in \mathcal{H}$ is a cyclic vector for $T$ (i.e. $\mathrm{span}\{\xi, T\xi, T^2\xi, \hdots\}$ is dense in $\mathcal{H}$). Then by applying the Gram-Schmidt procedure to $\{\xi, T\xi, T^2\xi, \hdots\}$ we 
obtain an orthonormal basis $\{\eta_1, \eta_2, \eta_3, \hdots\}$ for $\mathcal{H}$ such that the matrix 
representation of $T$ with respect to $\{\eta_1, \eta_2, \eta_3, \hdots\}$ is upper Hessenberg, and thus the matrix representation has only one subdiagonal. The question is therefore about the existence of a cyclic vector. Note that if $T$ does not have invariant subspaces then every vector $\xi \in \mathcal{B}(\mathcal{H})$ is a cyclic vector. Now what happens if $\xi$ is not cyclic for $T$? We may still form $\{\eta_1, \eta_2, \eta_3, \hdots\}$ as above however $\mathcal{H}_1 = \overline{\mathrm{span}\{\eta_1, \eta_2, \eta_3, \hdots\}}$ is now an invariant subspace for $T$ and $\mathcal{H}_1 \neq \mathcal{H}.$ We may still form a matrix representation of $T$ with respect to $\{\eta_1, \eta_2, \eta_3, \hdots\}$, but this will now be a matrix representation of $T|_{\mathcal{H}_1}.$ Obviously, we can have that $\sigma(T|_{\mathcal{H}_1}) \subsetneq \sigma(T)$.

However, the following example shows that the class of matrices for which we can compute the IQR algorithm covers a wide number of applications. In particular, it includes all finite interaction Hamiltonians on graphs. Such operators play a prominent role in solid state physics \cite{mattis1986few,mogilner1991hamiltonians} describing propagation of waves and spin waves as well as encompassing Jacobi operators studied in many physical models and integrable lattices \cite{teschl2000jacobi}.

\begin{example}
Consider a connected, undirected graph $G$, such that each vertex degree is finite and the set of vertices $V(G)$ is countably infinite. Consider the set of all bounded operators $A$ on $l^2(V(G))\cong l^2(\mathbb{N})$ such that the set $S(v):=\{w\in V:\left\langle w,Av\right\rangle\neq0\}$ is finite for any $v\in V$. Suppose our enumeration of the vertices obeys the following pattern. $v_1$'s neighbours (including itself) are $S_1=\{v_1,v_2,...,v_{q_1}\}$ for some finite $q_1$. The set of neighbours of these vertices is $S_2=\{e_1,...,e_{q_2}\}$ for some finite $q_2$ where we continue the enumeration of $S_1$ and this process continues inductively enumerating $S_m$. If we know $S(v)$ for all $v\in V$ then we can find an $f:\mathbb{N}\rightarrow\mathbb{N}$ such that $A_{j,m}=0$ if $|j|>f(m)$. We simply choose $f(n)=q_{r_n}$ where $r_n$ is minimal such that $\cup_{j\leq n}S(v_j)\subset S_{r_n}$.
\end{example}

\subsection{Invertible operators}
\label{invert_comp}

More generally, given an invertible operator $T$ with information on how its columns decay at infinity we can compute finite sections of the IQR iterates with \textit{error control}. For computing spectral properties, we can assume, by shifting $T\rightarrow T+\lambda I$ then translating by $-\lambda$ back, that the operator we are interested in is invertible, hence the invertibility criterion is not that restrictive. Throughout we will use the following lemma which says that for invertible operators, the QR decomposition is essentially unique.

\begin{lemma}
Let $T$ be an invertible operator (viewed as a matrix acting on $l^2(\mathbb{N})$), then there exists a unique decomposition $T=QR$ with $Q$ unitary and $R$ invertible, upper triangular such that $R_{ii}\in\mathbb{R}_{>0}$. Furthermore, any other ``QR'' decomposition $T=Q'R'$ has a diagonal matrix $D=\mathrm{Diag}(t_1,t_2,...)$ such that $\left|t_i\right|=1$ and $Q=Q'D$. In other words, the QR decomposition is unique up to phase choices.
\end{lemma}
\begin{proof}
Consider the QR decomposition already discussed in this paper, $T=Q''R''$. $T$ is invertible and hence $Q''$ is a surjective isometry so is unitary. Hence $R''=Q''^*T$ is invertible. Being upper triangular, it follows that $R''_{ii}\neq 0$ for all $i$. Choose $t_i\in \mathbb{T}$ such that $t_iR''_{ii}\in\mathbb{R}_{>0}$ and set $D=\mathrm{Diag}(t_1,t_2,...)$. Letting $Q=Q''D^*$ and $R=DR''$ we clearly have the decomposition as claimed.

Now suppose that $T=Q'R'$ then we can write $Q=Q'R'R^{-1}$. It follows that $R'R^{-1}$ is a unitary upper triangular matrix and hence must be of the form $D=\mathrm{Diag}(t_1,t_2,...)$ with $\left|t_i\right|=1$.
\end{proof}

Another way to see this result is to note that the columns of $Q$ are obtained by applying the Gram-Schmidt procedure to the columns of $T$. The restriction that $R_{ii}\in\mathbb{R}_{>0}$ can also be incorporated into Theorem \ref{QR-stuff}. Theorem \ref{QR-stuff} (in this subcase of invertibility) is then a consequence of the fact that if $T$ has quasi-banded subdiagonals with respect to $f$ then
$$
P_mT^nP_m=P_m(P_{f_n(m)}TP_{f_n(m)})^nP_m
$$
and the relations (\ref{pwr_ref}) - we can apply Gram-Schmidt (or a more stable modified version) to the columns of $P_{f_n(m)}TP_{f_n(m)}$ and truncate the resulting matrix.

Assume that given $T\in\mathcal{B}(l^2(\mathbb{N}))$ invertible (not necessarily with quasi-banded subdiagonals), we can evaluate an increasing family of increasing functions $g^j:\mathbb{N}\rightarrow\mathbb{N}$ such that defining the matrix $T_{(j)}$ with columns $\{P_{g^j(n)}Te_n\}$ we have that $T_{(j)}$ is invertible and
\begin{equation}
\label{col_decay}
\left\|(P_{g^j(n)}-I)Te_n\right\|\leq \frac{1}{j}.
\end{equation}
It is easy to see that such a sequence of functions must exist since any $S$ with $\left\|S-T\right\|\leq\left\|T^{-1}\right\|^{-1}$ is invertible. Given this information, without loss of generality by increasing the $g^j$s pointwise if necessary, applying H{\"o}lder's inequality and taking subsequences, we may assume that $\left\|T_{(j)}-T\right\|\leq 1/{j}.$ In other words, given a sequence of functions satisfying (\ref{col_decay}) we can evaluate a sequence of functions with this stronger condition. The following says that given such a sequence of functions, we can compute the truncations $P_mT_nP_m$ to a given precision.

\begin{theorem}
\label{can_compute_invertible}
Suppose $T\in\mathcal{B}(l^2(\mathbb{N}))$ is invertible and the family of functions $\{g^j\}$ are as above. Suppose also that we are given a bound $C$ such that $\left\|T\right\|\leq C$. Let $\epsilon>0$ and $m,n\in\mathbb{N}$, then we can choose $j$ such that applying Theorem \ref{QR-stuff} (with the diagonal operators to ensure $R_{ii}>0$) to $T_{(j)}$ using the function $g^j$ instead of $f$, we have the guaranteed bound
$$
\left\|P_mT_nP_m-P_mT_{(j),n}P_m\right\|\leq\epsilon,
$$
where $T_{(j),n}$ denotes the $n$-th IQR iterate of $T_{(j)}$.
\end{theorem}

\begin{proof}[\textbf{Proof of Theorem \ref{can_compute_invertible}}:]
First consider the error when applying Theorem \ref{QR-stuff} to $T_{(j)}$ with $g^j$ for any fixed $j$. We will show that we can compute an error bound which converges to zero as $j\rightarrow\infty$ and from this the theorem easily follows by successively computing the bound and halting when this bound is less than $\epsilon$.

Write the QR decompositions
$$
T^n=\hat Q_n \hat R_n,\quad  (T_{(j)})^n=\hat{Q}_{(j),n}\hat{R}_{(j),n}.
$$
We have $\left\|T-T_{(j)}\right\|\leq1/j$ and hence, by writing $T_{(j)}=T+(T_{(j)}-T)$, that
$$
\left\|T^n-(T_{(j)})^n\right\|\leq \sum_{k=1}^n {n \choose k}\frac{1}{j^k}C^{n-k}\leq \frac{(C+1)^n}{j}=\frac{\tilde C}{j},
$$
where $\tilde C=(C+1)^n$. The columns of $\hat Q_n$ and $\hat{Q}_{(j),n}$ are simply the columns of the matrices $T^n$ and $(T_{(j)})^n$ after the application of Gram-Schmidt. Let the first $m$ columns of $T^n$ and $(T_{(j)})^n$ be denoted by $\{t_k\}_{k=1}^m$ and $\{\tilde t_k^j\}_{k=1}^m$ respectively and let $\{q_k\}_{k=1}^m$ and $\{\tilde q_k^j\}_{k=1}^m$ be the vectors obtained after applying Gram-Schmidt to these sequences of vectors. We then have
\begin{equation}
\label{othog_rel_q}
\begin{split}
\|q_1-\tilde q_1^j\|&=\left\|\frac{t_1}{\|t_1\|}-\frac{\tilde t_1^j}{\|\tilde t_1^j\|}\right\|\\
&=\left\|\frac{t_1(\|\tilde t_1^j\|-\|t_1\|)}{\|t_1\|\|\tilde t_1^j\|}-\frac{(\tilde t_1^j-t_1)\|t_1\|}{\|t_1\|\|\tilde t_1^j\|}\right\|\leq \frac{2\|t_1-\tilde t_1^j\|}{\|\tilde t_1^j\|}\leq\frac{2\tilde C}{j\|\tilde t_1^j\|}.
\end{split}
\end{equation}

For a vector $v$ of unit norm, let $P_{\perp v}$ denote the orthogonal projection onto the space of vectors perpendicular to $v$. Note that for two such vectors $v,w$, we have $\left\|P_{\perp v}-P_{\perp w}\right\|\leq\left\|v-w\right\|$. Let
\begin{equation}
\label{proj_GS}
v_k=P_{\perp q_{k-1}}\cdots P_{\perp q_{1}}t_k,\quad \tilde v_k^j=P_{\perp \tilde q_{k-1}^j}\cdots P_{\perp \tilde q_{1}^j}\tilde t_k^j,
\end{equation}
then $q_k$ are just the normalised version of $v_k$ and likewise $\tilde q_k^j$ are just the normalised version of $\tilde v_k^j$. Suppose that for $\mu<k$ we have $\|q_\mu-\tilde q_\mu^j\|\leq \delta$ for some $\delta>0$. Then applying the above products of projections we have
\begin{align*}
\|v_k-\tilde v_k^j\|&\leq \|P_{\perp q_{k-1}}\cdots P_{\perp q_{1}}(t_k-\tilde t_k^j)\| + \|P_{\perp q_{k-1}}\cdots P_{\perp q_{1}}\tilde t_k^j-\tilde v_k^j\|\\
&\leq \|t_k-\tilde t_k^j\| + \| P_{\perp q_{k-1}}\cdots P_{\perp q_{1}}-P_{\perp \tilde q^j_{k-1}}\cdots P_{\perp \tilde q^j_{1}}\|\|\tilde t^j_k\|\\
&\leq \|t_k-\tilde t_k^j\|+(k-1)\delta\|\tilde t_k^j\|.
\end{align*}
In the last line we have used the fact that if the operators $\{A_l\}_{l=1}^m$ and $\{B_l\}_{l=1}^m$ have norm bounded by $1$, then
$$
\left\|\prod_{l=1}^mA_l-\prod_{l=1}^mB_l\right\|\leq\sum_{l=1}^m\|A_l-B_l\|.
$$
Applying the same argument as in the inequalities (\ref{othog_rel_q}) we see that
\begin{equation}
\label{proj_GS2}
\|q_k-\tilde q_k^j\|\leq \frac{2 (\|t_k-\tilde t_k^j\|+(k-1)\delta\|\tilde t_k^j\|)}{\|\tilde v_k^j\|}\leq \frac{2 ({\tilde C}/{j}+2(k-1)\delta\tilde C)}{\|\tilde v_k^j\|},
\end{equation}
since $\|\tilde t_k^j\|\leq C+\tilde C/j\leq 2\tilde C$. Now note that we can compute the $\|\tilde v_k^j\|$ from the proof of Theorem \ref{QR-stuff}. Set $\delta_1(j)=\frac{2\tilde C}{j\|\tilde t_1^j\|}$ and for $1<k\leq m$ define iteratively
$$
\delta_k(j)=\max\Big\{\delta_{k-1}(j),\frac{2 ({\tilde C}/{j}+2(k-1)\delta_{k-1}(j)\tilde C)}{\|\tilde v_k^j\|}\Big\}.
$$
We must have $\|q_k-\tilde q_k^j\|\leq\delta_m(j)$ for $1\leq k\leq m$ where we have now shown the $j$ dependence as an argument.

It follows that $\|(\hat Q_n-\hat{Q}_{(j),n})P_m\|\leq \sqrt{m}\delta_m(j)$ and hence that
\begin{align*}
\|P_mT_nP_m-P_mT_{(j),n}P_m\|&\leq \|P_m(\hat Q_n-\hat{Q}_{(j),n})^*T \hat Q_nP_m\|+\|P_m\hat{Q}_{(j),n}^*(T\hat Q_n -T_{(j)}\hat{Q}_{(j),n})P_m\|\\
&\leq \sqrt{m}\delta_m(j)C +\|(T-T_{(j)})\hat Q_{(j),n}P_m\|+\|T(\hat Q_n-\hat{Q}_{(j),n})P_m\|\\
&\leq 2\sqrt{m}\delta_m(j)C+\frac{1}{j}.
\end{align*}
So we need only show that $\delta_m(j)\rightarrow 0$ as $j\rightarrow\infty$. Note that as $j\rightarrow\infty$, the columns of $(T_{(j)})^n$ converge to that of $T^n$. It follows that $\tilde t_k^j$ converge to $t_k$ and $\tilde q^j_1$ converges to $q_1$. An easy inductive argument using (\ref{proj_GS}) and (\ref{proj_GS2}) shows that the vectors $\tilde q_{k}^j$ converge to $q_{k}$ and $\|\tilde v_k^j\|$ are bounded below. $\delta_m(j)\rightarrow 0$ now follows.
\end{proof}

\section{SCI classification theorems}
\label{class_thms}

In this section we will apply the above results to prove three new classification theorems in the SCI hierarchy. First, assume that $T \in \mathcal{B}(l^2(\mathbb{N}))$ is an invertible normal operator with $\sigma(T) = \omega \cup \Psi$, where $\omega \cap \Psi = \emptyset$, $\omega = \{\lambda_i\}_{i=1}^N,$ and the $\lambda_i$'s are isolated eigenvalues with multiplicity $m_i$ satisfying $|\lambda_1| > \hdots > |\lambda_N|.$ As usual, we also assume that $\sup\{|\theta|:\theta \in \Psi\} < |\lambda_N|$ and set 
\begin{equation}\label{eq:the_M}
M := m_1 + \hdots +m_N\in\mathbb{N}\cup\{\infty\}. 
\end{equation}
In this section we will assume for simplicity that all the $m_i$ except possibly $m_N$ are finite.
To be able to obtain the classification results we need two key assumptions.
\begin{itemize}
\item[(I)] ({\it Column decay}): We assume a much weaker condition than bandedness of the infinite matrix. Indeed, we suppose a known decay of the elements in the columns of $T$ that is described through a family of increasing functions $\{g^j\}_{j\in \mathbb{N}}$. In particular, $g^j:\mathbb{N}\rightarrow\mathbb{N}$ is such that defining the infinite matrix $T_{(j)}$ with columns $\{P_{g^j(n)}Te_n\}_{n\in\mathbb{N}}$ we have that $T_{(j)}$ is invertible and
\begin{equation}
\label{col_decay2}
\left\|(P_{g^j(n)}-I)Te_n\right\|\leq \frac{1}{j}, \qquad n \in \mathbb{N}.
\end{equation}

\item[(II)] ({\it Distance to span of eigenvectors}): In order to obtain error control ($\Delta_1$ classification) one needs to control the hidden constant in the $O(r^n)$ estimate in \eqref{eq:big_O}. This is done as follows, where $\{Q_n\}_{n\in\mathbb{N}}$ is a $Q$-sequence of $T$ with respect to $\{e_j\}_{j\in\mathbb{N}}$. Given finite $k<M+1$ with $m_1+...+m_{N-1}<k$, we will assume that if $l<N$ then $\{\chi_{\{\lambda_1,...,\lambda_l\}}(T)e_j\}_{j=1}^{m_1+...+m_{l}}$ are linearly independent. We also assume that $\{\chi_{\{\lambda_1,...,\lambda_{N}\}}(T)e_j\}_{j=1}^{k}$ are linearly independent. This simply ensures that the IQR algorithm converges with the expected ordering (largest eigenvalue in the first diagonal entry then in descending order). It follows from Theorems \ref{new_QR} and \ref{QR-theorem}, that there exist eigenspaces $E_1,...,E_{N}$ (with the last space depending on $k$ and the vectors $\{e_j\}$) corresponding to the eigenvalues $\lambda_{1},...,\lambda_{N}$ such that
\begin{itemize}
	\item $E_i=\mathrm{ker}(T-\lambda_i I)$ is the full eigenspace if $i<N$
	\item $\hat\delta\Big(\bigoplus_{i=1}^lE_i,\mathrm{span}\{ Q_ne_{j}\}_{j=1}^{\min\{m_1+...+m_{l},k\}}\Big)\rightarrow 0$ as $n\rightarrow\infty$ for $l=1,...,N$.
\end{itemize}
%There then exists a new basis $\{\tilde e_j\}_{j=1}^k$ for $\mathrm{span}\{e_j\}_{j=1}^k$ with the following properties (see Section \ref{first_proofs} for a construction using Gram-Schmidt):
%\begin{itemize}
%	\item For any $1\leq i\leq k$, $\mathrm{span}\{\tilde e_j\}_{j=1}^i=\mathrm{span}\{e_j\}_{j=1}^i$.
%	\item If $j>m_1+...+m_l$ then $\chi_{\lambda_l}(T)\tilde e_j=0$.
%	\item The vectors $\{\chi_{\lambda_l}\tilde e_j\}_{j=m_1+...+m_{l-1}+1}^{\min\{m_1+...+m_{l},k\}}$ are orthonormal.
%\end{itemize}
We then define the initial supremum subspace angle by
\begin{equation}
\label{sup_subspace_angle}
\Phi(T,\{e_j\}_{j=1}^k):=\sup_{l=1,...,N}\phi\Big(\bigoplus_{i=1}^lE_i,\mathrm{span}\{ e_{j}\}_{j=1}^{\min\{m_1+...+m_{l},k\}}\Big),
\end{equation}
%Since the Gram-Schmidt procedure is unique up to phase shifts, this is well defined.
where $\phi$, defined by (\ref{subspace_angle}), denotes the subspace angle. Our assumptions and the proofs in Section \ref{conv_main} show that $\Phi(T,\{e_j\}_{j=1}^k)<\pi/2$ and hence the key quantity $\tan\big(\Phi(T,\{e_j\}_{j=1}^k)\big)$ is finite.
\end{itemize}

\begin{remark}
The quantity $\tan\big(\Phi(T,\{e_j\}_{j=1}^k)\big)$ can be viewed as a measure of how far $\{e_j\}_{j=1}^k$ is from $\{q_j\}_{j=1}^k$, the $k$ eigenvectors of $T$ corresponding to the first $k$ eigenvalues (including multiplicity and preserving order). Hence it gives an estimate of how good the initial approximation $\{e_j\}_{j=1}^k$ to $\{q_j\}_{j=1}^k$ is. Indeed, we know from \eqref{eq:big_O} that the convergence rate is $O(r^n)$, and the hidden constant $C$ depends exactly on this behaviour. In particular, if $e_j = q_j$ for $j \leq k$ then $C = 0$.
\end{remark}

Define also
\[
r(T)=\max\{\left|\lambda_2/\lambda_1\right|,...,\left|\lambda_{N}/\lambda_{N-1}\right|,\rho(T)/\left|\lambda_{N}\right|\}, \qquad \rho(T)=\mathrm{sup}\{\left|z\right|:z\in\Psi\}.
\]
We can now define the class of operators $\Omega^k_{t,L}$ for the classification theorem.

\begin{definition}%[Defining the class of operators]
\label{complicate_def}
Given $k\in\mathbb{N}$, $t\in(0,1)$ and $L>0$, let $\Omega^k_{t,L}$ denote the class of invertible normal operators $T$ acting on $l^2(\mathbb{N})$ with $\left\|T\right\|\leq L$ such that:
\begin{enumerate}
	\item There exists the decomposition $\sigma(T) = \omega \cup \Psi$ as above with $m_1+...+m_{N-1}< k\leq M$, where $M$ is defined in \eqref{eq:the_M}.
	\item If $m_1+...+m_l<k$ then $\{\chi_{\{\lambda_1,...,\lambda_l\}}(T)e_j\}_{j=1}^{m_1+...+m_l}$ are linearly independent. Also, the vectors $\{\chi_{\{\lambda_1,...,\lambda_N\}}(T)e_j\}_{j=1}^{k}$ are linearly independent.
	\item We have access to functions $g^j:\mathbb{N}\rightarrow\mathbb{N}$ with (\ref{col_decay2}).
	\item It holds that $r(T)\leq t$ and $\tan\big(\Phi(T,\{e_j\}_{j=1}^k)\big)\leq L$.
\end{enumerate}
\end{definition}

We can now define the computational problem that we want to classify in the SCI hierarchy. Consider for any $T\in \Omega^k_{t,L}$, the problem of computing the $k$-th largest eigenvalues (including multiplicity) and the corresponding eigenspaces. In other words we consider the set valued mapping 
\[
\Xi_1(T) = \mathcal{S} \subset \mathcal{M} = \mathbb{C}^k\times \big(l^2(\mathbb{N})\big)^k
\]
where we define 
\begin{align*}
\mathcal{S}&:=\Big\{(\underbrace{\lambda_1,...,\lambda_1}_{\text{$m_1$ times}},...,\underbrace{\lambda_N,...,\lambda_N}_{\text{$k-(m_1+...+m_{N-1})$ times}})\times (\hat q_1,...,\hat q_k):\\
&\text{s.t. $\{\hat q_j\}_{j=m_1+...+m_{l-1}+1}^{m_1+...+m_l}$ is an orthonormal basis of $\mathrm{ran}(\chi_{\lambda_l}(T))$ for $l<N$}\\
&\text{and $\{\hat q_j\}_{j=m_1+...+m_{N-1}+1}^{k}$ is an orthonormal basis for a subspace of $\mathrm{ran}(\chi_{\lambda_N}(T))$}\Big\}.
\end{align*}
As discussed in Remark \ref{rem:set_valued} in \S \ref{SCI_basics}, where we review the SCI hierarchy, 
when we speak of convergence of $\mathcal{M}\ni\Gamma_n(T)$ to $\Xi_1(T)$, we define, with a slight abuse of notation, 
\[
\mathrm{dist}(\Gamma_n(T),\Xi_1(T)) := 
\inf_{y \in \Xi_1(T)} d_{\mathcal{M}}(\Gamma_n(T),y) \rightarrow 0. 
\]
%
%\[
%\inf_{x\in \Gamma_n(T), y \in \Xi(T)} d_{\mathcal{M}}(x,y) \rightarrow 0.
%\]
Having established the basic definition we can now present the classification theorem.

\begin{theorem}[$\Delta_1$ classification for the extremal part of the spectrum]
\label{class1}
Given the above setup we have $\{\Xi_1,\Omega^k_{t,L}\}\in \Delta_1.$ In other words, for all $n\in\mathbb{N}$, there exists a general tower using radicals, $\Gamma_n(T)$, such that for all $T\in\Omega^k_{t,L}$,
$$
\mathrm{dist}(\Gamma_n(T),\Xi_1(T))\leq 2^{-n}.
$$
\end{theorem}

\begin{remark}
Note that this means that we converge to the $k$ largest magnitude eigenvalues \textit{in order} with error control, and not just arbitrary points of the spectrum. This is in contrast to most $\Sigma_1$ classifications in the SCI hierarchy where the best we can hope for is to bound $\mathrm{dist}(z,\sigma(T))$ for $z\in\mathbb{C}$.
\end{remark}

\begin{proof}[\textbf{Proof of Theorem \ref{class1}}:]
Let $T\in\Omega_{t,L}^k$ then by the definition of $\Omega_{t,L}^k$, we may take $\hat e_j=e_j$ for $j=1,...,k$ in the arguments in Section \ref{prelims}. The first step is to bound $Z(T,\{e_j\}_{j=1}^{k})$ in terms of $\Phi(T,\{e_j\}_{j=1}^k)$. Let $\{\tilde e_j\}_{j=1}^k$ denote the basis described in Section \ref{prelims}. In our case:
\begin{itemize}
	\item For any $1\leq i\leq k$, $\mathrm{span}\{\tilde e_j\}_{j=1}^i=\mathrm{span}\{e_j\}_{j=1}^i$.
	\item If $j>m_1+...+m_l$ then $\chi_{\lambda_l}(T)\tilde e_j=0$.
	\item The vectors $\{\chi_{\lambda_l}(T)\tilde e_j\}_{j=m_1+...+m_{l-1}+1}^{\min\{m_1+...+m_{l},k\}}$ are orthonormal.
\end{itemize}
Let $\delta_j=\|\tilde e_j\|$ then we must have that if $m_1+...m_{l-1}<j\leq m_1+...+m_l$ then
\begin{align*}
\frac{\delta_j^2-1}{\delta_j^2}&\leq \delta\Big(\mathrm{span}\{\tilde e_j\},\bigoplus_{i=1}^l\mathrm{span}\{\chi_{\{\lambda_i\}}(T)\tilde e_j\}_{j=m_1+...m_{i-1}+1}^{\min\{m_1+...+m_{i},k\}}\Big)^2\\%\mathrm{span}\{\chi_{\{\lambda_l\}}(T)\tilde e_j\}_{j=m_1+..._m_{l-1}+1}^{\min\{m_1+...+m_{l},k\}})^2\\
&\leq \delta\Big(\mathrm{span}\{\tilde e_j\}_{j=1}^{\min\{m_1+...+m_{l},k\}},\bigoplus_{i=1}^l\mathrm{span}\{\chi_{\{\lambda_i\}}(T)\tilde e_j\}_{j=m_1+...m_{i-1}+1}^{\min\{m_1+...+m_{i},k\}}\Big)^2\\
&=\delta\Big(\mathrm{span}\{e_j\}_{j=1}^{\min\{m_1+...+m_{l},k\}},\bigoplus_{i=1}^l E_i\Big)^2\\
&\leq \sin^2\big(\Phi(T,\{e_j\}_{j=1}^k)\big)
\end{align*}
Where the first line holds since the nearest point to $\tilde e_j$ in $\bigoplus_{i=1}^l\mathrm{span}\{\chi_{\{\lambda_i\}}(T)\tilde e_j\}_{j=m_1+...m_{i-1}+1}^{\min\{m_1+...+m_{i},k\}}$ is simply $\chi_{\lambda_l}(T)\tilde e_j$ and the $E_i$ are defined as above and in (\ref{need_ejs}). Rearranging, this implies that
$$
\delta_j^2\leq \frac{1}{1-\sin^2\big(\Phi(T,\{e_j\}_{j=1}^k)\big)}=\frac{1}{\cos^2\big(\Phi(T,\{e_j\}_{j=1}^k)}.
$$
Hence it follows that
$$
Z(T,\{e_j\}_{j=1}^{k})=\Big(\sum_{j=1}^k \delta_j^2-1\Big)^{\frac{1}{2}}\leq \Big(\sum_{j=1}^k \tan^2\big(\Phi(T,\{e_j\}_{j=1}^k)\big)\Big)^{\frac{1}{2}}\leq \sqrt{k}L.
$$

In particular, Theorem \ref{QR-theorem} and its proof now implies that
$$
\hat\delta(\mathrm{span}\{\hat q_j\},\mathrm{span}\{Q_me_j\})\leq B(j)\sqrt{k}Lt^m,
$$
where $\{\hat q_j\}_{j=1}^k$ are orthonormal eigenvectors of $T$ and $Q_m$ is a $Q-$sequence of $T$. In particular, $\{B(j)\}_{j=1}^k$ can be computed in finitely many arithmetic operations from the induction proof of Theorem \ref{QR-theorem}. %Using (\ref{fd_property}) we have 
%$$
%\hat\delta(\mathrm{span}\{\hat q_j\},\mathrm{span}\{Q_me_j\})\leq B(j)kLt^m.
%$$
It follows that there exists $z_{j,m}\in\mathbb{C}$ of unit modulus such that defining $\beta=\max\{B(1),...,B(k)\}\sqrt{k}L$, we have
$$
\left\|Q_me_j-z_{j,m}\hat q_j\right\|\leq \beta t^m.
$$
Note that we do not need to assume knowledge of $N$ for this bound (trivially $N\leq k$). Using that $Q_m$ is an isometry, this implies that
$$
\left|\langle Q_m^*TQ_me_j,e_j\rangle - \lambda_{a_j}\right|\leq 2\left\|T\right\|\beta t^m\leq 2L\beta t^m,
$$
where $T\hat q_j=\lambda_{a_j}$. Note that we must have $\{\lambda_{a_j}\}_{j=m_1+...+m_{l-1}+1}^{m_1+...+m_l}=\lambda_l$ and $\{\lambda_{a_j}\}_{j=m_1+...+m_{N-1}+1}^{k}=\lambda_N$ by 3. in the definition of $\Omega_{t,L}^k$.

Given any $\epsilon>0$, choose $m$ large enough so that $2L\beta t^m\leq \epsilon$ and $\beta t^m\leq\epsilon$. The fact that $\left\|T\right\|\leq L$ and (\ref{col_decay2}) holds implies that we can compute $\langle Q_m^*TQ_me_j,e_j\rangle$ to accuracy $\epsilon$ using finitely many arithmetical and square root operations using Theorem \ref{can_compute_invertible}. Call these approximations $\tilde\lambda_1,\tilde\lambda_2,...,\tilde\lambda_k$. Furthermore, the proof of Theorem \ref{can_compute_invertible} also makes clear that we can compute $Q_me_j\in l^2(\mathbb{N})$ to accuracy $\epsilon$ using finitely many arithmetical and square root operations (the approximations have finite support). Call these approximations $\tilde q_1,\tilde q_2,...,\tilde q_k$. Then set
$$
\Gamma^{\epsilon}(T)=(\tilde\lambda_1,\tilde\lambda_2,...,\tilde\lambda_k)\times(\tilde q_1,\tilde q_2,...,\tilde q_k).
$$
The above estimates show that $\mathrm{dist}(\Gamma^{\epsilon}(T),\Xi_1(T))\leq 4k\epsilon$. The proof is completed by setting $\Gamma_{n}(T)=\Gamma^{2^{-(n+2)}/k}(T)$.
\end{proof}

Next suppose we have a continuous increasing function function $g:\mathbb{R}_{\geq 0}\rightarrow\mathbb{R}_{\geq 0}$ diverging at $\infty$ such that $g(0)=0$ and $g(x)\leq x$. Let $\Omega_{\mathrm{IQR}}^g$ be the set of all operators $T$ acting on $l^2(\mathbb{N})$ (i.e. we fix the representation w.r.t. the canonical basis) for which the IQR algorithm converges in the weak operator topology to a diagonal matrix with the same spectrum as $T$ and such that
$$
\left\|(T-zI)^{-1}\right\|^{-1}\geq g\big(\mathrm{dist}(z,\sigma(T))\big).
$$
Note that by Theorem \ref{new_QR} this includes all normal compact operators, $T$, such that $\{z\in\sigma(T):\left|z\right|=s\}$ has size at most $1$ for all $s>0$ (where we can take $g(x)=x$).\footnote{A simple compactness argument says that for any bounded operator $T$ there is a corresponding function $g$ that works.} We will allow evaluations of $g$ in our algorithms and also assume that we are given functions that satisfy (\ref{col_decay2}) and have an upper bound for $\|T\|$. We consider computing $\Xi_2(T)=\sigma(T)$ in the space of compact non-empty subsets of $\mathbb{C}$ with the Hausdorff metric.

\begin{theorem}[$\Sigma_1$ classification for spectrum]
\label{class2}
Given the above setup we have $\{\Xi_2,\Omega_{\mathrm{IQR}}^g\}\in \Sigma_1.$ In other words, there is a convergent sequence of general towers using radicals, $\Gamma_n(T)$, such that $\Gamma_n(T)\rightarrow\Xi_2(T)=\sigma(T)$ for any $T\in\Omega_{\mathrm{IQR}}^g$ and for all $n$ we have
$$
\Gamma_n(T)\subset \sigma(T)+B_{2^{-n}}(0).
$$
\end{theorem}

\begin{proof}[\textbf{Proof of Theorem \ref{class2}}:]
Let $T\in\Omega_{\mathrm{IQR}}^g$ and $Q_m$ be a $Q-$sequence of $T$. Fix $n\in\mathbb{N}$. Then Theorem \ref{can_compute_invertible} shows that we can compute any finite number of the diagonal entries of $Q_m^*TQ_m$ to any given accuracy using finitely many arithmetical and square root operations. Similarly, the proof shows that we can compute $TQ_me_j$ and $Q_me_j$ to any given accuracy in $l^2(\mathbb{N})$ (the approximations have finite support). Now let $\alpha_{j,m}$ be the computed approximations of $\langle Q_m^*TQ_me_j,e_j\rangle$ to accuracy $1/m$, then since $T\in\Omega_{\mathrm{IQR}}^g$ we have that $\lim_{m\rightarrow\infty}\alpha_{j,m}=\alpha_j\in\sigma(T)$. Furthermore, $\{\alpha_j:{j\in\mathbb{N}}\}$ is dense in $\sigma(T)$. We have that
$$
\left\|(T-\alpha_{j,m}I)^{-1}\right\|^{-1}\leq\left\|TQ_me_j-\alpha_{j,m}Q_me_j\right\|
$$
and hence that
\begin{equation}
\label{key_resolvent}
\mathrm{dist}(\alpha_{j,m},\sigma(T))\leq g^{-1}(\left\|TQ_me_j-\alpha_{j,m}Q_me_j\right\|).
\end{equation}
Given $m,j$, we can compute an upper bound $h_{j,m}$ for the right hand side of (\ref{key_resolvent}) by approximating the norm $\left\|TQ_me_j-\alpha_{j,m}Q_me_j\right\|$ from above to accuracy $1/m$ and finitely many evaluations of $g$. Namely, let $x_{j,m}$ be the approximation of $\left\|TQ_me_j-\alpha_{j,m}Q_me_j\right\|$ and set
$$
h_{j,m}=\frac{\min\{l\in\mathbb{N}:g(l/m)\geq x_{j,m}\}}{m}.
$$
It is then clear that $\lim_{m\rightarrow \infty}h_{j,m}=0$ and $h_{j,m}\geq g^{-1}(\left\|TQ_me_j-\alpha_{j,m}Q_me_j\right\|)$.

We set $\Gamma_{n}(T)=\{\alpha_{j,m(n,T)}:j=1,...,n\}$ where $m(n,T)$ is minimal such that $h_{j,m}\leq 2^{-n}$ for $j=1,...,n$. By (\ref{key_resolvent}), we must have that
$$
\Gamma_{n}(T)\subset\sigma(T)+B_{2^{-n}}(0).
$$
It is also clear that $\Gamma_n(T)\rightarrow\sigma(T)$ in the Hausdorff metric.
\end{proof}

The final result considers dominant invariant subspaces discussed in Theorem \ref{non_normal_prop2}. Let $M\in\mathbb{N}$, $t\in(0,1)$ and $L>0$. We let $\tilde\Omega^M_{t,L}$ denote the class of operators such that the assumptions of Theorem \ref{non_normal_prop2} hold (same $M$) and such that:
\begin{enumerate}
	\item $\beta/\alpha<t$
	\item $\max\Big\{\|T\|,\frac{\sin\big(\phi(\mathrm{span}\{e_j\}_{j=1}^M,\mathrm{ran}(P))\big)}{\cos\big(\phi(\mathrm{span}\{e_j\}_{j=1}^M,S)\big)}\Big(1+\frac{\|PT(I-P)\|}{\alpha-\beta}\Big)\Big\}\leq L$
\end{enumerate}
We also assume that we are given functions that satisfy (\ref{col_decay2}) and consider computing the dominant invariant subspace $\Xi_3(T)=\mathrm{ran}(P)$ in the space of $M$-dimensional subspaces of $l^2(\mathbb{N})$ equipped with the metric $\hat\delta$. %The following follows immediately from Theorems \ref{non_normal_prop2} and \ref{can_compute_invertible} and their proofs.

\begin{theorem}[$\Delta_1$ classification for dominant invariant subspace]
\label{class3}
Given the above setup we have $\{\Xi_3,\tilde\Omega^M_{t,L}\}\in \Delta_1.$ In other words, for all $n\in\mathbb{N}$, there exists a general tower using radicals, $\Gamma_n(T)$, each an $M$-dimensional subspace of $l^2(\mathbb{N})$, such that for all $T\in\tilde\Omega^M_{t,L}$,
$$
\hat \delta(\Gamma_n(T),\Xi_3(T))\leq 2^{-n}.
$$
\end{theorem}

\begin{proof}[\textbf{Proof of Theorem \ref{class3}}:]
Let $n\in\mathbb{N}$ and $T\in\tilde \Omega^M_{t,L}$. Then from Theorem \ref{non_normal_prop2}, we can choose $m$ large so that $t^mL<2^{-(n+1)}$, and hence
$$
\hat\delta(\mathrm{span}\{Q_me_j\}_{j=1}^M,\mathrm{ran}(P))<2^{-(n+1)}.
$$
Using Theorem \ref{can_compute_invertible} and its proof, given $\epsilon$ we can compute in finitely many arithmetical and square root operations, approximations $v_{m,j}(\epsilon)$ (of finite support) such that
$$
\|v_{m,j}(\epsilon)-Q_me_j\|\leq\epsilon.
$$
The vectors $\{Q_me_j\}_{j=1}^M$ are orthonormal, as are the approximations $\{v_{m,j}(\epsilon)\}_{j=1}^M$. A simple application of H\"older's inequality then yields
$$
\hat\delta(\mathrm{span}\{v_{m,j}(\epsilon)\}_{j=1}^M,\mathrm{span}\{Q_me_j\}_{j=1}^M)\leq \sqrt{M}\epsilon.
$$
By the triangle inequality, the proof of the theorem is complete by choosing $\epsilon$ such that $\sqrt{M}\epsilon\leq 2^{-(n+1)}$ and then setting $\Gamma_n(T)=\mathrm{span}\{v_{m,j}(\epsilon)\}_{j=1}^M$.
\end{proof}

\section{Examples and numerical simulations}
\label{sec_5}

The aim of the is section is threefold:
\begin{enumerate}
	\item To demonstrate the convergence and implementation results of Section \ref{conv_main}--\ref{class_thms} on practical examples.
	\item To demonstrate that, as well as the proven results, the IQR algorithm performs better than theoretically expected in many cases. In particular we conjecture that for normal operators whose essential spectrum has exactly one extremal point, the IQR algorithm will also converge to this point. We also demonstrate cases where this seems to hold even if there are multiple extreme points of the essential spectrum and even in non-normal cases.
	\item To compare the IQR algorithm to the finite section method and show that in some cases it considerably outperforms it. In general one can view $\sigma(P_mQ_n^*TQ_n|_{P_m\mathcal{H}})$ as a generalised version of the finite section method, now with two parameters ($m$ and $n$) that can be varied with $n$ controlling the number of IQR iterates. In some cases we find this avoids spectral pollution whilst still converging to the entire spectrum.
\end{enumerate}

Before embarking with some numerical examples, two remarks are in order. First, extra care has been taken in the case of non self-adjoint operators whose finite truncations can be non-normal and hence the computation of their spectra can be numerically unstable. Unless stated otherwise, all calculations were performed in double precision (in MATLAB) and have been checked against extended precision \cite{mct2015} to ensure that none of the results are due to numerical artefacts. Second, when dealing with operators acting on $l^2(\mathbb{Z})$ we use $\mathbb{N}$ as an index set by listing the canonical basis as $e_0,e_1,e_{-1},e_2,e_{-2},...$, allowing us to apply the IQR algorithm on $l^2(\mathbb{N})$. Of course different indexing is possible and in general this would lead to different implementations of the IQR algorithm,\footnote{A discussion of this is beyond the scope of this paper. In effect, for invertible operators, this corresponds to choosing the order of columns on which to perform a Gram-Schmidt type procedure.} but we stick with this ordering throughout.

\subsection{The finite section method}
\label{s:FS}

We first briefly say a few words on the finite section method, the standard means to discretise infinite matrices, since comparisons will be made later. If
$\{P_m\}_{m \in \mathbb{N}}$ is a sequence of finite-rank projections such that $P_{m+1} \geq P_m$ and $P_m \rightarrow I$ strongly, where $I$ is the identity, then the idea is to replace $T$ by the finite square matrix $P_m T |_{P_m\mathcal{H}}$ (typically, one takes $P_m$ to be the orthogonal projection onto $\mathrm{span} \{e_1,\ldots,e_m \}$). Thus, to find $\sigma(T)$, we instead compute $\sigma(P_m T |_{P_m\mathcal{H}})$. However, there can be significant issues when using the finite section method. In general, there is no guarantee that the computed spectra $\sigma(P_m T |_{P_m\mathcal{H}})$ need converge to $\sigma(T)$.

For example, consider the shift operator $Se_j = e_{j+1}$ on $l^2(\mathbb{N})$. If $P_m$ projects onto $\mathrm{span} \{e_1,\hdots,e_m\},$ we would get that $\sigma((P_mS |_{P_m\mathcal{H}}) = \{0\}$ for all $m$, whereas $\sigma(S)$ is the closed unit disc. We can also have that $\sigma(P_mT |_{P_m\mathcal{H}}) \nsubseteq \sigma(T).$ For example, let
\begin{equation}\label{eq:the_A}
A = \left(
\begin{matrix}
 a_1      & i     &        &       &       \\
 1     & a_2     & i  &  &   \\
      & 1     & a_3       & i      &   \\
      &      & 1       & a_4      &  \ddots \\
  &   &  & \ddots & \ddots\\
\end{matrix}
\right),
\end{equation}
where $a_j=5\cos(j)/4+2i\sin(j)$. To gain an accurate picture of the spectrum, note that $A$ is banded and hence we can compute approximates to the pseudospectrum \cite{hansen2011}. In order to approximate the spectrum in the best possible way we must take $\epsilon$ as small as possible. Unfortunately, there is a restriction to how small $\epsilon$ can be depending on $\epsilon_{\mathrm{mach}}$ (machine precision) of the software used. To illustrate this, observe that the approximates are given by (a discrete version of)
\begin{equation}\label{root}
\begin{split}
\Gamma_{m}(A) &= \{z \in \mathbb{C}
: \min \{ \sqrt{\lambda} : \lambda \in \sigma_0(P_{m}(A -
  z)^*(A - z)\vert_{P_{m}\mathcal{H}})\} \leq \epsilon\}
\\& \qquad \quad \cup  \{z \in \mathbb{C}
: \min \{ \sqrt{\lambda} : \lambda \in \sigma_0(P_{m}(A -
  z)(A - z)^*\vert_{P_{m}\mathcal{H}})\} \leq \epsilon\}.
\end{split}
\end{equation}
Thus, ignoring the additional error in computing the smallest singular values (denoted by $\sigma_0$) and assuming $A$ to have matrix entries of order $1$, computing $\Gamma_{m}(A)$ will have the same challenges as if one squares a real number and then takes its square root. In particular, due to the floating point arithmetic used in the software and (\ref{root}) we must at least have that 
$$
\epsilon \gtrsim \sqrt{\epsilon_{\mathrm{mach}}},
$$
and this puts a serious restriction on our computation, particularly for the non-normal case where the distance $d_H(\sigma(T),\sigma_{\epsilon}(T))$ may be large (though we always have $\sigma(T)\subset\sigma_{\epsilon}(T)$). However, it is possible to detect spectral pollution outside of $\sigma_{\epsilon}(T)$ if we can approximate it well.

The phenomenon of ``spectral pollution'' occurs for $A$: namely, the computed spectrum $\sigma(P_mA |_{P_m\mathcal{H}})$ contains elements that have nothing to do with $\sigma(A).$ This is visualised in Fig. \ref{pollution1}, an example with spectral pollution $z\notin\sigma_{1/10}(A)$ where the same phenomenon occurs for larger $m$. The spectral pollution phenomenon is well known. As the following theorem suggests, such pollution can be arbitrarily bad.

\begin{figure}%[htp]
\centering
\includegraphics[width=0.495\textwidth,trim={32mm 90mm 35mm 95mm},clip]{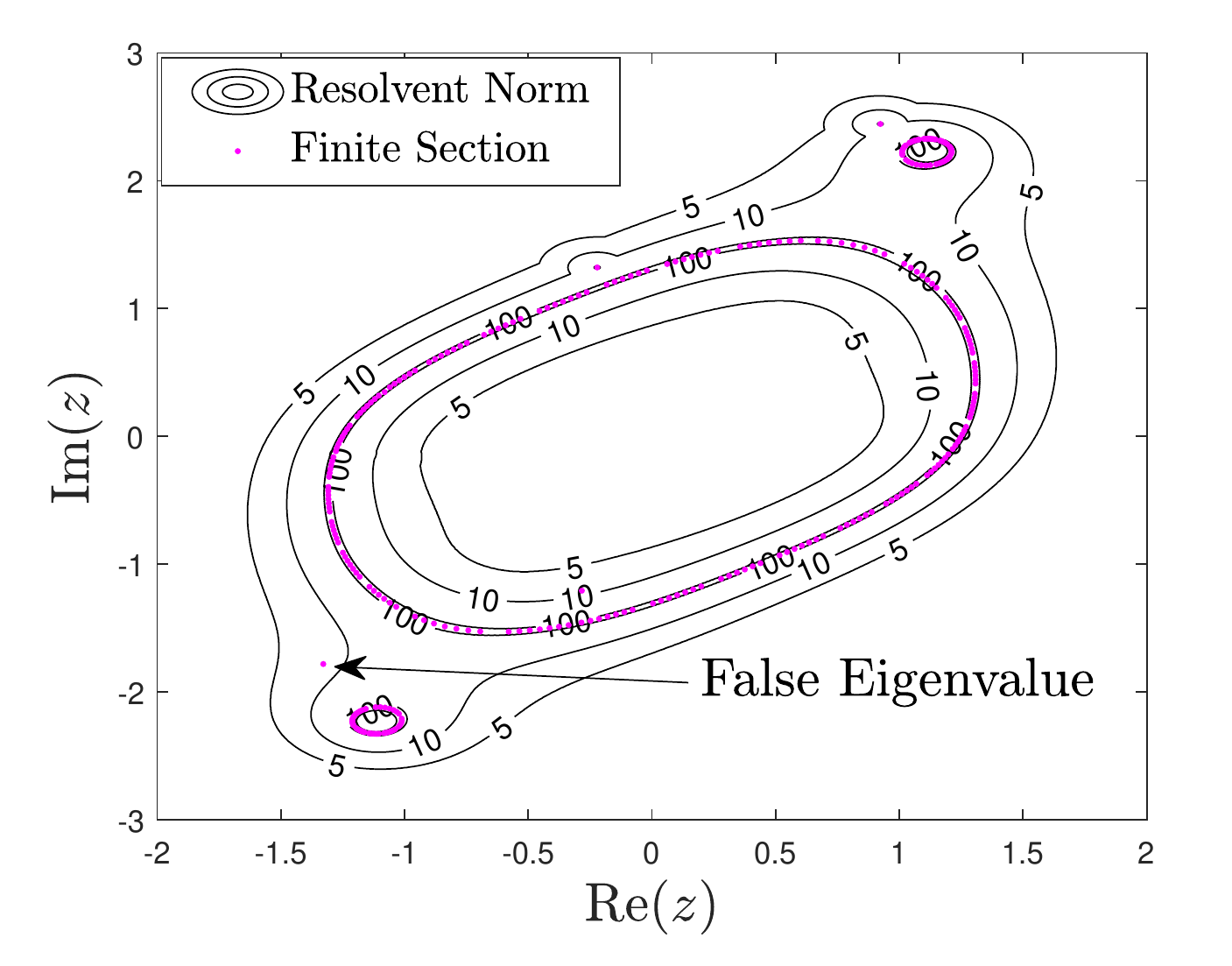}
\includegraphics[width=0.495\textwidth,trim={32mm 90mm 35mm 95mm},clip]{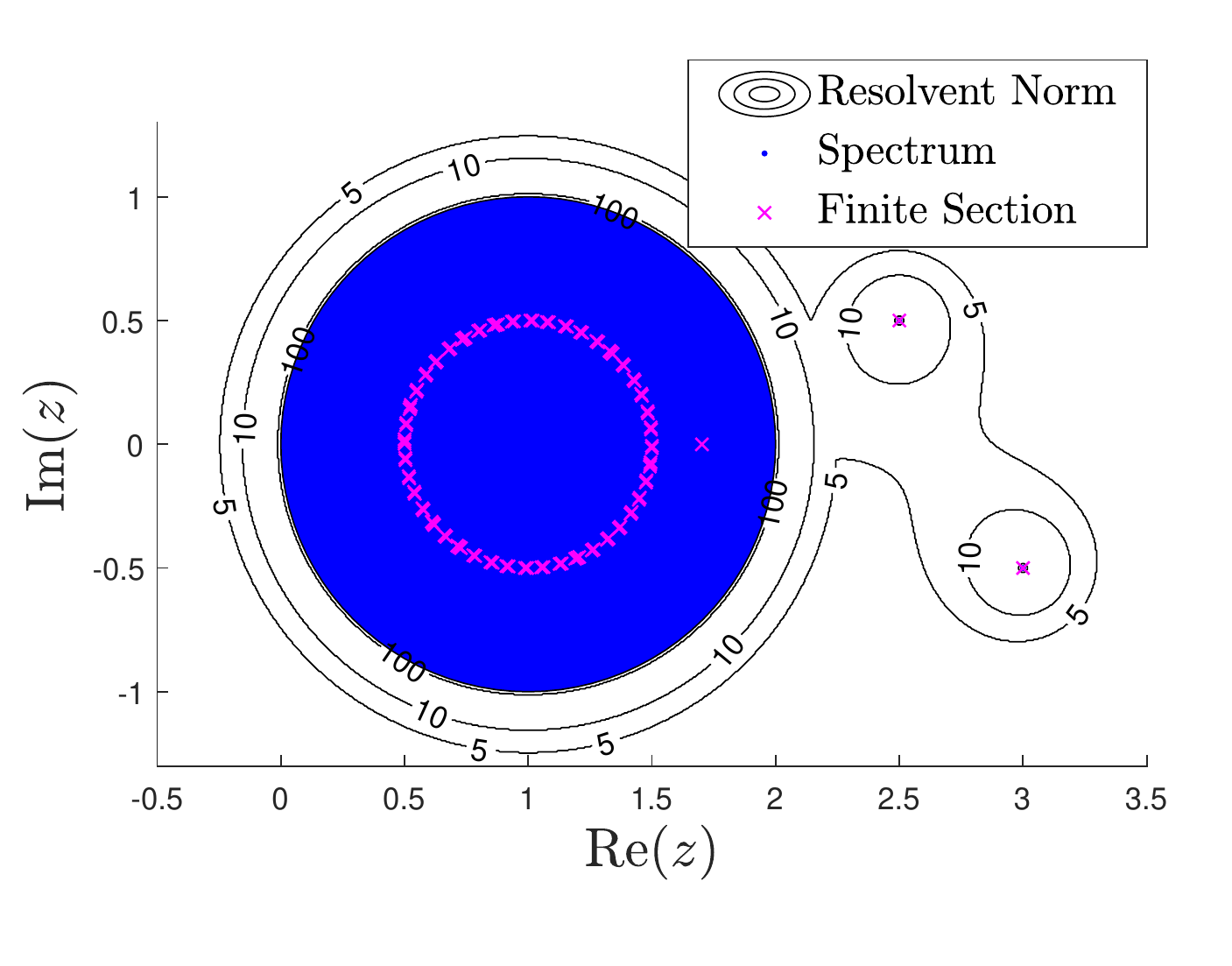}
\caption{Left: $\sigma_{\epsilon}(A)$ plotted as contours of the resolvent norm, as well as $\sigma(P_mA |_{P_m\mathcal{H}})$ for $m = 300$ with the false eigenvalue (recall that $\sigma(A) \subseteq \sigma_{\epsilon}(A)$). Right: $\sigma_{\epsilon}(T)$, $\sigma(T)$ and $\sigma(P_mT |_{P_m\mathcal{H}})$ for $m = 100$.}
\label{pollution1}
\end{figure}

\begin{theorem}[Pokrzywa \cite{Pokrzywa_79}]\label{polish}
Let $A \in \mathcal{B}(\mathcal{H})$ and $\{P_m\}$ be a sequence of
finite-dimensional projections converging strongly to the
identity. Suppose that $S \subset W_e(A).$ Then there exists a sequence
$\{\tilde{P}_m\}$ of finite-dimensional projections such that $P_m < \tilde{P}_m$ (so
$\tilde{P}_m \rightarrow I$ strongly) and 
\[
d_H(\sigma(A_m) \cup S, \sigma(\tilde A_m)) \rightarrow 0,
 \quad \text{as }m \rightarrow \infty, 
\]
where
\[
A_m = P_mA |_{P_m\mathcal{H}}, \qquad \tilde A_m =
\tilde P_mA |_{\tilde P_m\mathcal{H}} 
\]
and $d_H$ denotes the Hausdorff metric.
\end{theorem}

Despite this result, the finite section can perform quite well. This is the case for self adjoint operators \cite{Arveson_cnum_lin94,brown2007quasi,hansen2008} and it is also well suited for the computation of pseudospectra of Toeplitz operators \cite{Bottcher_pseu,Bottcher_book}. Moreover, in general, we have the following (recall that $W_e(T)$ is the convex hull of the essential spectrum for $T$ normal):

\begin{theorem}[Pokrzywa \cite{Pokrzywa_79}]\label{t:conv}
Let $T \in \mathcal{B}(\mathcal{H})$ and $\{P_m\}$ be a sequence of
finite-dimensional projections converging strongly to the
identity. If $\lambda \notin W_e(T)$ then $\lambda \in
\sigma(T)$ if and only if 
$$
\mathrm{dist}(\lambda,\sigma( P_mT|_{P_m\mathcal{H}})) \longrightarrow 0,
\qquad \text{as } m \rightarrow \infty.
$$ 
\end{theorem}

However, if we want to use the finite section method and rely on Theorem
\ref{t:conv} we must know $W_e(T),$ and that may be unpleasant to
compute. Alternatively, we could hope that $\sigma_{\mathrm{ess}}(T)$ is close to $W_e(T)$. For example if $T$ is hypo-normal
($T^*T - TT^* \geq 0$) then
$$
\mathrm{conv}(\sigma_{\mathrm{ess}}(T)) = W_e(T),
$$
where $\mathrm{conv}(\sigma_{\mathrm{ess}}(T))$ denotes the convex hull of
$\sigma_{\mathrm{ess}}(T).$ But what if we have a ``very non-normal'' operator? 

Another problem we may encounter using the finite section method is that even
though $\sigma_d(T)$ may be recovered, one may get a very misleading
picture of the rest of the spectrum. Such problems are illustrated in
the following simple example. Let

\begin{equation}\label{eq:the_T}
T = \left(
\begin{array}{cccc|cccc}
 2.5 +0.5i  & 0        & 0      & 0     &  0     & 0      &0        &\cdots\\
 1          & 3 -0.5i  & 0      & 0     &  0     & 0      &0        &\cdots\\
 0          & 1        & 1.7    & 0.05  &  0     & 0      &0        &\cdots\\
 0          & 0        & 0.05   & t_4     &  0     & 0      &0        &\cdots\\
\hline
 0          & 0        & 0      & 0     &  t_5     & 0      &0        &\cdots\\
 0          & 0        & 0      & 0     &  1     & t_6      &0        &\cdots\\
 0          & 0        & 0      & 0     &  0     & 1      &t_7        &\cdots\\
                                                                
\vdots      & \vdots   & \vdots & \vdots& \vdots & \vdots &\ddots   &\ddots\\
\end{array}
\right),
\end{equation}
where $t_j=1+0.5(\sin(j)+i\cos(j))$ for $j \geq 4.$ This operator decomposes into an upper $4\times 4$ block and an operator acting on the perpendicular subspace. It is also possible to compute the spectrum analytically (it consists of a disc of radius $1$ centred at $1$ together with two isolated eigenvalues). Again, we can compute the pseudospectrum of $T$ (Fig. \ref{pollution1}) to reveal that whilst the eigenvalues produced by the finite section method are correct, they do not capture the entire spectrum. It is straightforward to adapt this example (e.g. by changing basis) to have the same phenomena without an obvious decomposition of the operator into a finite part and triangular part. Without the support from the picture of the pseudospectrum, the finite section method does not provide information regarding the boundary of the essential numerical range of $T$ - there is a misleading circle of eigenvalues of $P_mT|_{P_m\mathcal{H}}$ which do not occur along the boundary of the essential spectrum but are simply given by the diagonal entries $\{t_5,t_6,...,t_m\}$.

\begin{remark}
The previous examples demonstrated that, in general, the finite section method is not always suitable for computing spectra. Rather then working with square sections of the infinite matrix $T$, one should work with \textit{uneven sections} $P_n T P_m$, where the parameters $n$ and $m$ are allowed to vary independently. Indeed, the algorithms presented in \cite{colb1,hansen2011} use this method. In effect, we need to know how large $n$ should be to retain enough information of the operator $TP_m$. This type of idea is also used implicitly in the IQR algorithm (see Section \ref{implement}).
\end{remark}

\subsection{Numerical examples I: normal operators}
\label{s:num_1}

\begin{example}[Convergence of the IQR algorithm]
We begin with two simple examples that demonstrate the linear (or exponential) convergence proven in Theorem \ref{new_QR} and Corollary \ref{new_QRCOR} (and its generalisations). Consider first the one-dimensional discrete Schr\"odinger operator given by
$$ T_1 = \left(
\begin{matrix}
 v_1      & 1     &        &       &        \\
 1      & v_2     & 1  &  &   \\
       & 1     & v_3       & 1      &   \\
      &      & 1       & v_4      &  \ddots \\
 &   &  & \ddots & \ddots\\
\end{matrix}
\right),
$$
where $v_j=5\sin(j)^2/\sqrt{j}$ if $j\leq 10$ and $v_j=0$ otherwise. As a compact (in fact finite rank) perturbation of the free Laplacian, $\sigma(T_1)$ consists of the interval $[-2,2]$ together with isolated eigenvalues of finite multiplicity which can be computed \cite{webb2017spectra}. The second operator, $T_2$, consists of taking the operator
$$
T_0=\left(
\begin{matrix}
 2      & 0     & 0       & 0 \\
 0      & \frac{3i}{2}     & 0       & 0 \\
 0      & 0     & -\frac{5}{4}       & 0 \\
0      & 0     & 0       & -\frac{9i}{8} 
\end{matrix}
\right)\bigoplus U_{1},
$$
where $U_{k}$ denotes the bilateral shift $e_j\rightarrow e_{j+k}$, writing this as an operator on $l^2(\mathbb{N})$ and then mixing the spaces via a random unitary transformation on the span of the first $9$ basis vectors. This ensures $T_2$ is not written in block form but has known eigenvalues. We have plotted the difference in norm between the first $j\times j$ block of each $Q_n^*T_lQ_n$ and the diagonal operator formed via the largest $j$ eigenvalues for $j=1,2,3$ and $4$ in Fig. \ref{normal1}. The plot clearly shows the exponential convergence.
\end{example}

\begin{figure}%[htp]
\centering
\includegraphics[width=0.495\textwidth,trim={32mm 92mm 35mm 92mm},clip]{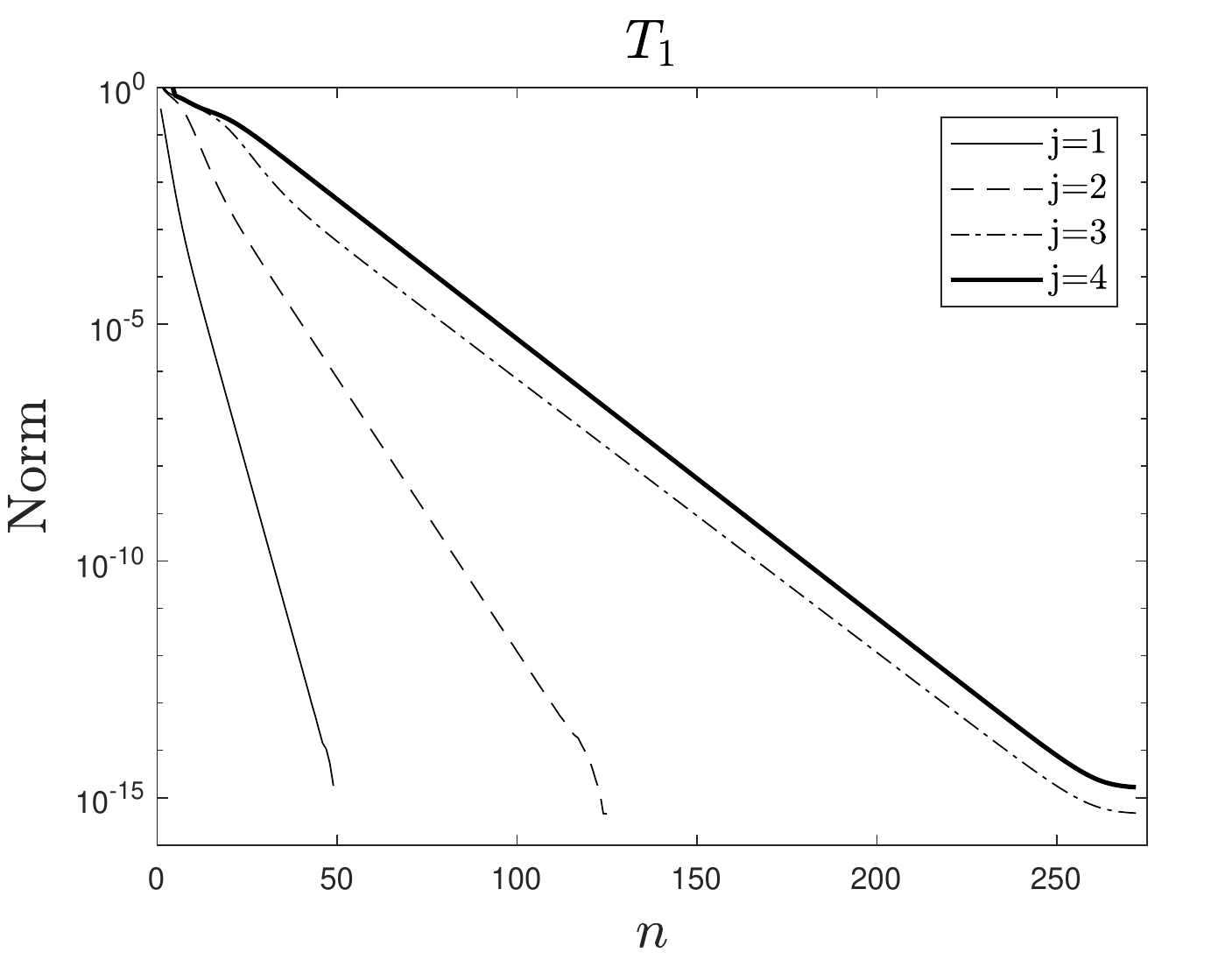}
\includegraphics[width=0.495\textwidth,trim={32mm 92mm 35mm 92mm},clip]{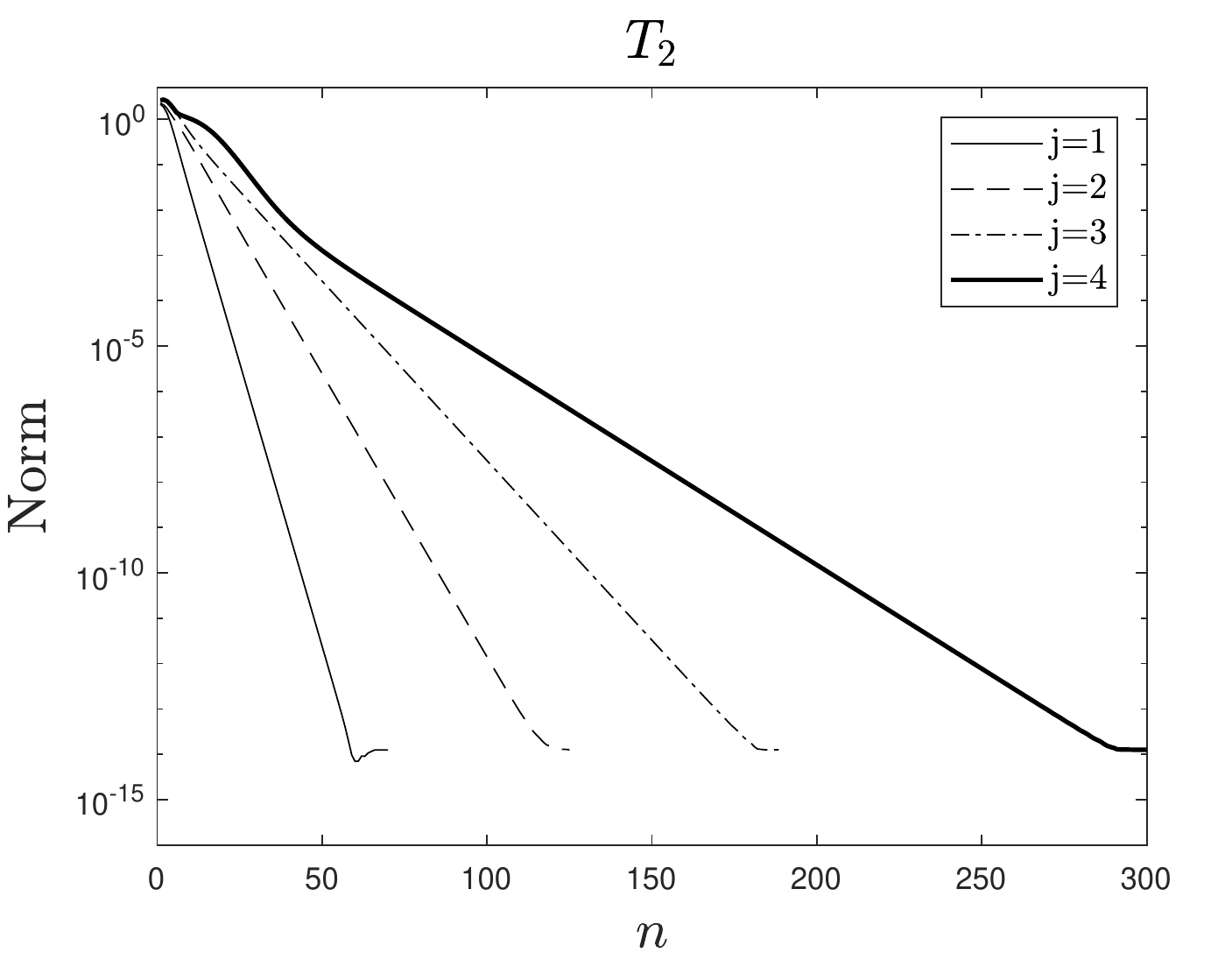}
\caption{Exponential convergence to the diagonal blocks for $T_1$ and $T_2$.}
\label{normal1}
\end{figure}

\begin{example}[Convergence to extremal parts of the spectrum]
To see why we may need some condition on $\sigma(T)$ for convergence of the IQR algorithm to the extreme parts of the spectrum, we consider Laurent and Toeplitz operators with symbol given by a trigonometric polynomial
$$
a(t)=\sum_{j=-k}^{j=k}a_jt^j.
$$
Given such a symbol, we define Laurent and Toeplitz operators
$$
L(a)=\left(
\begin{array}{ccc|cccc}
 \cdots & \cdots & \cdots & \cdots & \cdots & \cdots & \cdots \\
 \cdots & a_0 & a_{-1} & a_{-2} & a_{-3} & a_{-4} & \cdots \\
 \cdots & a_1 & a_0 & a_{-1} & a_{-2} & a_{-3} & \cdots \\
\hline
\cdots & a_2 & a_1 & a_0 & a_{-1} & a_{-2} & \cdots \\
\cdots & a_3 & a_2 & a_1 & a_0 & a_{-1} & \cdots \\
\cdots & a_4 & a_3 & a_2 & a_1 & a_0 & \cdots \\
\cdots & \cdots & \cdots & \cdots & \cdots & \cdots & \cdots 
\end{array}
\right),\quad T(a)= \left(
\begin{array}{cccc}
a_0 & a_{-1} & a_{-2} & \cdots \\
a_1 & a_0 & a_{-1} & \cdots \\
a_2 & a_1 & a_0 & \cdots \\
\cdots & \cdots & \cdots & \cdots 
\end{array}
\right),
$$
acting on $l^2(\mathbb{Z})$ and $l^2(\mathbb{N})$ respectively. Note that $L(a)$ is always normal whereas $T(a)$ need not be (see for example \cite{Bottcher_book}). A simple example already mentioned is $a(t)=t$ which gives rise to the bilateral and unilateral shifts $L(a)=U_1$ and $T(a)=S$. In this case, both of these operators are invariant under iterations of the IQR algorithm and hence their finite sections $P_mQ_n^*TQ_n|_{P_m\mathcal{H}}$ always have spectrum $\{0\}$. In the case of $L(a)$ this is an example of spectral pollution, whereas in the case of $T(a)$ this does not capture the extremal parts of the spectrum. Regarding pure finite section, the following beautiful result is known:

\begin{theorem}[Schmidt-Spitzer \cite{schmidt1960toeplitz}]
\label{nice_teop}
If $a$ is a trigonometric polynomial then we have the following convergence in the Hausdorff metric:
$$
\lim_{m\rightarrow\infty} \sigma(P_mL(a)|_{P_m\mathcal{H}})=\lim_{m\rightarrow\infty} \sigma(P_mT(a)|_{P_m\mathcal{H}})=\bigcap_{r\in(0,\infty)}\sigma(T(a_r))=:\Upsilon(a),
$$
where $a_r(t)=a(rt)$. Furthermore, this limit set is a connected finite union of analytic arcs, each pair of which has at most endpoints in common.
\end{theorem}

It is straightforward to construct examples where it appears that both $\lim_{n\rightarrow\infty}P_mQ_n^*T(a)Q_n|_{P_m\mathcal{H}}$ and $\lim_{n\rightarrow\infty}P_mQ_n^*L(a)Q_n|_{P_m\mathcal{H}}$ exist and are either the extreme parts of $\sigma(L(a))$ or of $\Upsilon(a)$. For example consider the symbols
$$a(t)=\frac{t^3+t^{-1}}{2},\quad \tilde{a}(t)=t+it^{-2}.$$
Fig. \ref{rose1} shows the outputs of the IQR algorithm and plain finite section for the corresponding Laurent and Toeplitz operators for $m=50$ and $n=1$ and $n=300$. In the case of $a$, it appears that both limit sets are the extremal parts of $\sigma(L(a))$ (together with $0$ if $m$ is not a multiple of $4$). Whereas in the case of $\tilde{a}$ it appears that $\lim_{n\rightarrow\infty}P_mQ_n^*T(a)Q_n|_{P_m\mathcal{H}}$ is the extremal parts of $\Upsilon(a)$ and $\lim_{n\rightarrow\infty}P_mQ_n^*L(a)Q_n|_{P_m\mathcal{H}}$ is the extremal parts of $\sigma(L(a))$ (again together with a finite collection of points depending on the value of $m$ modulo $3$). Curiously, in both cases we observed convergence in the strong operator topology to block diagonal operators (up to unitary equivalence in each sublock), whose blocks have spectra corresponding to the limiting sets (hence the dependence on remainder of $m$ modulo $2$ or $3$). However, in contrast to convergence to points in the discrete spectrum, convergence to these operators was only algebraic. This is shown in Fig. \ref{rose2} where we have plotted the Hausdorff distance between the limiting set and the eigenvalues of the first diagonal block. We also shifted the operators ($+1.1 I$ for $a$ and $-1.5iI$ for $\tilde{a}$) so that the extremal points correspond to exactly one point. In this scenario and for all operators (Laurent or Toeplitz) the IQR algorithm converges strongly to a diagonal operator whose diagonal entries are the corresponding extremal point of $\sigma(L(a))$. This convergence is also shown in Fig. \ref{rose2} and we observed a slower rate of convergence than before. This is possibly due to points from the other tips of the petals of $\sigma(L(a))$ converging as we increase $n$. It would be interesting to see if some form of Theorem \ref{nice_teop} holds for the IQR algorithm (now taking $n\rightarrow\infty$). Given the examples presented here, such a statement would likely be quite complicated. However, we conjecture that if a normal operator has exactly one extreme point of its essential spectrum (and finitely many eigenvalues of magnitude greater than $r_{\mathrm{ess}}$) then this extreme point will be recovered in the limit $n\rightarrow\infty$ for large enough $m$.

\begin{figure}%[htp]
\centering
\includegraphics[width=0.495\textwidth,trim={32mm 92mm 35mm 92mm},clip]{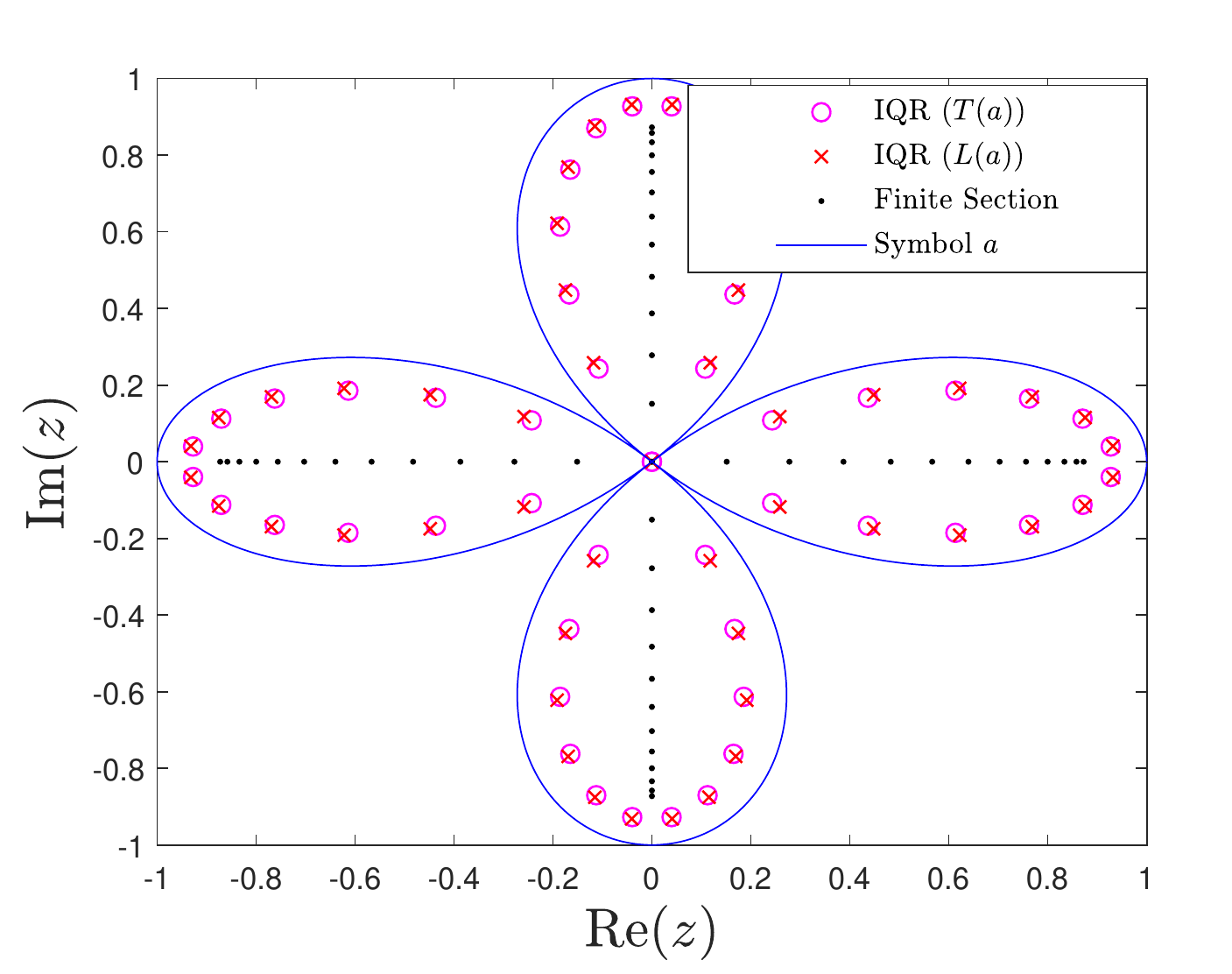}
\includegraphics[width=0.495\textwidth,trim={32mm 92mm 35mm 92mm},clip]{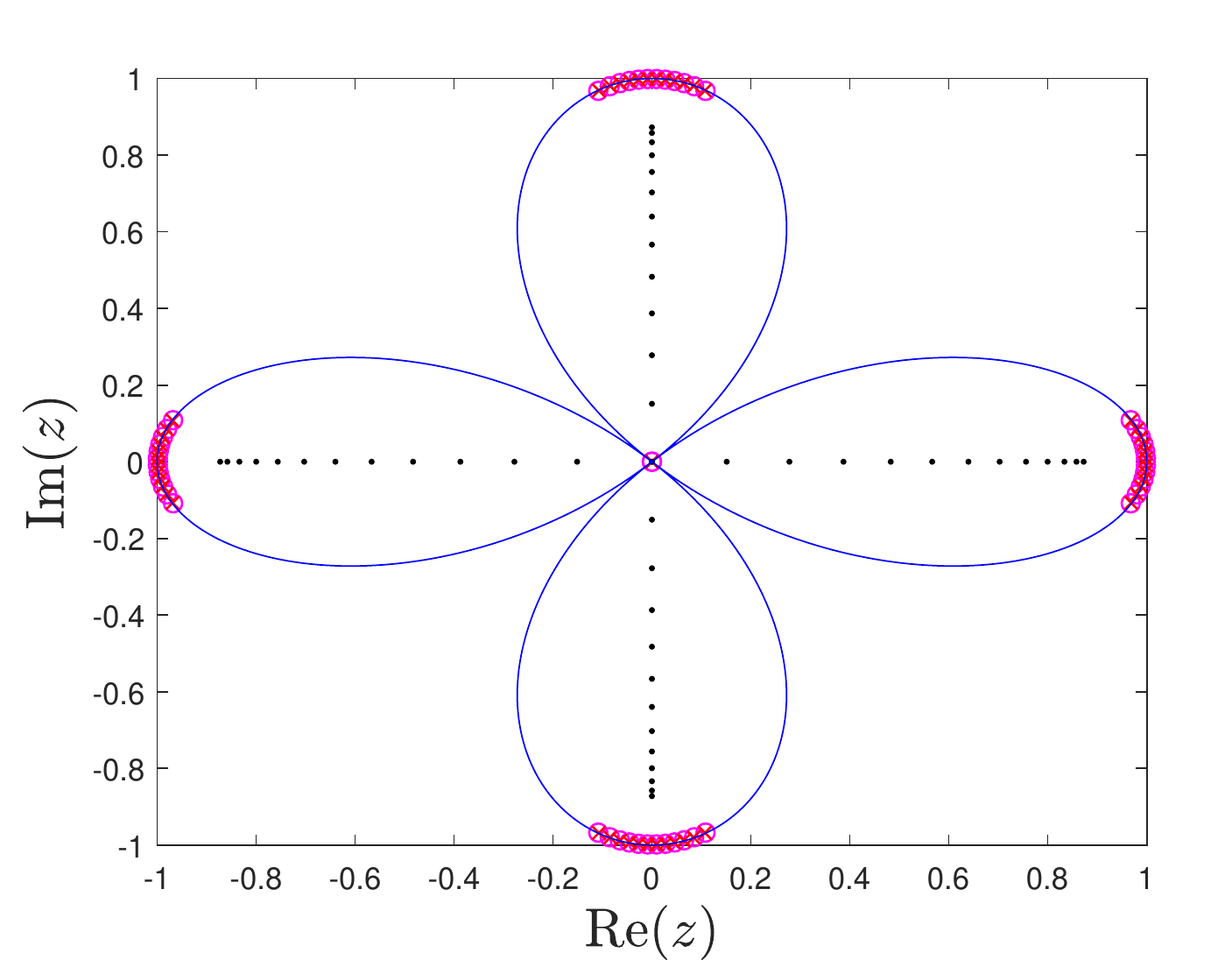}
\includegraphics[width=0.495\textwidth,trim={32mm 92mm 35mm 92mm},clip]{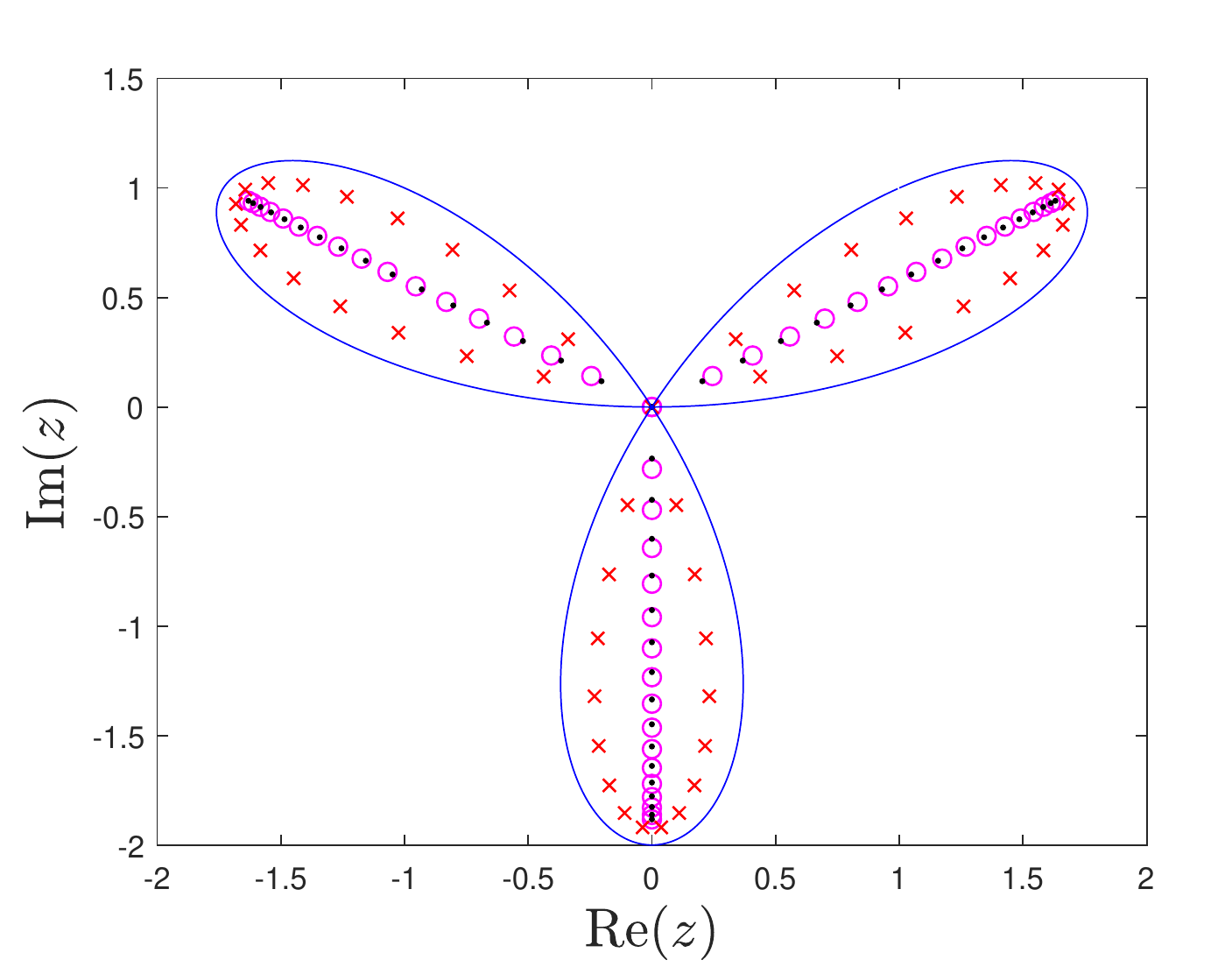}
\includegraphics[width=0.495\textwidth,trim={32mm 92mm 35mm 92mm},clip]{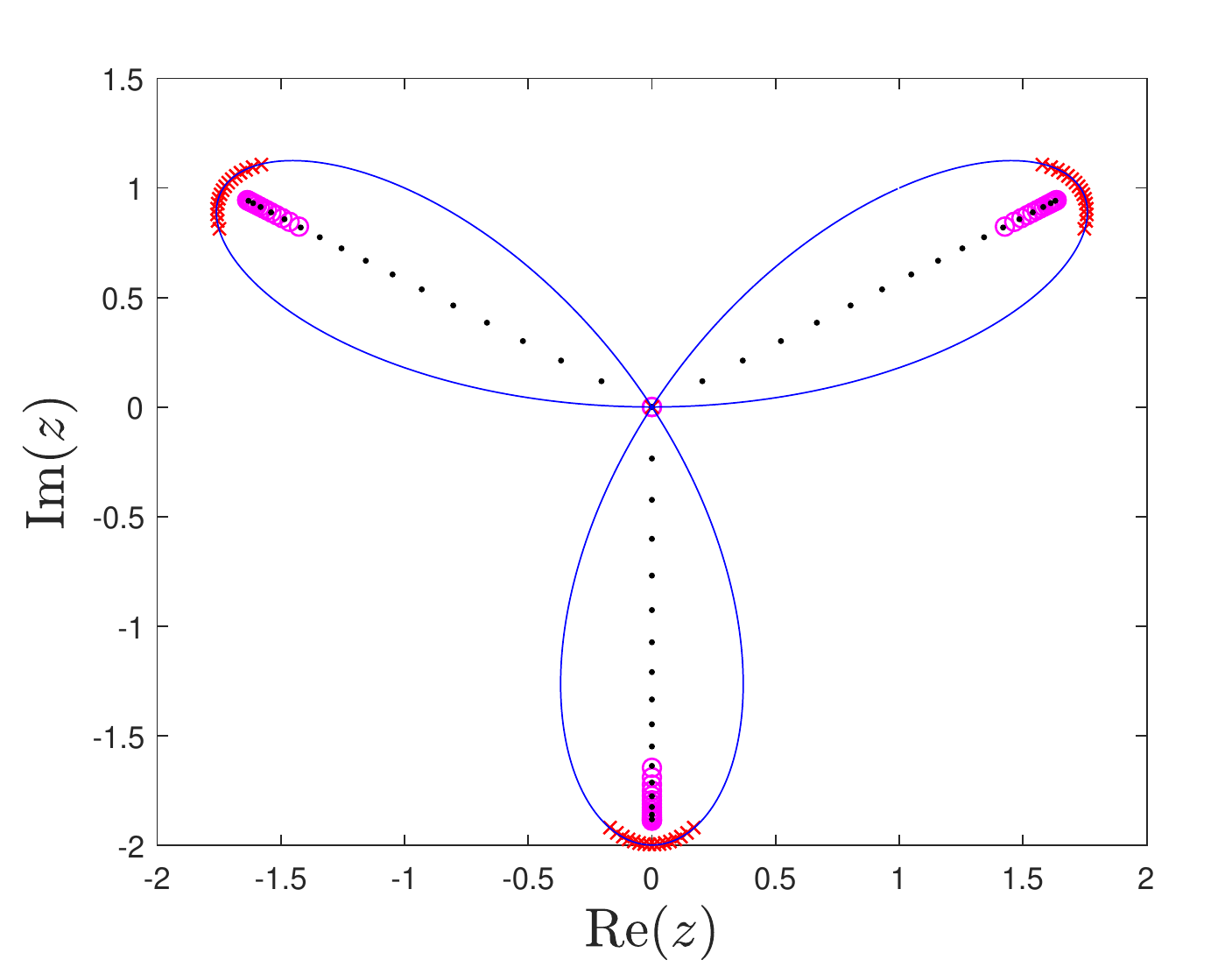}
\caption{Top: Output of IQR and finite section on $T(a)$ and $L(a)$ for $m=50$ and $n=1$ (left), $n=300$ (right). Bottom: Same but for the symbol $\tilde {a}$. In both cases for a given symbol $b$, $\sigma(L(b))$ is given by $\{b(z):z\in\mathbb{T}\}$ (shown) and $\sigma(T(b))$ is given by $\sigma(L(b))\cup\{z\in\mathbb{C}\backslash b(\mathbb{T}):\mathrm{wind}(b,z)\neq 0\}$.}
\label{rose1}
\end{figure}

\begin{figure}%[htp]
\centering
\includegraphics[width=0.495\textwidth,trim={32mm 92mm 35mm 92mm},clip]{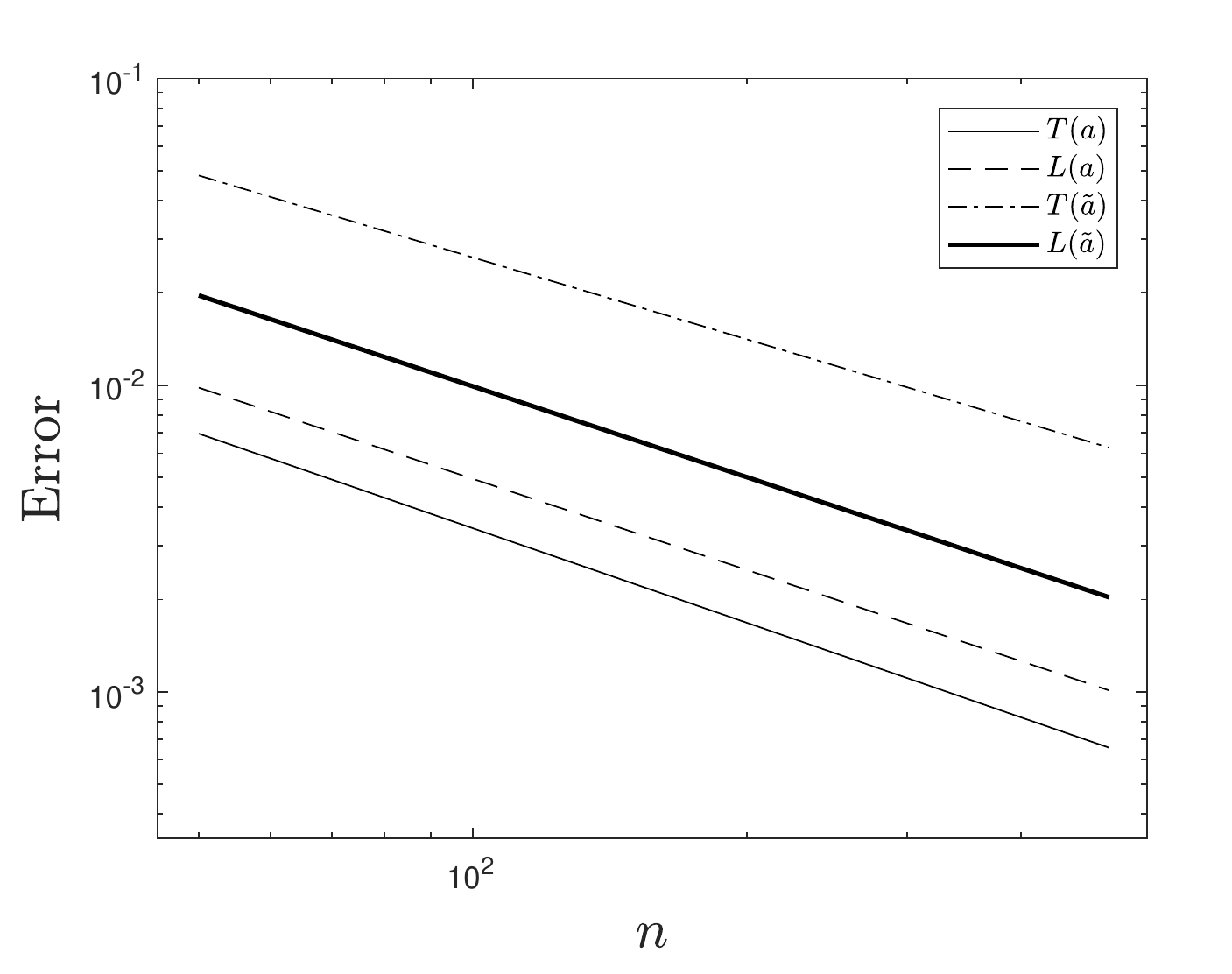}
\includegraphics[width=0.495\textwidth,trim={32mm 92mm 35mm 92mm},clip]{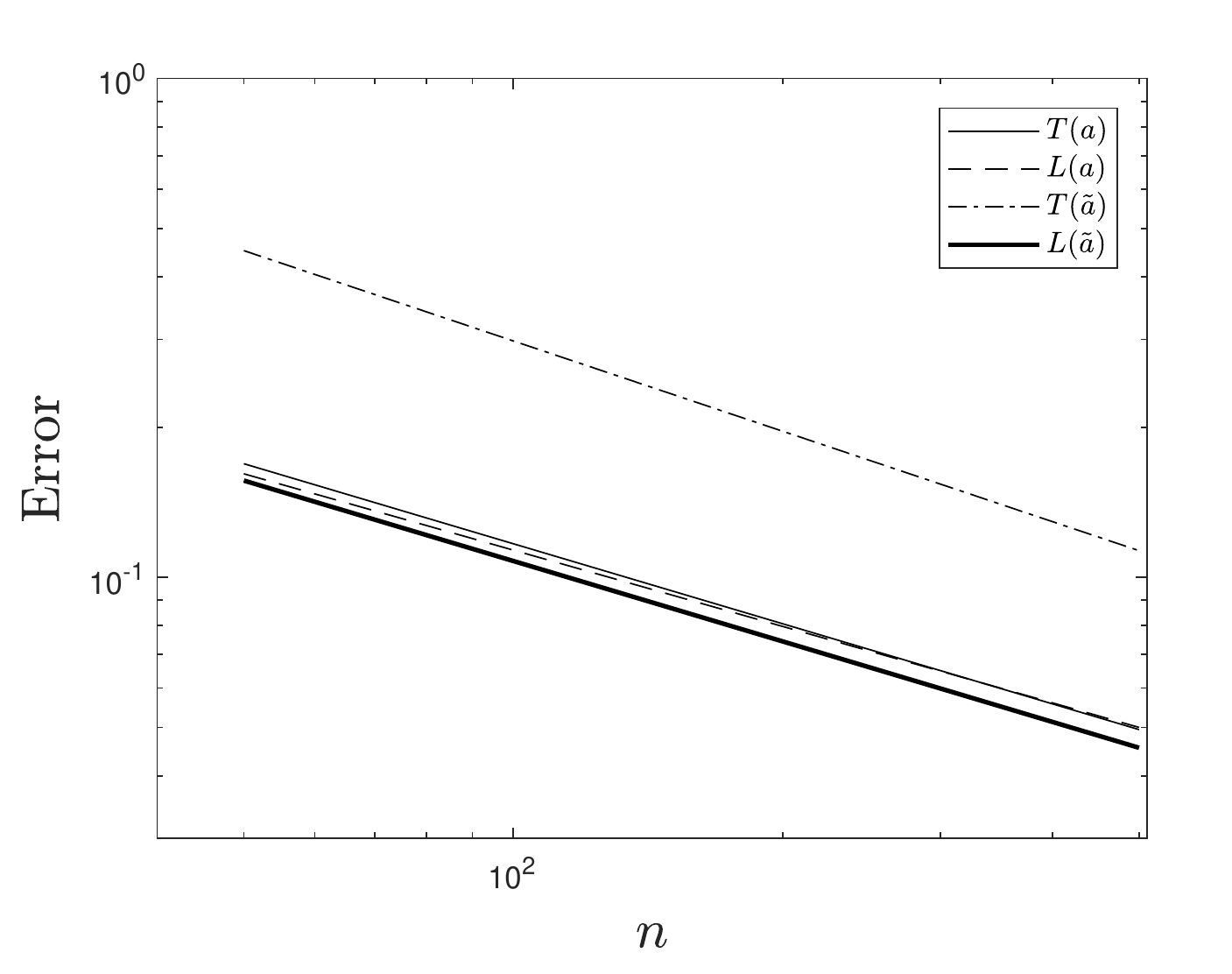}
\caption{Left: Algebraic convergence to block diagonal operators. Right: Algebraic convergence to diagonal operators. In both cases we have plotted the difference in eigenvalues of the first block as we increase $n$.}
\label{rose2}
\end{figure}
\end{example}

\begin{example}[IQR and avoiding spectral pollution]
In this example we consider whether the IQR algorithm may be used as a tool to avoid spectral pollution. Sometimes when considering $\sigma(P_m T|_{P_m\mathcal{H}})$, spectral pollution can be detected by changing $m$ (edge states which correspond to spectral pollution are often unstable but this is not always the case). In general, $\sigma(P_m Q_n^*TQ_n|_{P_m\mathcal{H}})$ can be considered as a generalised version of finite section with a finite number ($n$) of IQR iterates being performed on the infinite dimensional operator before truncation. If $Q_n$ is unitary, then this simply changes the basis before truncation and such a change may reduce (or change) spectral pollution allowing it to be detected. Here we consider
$$
T_3=\left(
\begin{matrix}
 0      & 3     &        &       &        \\
 3      & 0     & 1  &  &   \\
       & 1     & 0       & 3      &   \\
      &      & 3       & 0      &  \ddots \\
  &   &  & \ddots & \ddots\\
\end{matrix}
\right).
$$
The spectrum of $T_3$ is $[-4,-2]\cup[2,4]$. However, if $m$ is odd then $0\in\sigma(P_m T_3|_{P_m\mathcal{H}})$. We shifted the operator by considering $T_3+0.2I$ (and then shifted back for the spectrum). Fig. \ref{jacobi1} shows the Hausdorff distance between $\sigma(P_m Q_n^*(T_3+0.2I)Q_n|_{P_m\mathcal{H}})-0.2I$ and $\sigma(T_3)$ as $n$ varies for different $m$. The spikes in the distance correspond to eigenvalues leaving the interval $[-4,-2]$ and crossing to $[2,4]$ (also shown in Fig. \ref{jacobi1}). The increase in distance as $m$ decreases (for large $n$) is due less of the interval $[-4,-2]$ being approximated. It appears that the IQR algorithm can be an effective tool at detecting spectral pollution - certainly a mixture of varying $m$ and $n$ will be more effective than just varying $m$.

\begin{figure}%[htp]
\centering
\includegraphics[width=0.495\textwidth,trim={32mm 92mm 35mm 92mm},clip]{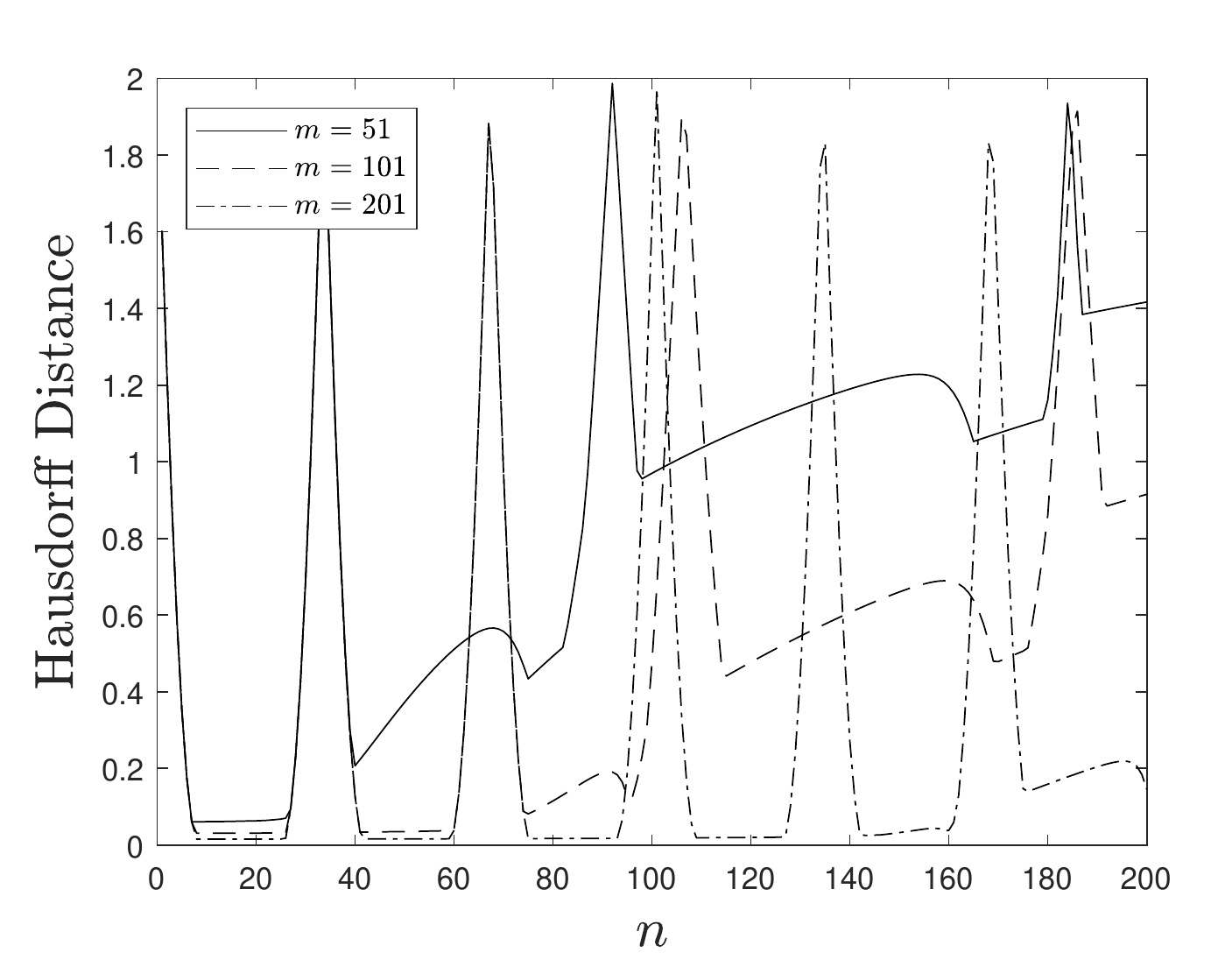}
\includegraphics[width=0.495\textwidth,trim={32mm 92mm 35mm 92mm},clip]{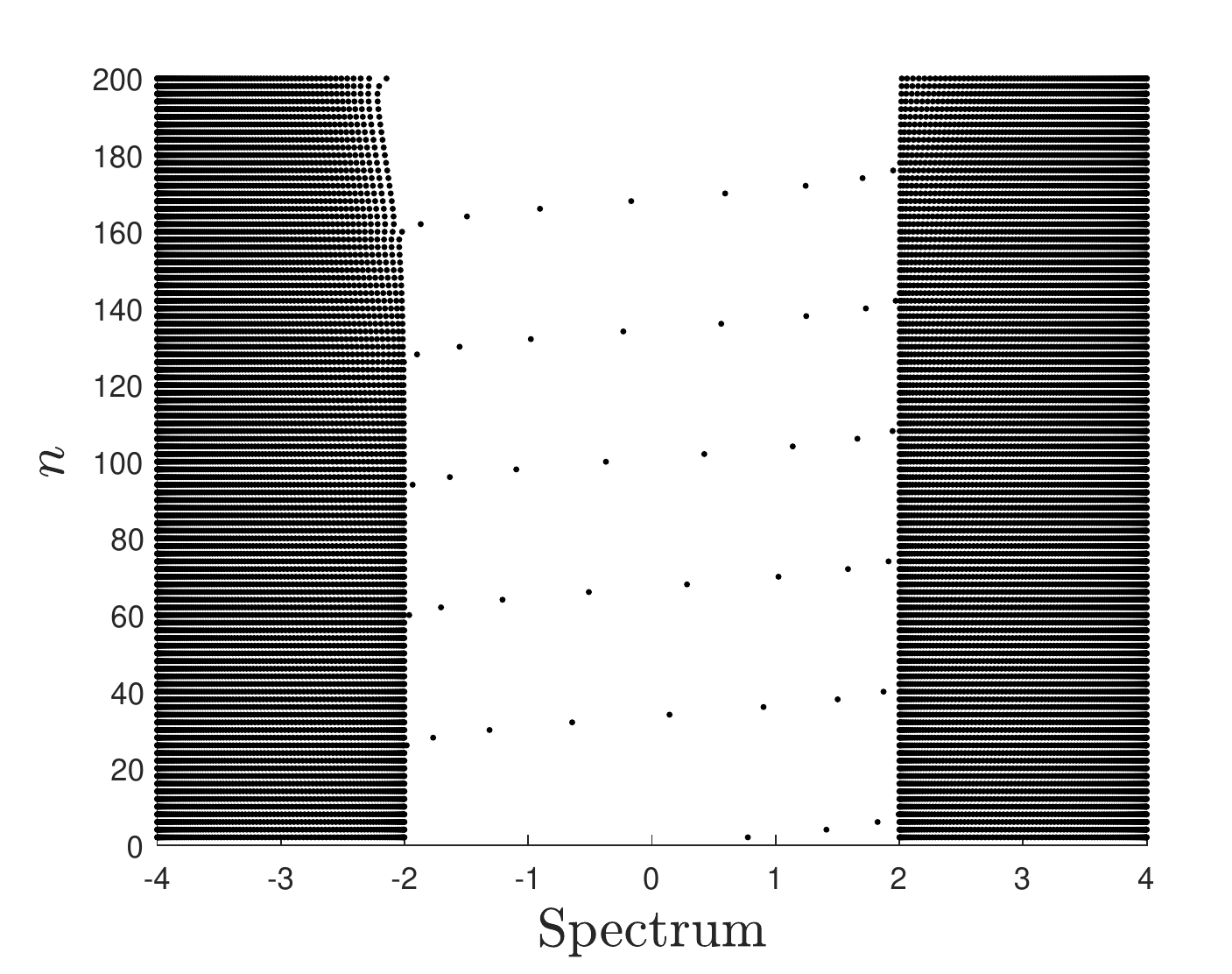}
\caption{Left: $d_{H}(\sigma(P_m Q_n^*(T_3+0.2I)Q_n|_{P_m\mathcal{H}})-0.2I,\sigma(T_3))$ as a function of $n$ for different $m$. Right: $\sigma(P_m Q_n^*(T_3+0.2I)Q_n|_{P_m\mathcal{H}})-0.2I$ as a function of $n$ for $m=201$. Note the crossing of eigenvalues across the spectral gap.}
\label{jacobi1}
\end{figure}

Another example of this is given by the operator $L(a)$ considered previously. For fixed $m$ we found that
$$
\lim_{n\rightarrow\infty}\sup_{z\in\sigma(P_mQ_n^*L(a)Q_n|_{P_m\mathcal{H}})}\mathrm{dist}(z,\sigma(L(a)))=0.
$$
However, for finite section, spectral pollution occurs for all large $m$
$$
\lim_{m\rightarrow\infty}\sup_{z\in\sigma(P_mL(a)|_{P_m\mathcal{H}})}\mathrm{dist}(z,\sigma(L(a)))>0
$$
and the IQR algorithm can only recover the extreme parts of the spectrum
$$
\lim_{n\rightarrow\infty}d_{H}(\sigma(P_mQ_n^*L(a)Q_n|_{P_m\mathcal{H}}),\sigma(L(a)))>0.
$$
Despite this, we found that for small fixed $n>0$ it appears that
$$
\lim_{m\rightarrow\infty}d_{H}(\sigma(P_mQ_n^*L(a)Q_n|_{P_m\mathcal{H}}),\sigma(L(a)))=0.
$$
This is shown in Fig. \ref{rose3} with similar results for $L(\tilde{a})$.

\begin{figure}%[htp]
\centering
\includegraphics[width=0.495\textwidth,trim={32mm 92mm 35mm 92mm},clip]{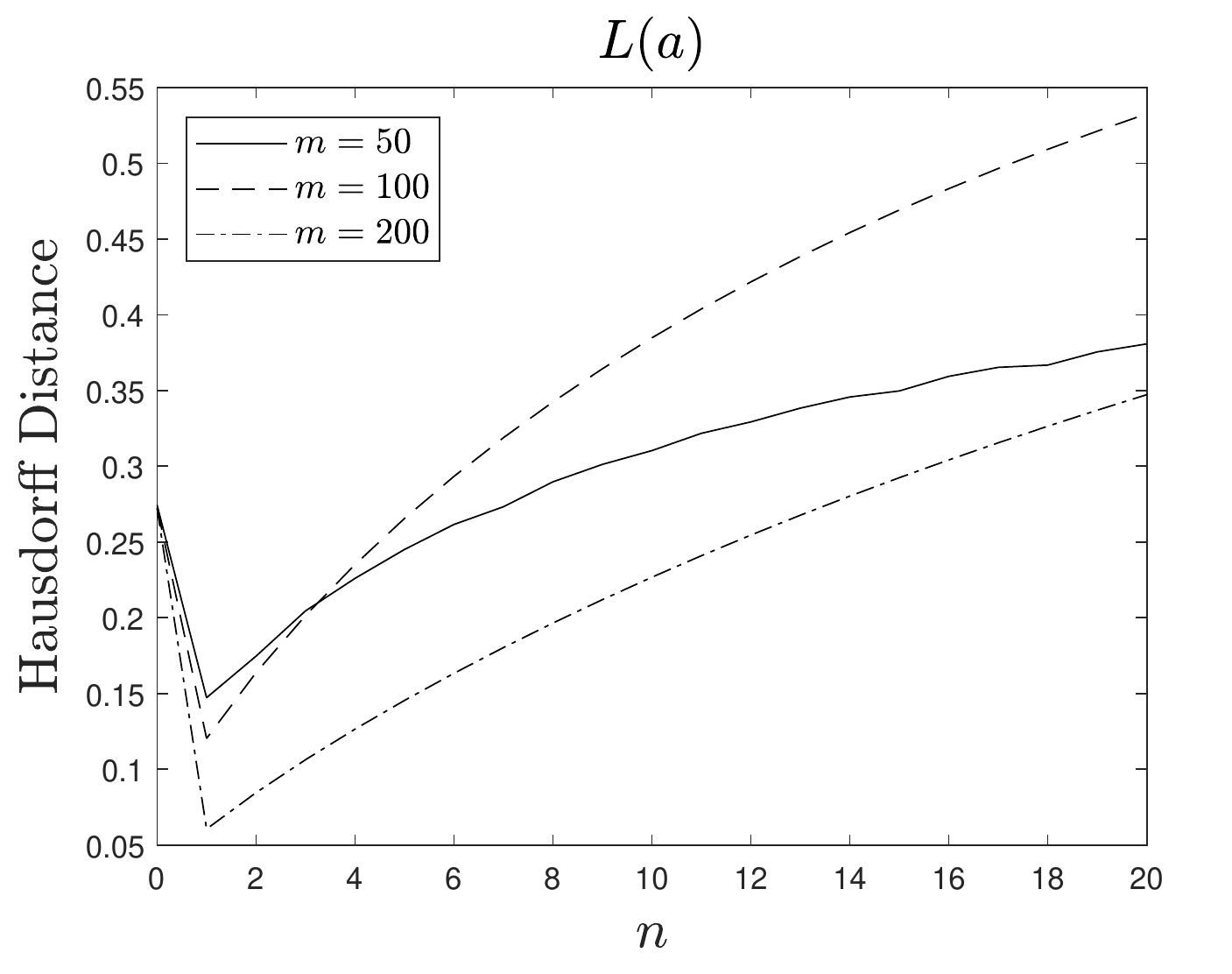}
\includegraphics[width=0.495\textwidth,trim={32mm 92mm 35mm 92mm},clip]{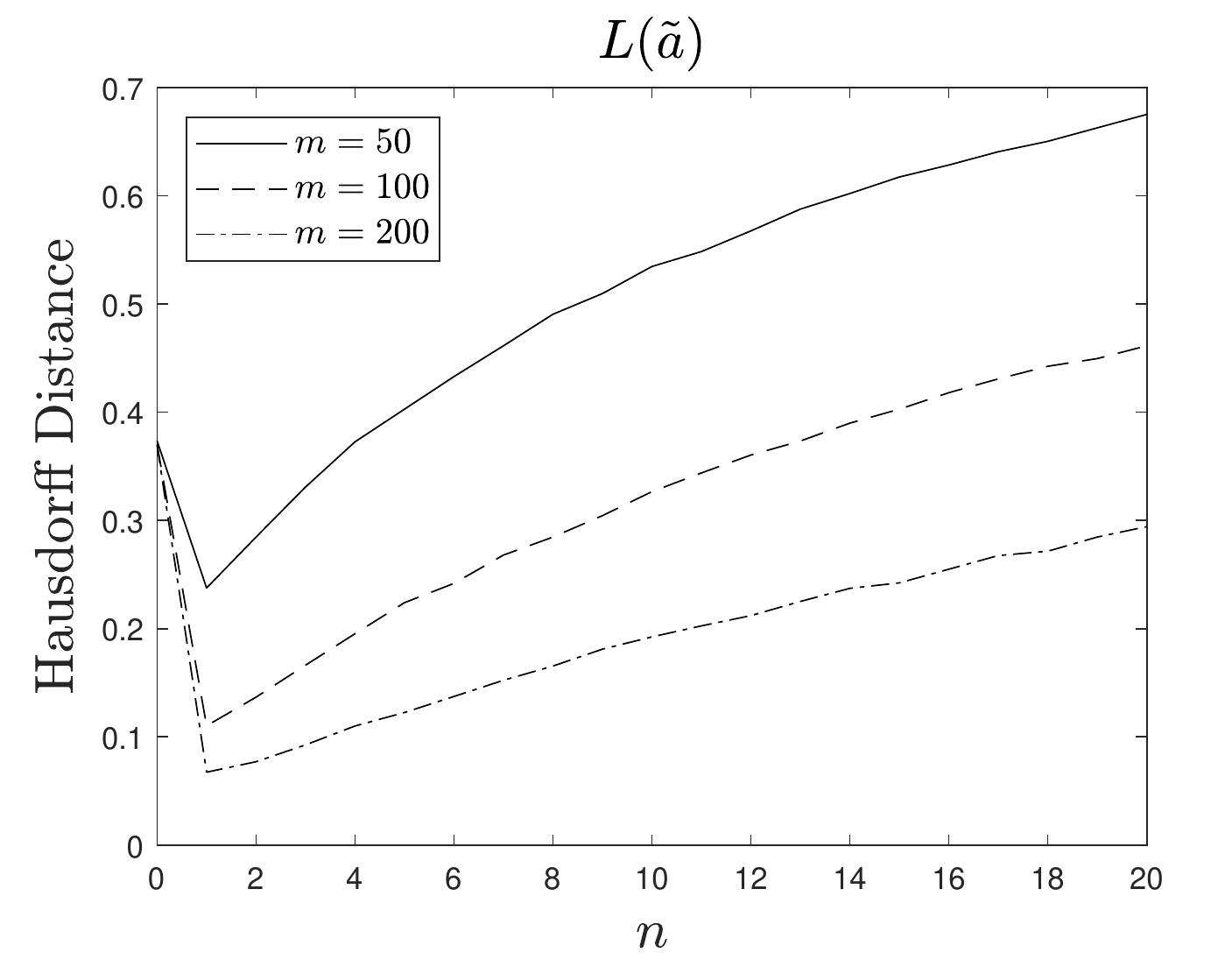}
\caption{Left: $d_{H}(\sigma(P_mQ_n^*L(a)Q_n|_{P_m\mathcal{H}}),\sigma(L(a)))$ as a function of $n$ for different $m$. $d_{H}(\sigma(P_mQ_n^*L(\tilde a)Q_n|_{P_m\mathcal{H}}),\sigma(L(\tilde a)))$ as a function of $n$ for different $m$. }
\label{rose3}
\end{figure}

\end{example}

\subsection{Numerical examples II: non-normal operators}\label{magic}

Although Theorem \ref{new_QR} considers normal operators, Theorems \ref{non_normal_prop} and \ref{non_normal_prop2} suggest the IQR algorithm may also be useful for non-normal operators. Indeed, the results presented here demonstrate that in practice the IQR algorithm can work very well for non-normal problems. If an infinite matrix $T$ has $m$ isolated eigenvalues $\{\lambda_1, \hdots, \lambda_m\}$ (repeated according to multiplicity) outside $r_{\mathrm{ess}}(T)$ (the essential spectral radius), then Theorems \ref{non_normal_prop} and \ref{non_normal_prop2} suggest that the eigenvalues will appear
on the diagonal of 
$
P_mQ_n^*TQ_n|_{P_m\mathcal{H}}
$
as $n \rightarrow \infty$, i.e. 
$$
\sigma(P_mQ_n^*TQ_n|_{P_m\mathcal{H}}) \longrightarrow \{\lambda_1, \hdots, \lambda_m\},
 \qquad \text{as }n\rightarrow \infty.
$$
We will verify this numerically in the next examples. However, we will see that not only do we get convergence to the eigenvalues, but often we also pick up parts of the boundary of the essential spectrum (this was the case when considering $T(a)$ but appeared not to be the case for $T(\tilde a)$). This phenomenon is not accounted for in the previous exposition where normality was crucial for proving Theorem \ref{new_QR}.

\begin{example}[Recovering the extremal part of the spectrum]
Let us return to the infinite matrices $A$ in \eqref{eq:the_A} and $T$ in \eqref{eq:the_T} from Section \ref{s:FS}. We have run the IQR algorithm with $n=1000$ and $n=300$ for $A$ and $T$ respectively, shown in Fig. \ref{infQR1}. We see that if one takes a finite section after running the IQR algorithm, then part of the boundary of the essential spectrum also appears, along with the discrete spectrum $\sigma_d(A)$. Note that the part of the boundary that is captured is the extreme part (points with largest modulus). It seems that after running the IQR algorithm, the spectral information from the largest isolated eigenvalues and the largest approximate point spectrum is ``squeezed up'' to the upper and leftmost portions of the matrix. This is not completely counter-intuitive given (\ref{pwr_ref}) and is what normally happens in finite dimensions. For both examples, we found that the IQR iterates converges to an upper triangular matrix (analogous to the finite dimensional case) in agreement with Theorems \ref{non_normal_prop} and \ref{non_normal_prop2}. The convergence of the upper $1\times 1$ block for $A$ (corresponding to the dominant eigenvalue) and $4\times 4$ non-diagonal block for $T$ are shown in Fig. \ref{non_normal_rate} where we have plotted the difference in norm.

\begin{figure}%[htp]
\centering
\includegraphics[width=0.495\textwidth,trim={32mm 90mm 35mm 95mm},clip]{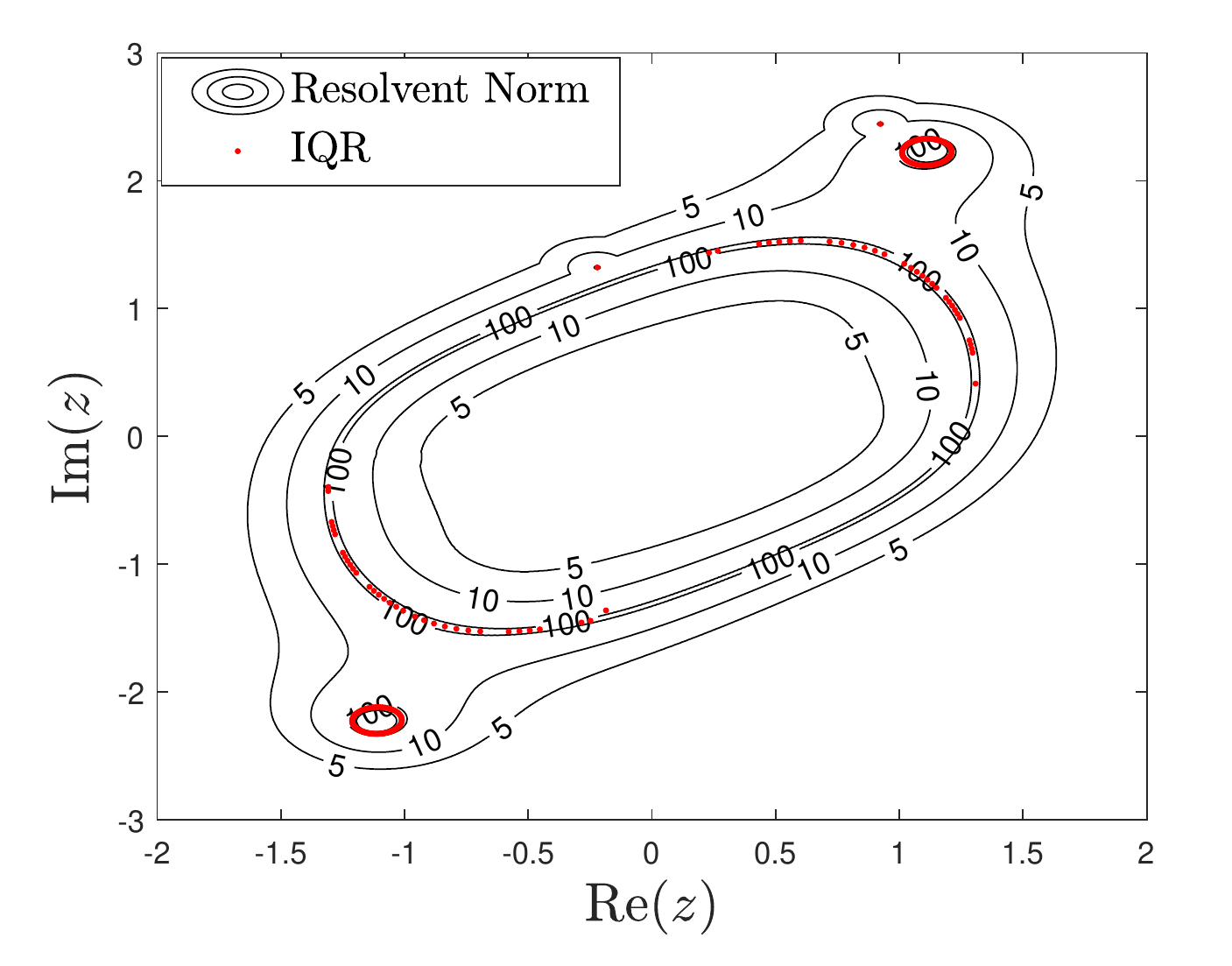}
\includegraphics[width=0.495\textwidth,trim={32mm 90mm 35mm 95mm},clip]{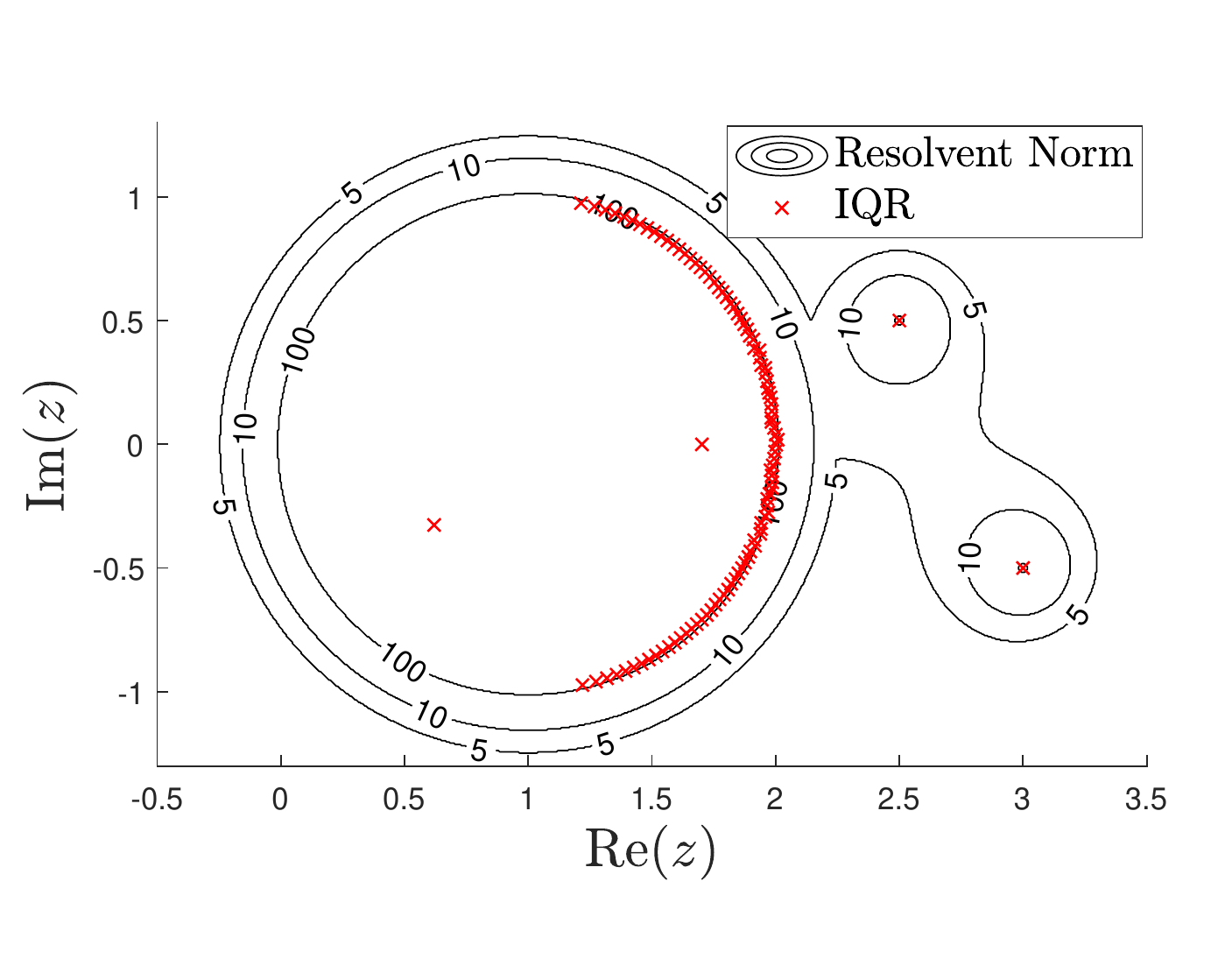}
\caption{Left: Output of the IQR algorithm $\sigma(P_{m}Q_n^*AQ_n|_{P_{m}})$ for $m=300$ and $n=1000$. Right: Output of the IQR algorithm $\sigma(P_{m}Q_n^*TQ_n|_{P_{m}})$ for $m=100$ and $n=300$.}
\label{infQR1}
\end{figure}

We also discuss another drawback of the algorithm, $\Gamma_m$, to compute the pseudospectrum by perturbing the operator $T$. Let $\widetilde{T}$ be the operator obtained from $T$ if we set $\widetilde{T}_{5,4}=5\times 10^7$ (note that this gets rid of the block form). The computation of $\Gamma_m$ involves squaring the operator and hence leads to matrices of norm of order $10^{15}$, making it impossible to compute $\sigma_{\epsilon}(\widetilde{T})$ using double precision for $\epsilon \lesssim 10^{-1}$. However, the IQR algorithm shares the pleasant feature of finite section in allowing a wider range of magnitudes of the matrix entries of the operator. The output for $n=300$, $m=100$ is shown in Fig. \ref{non_normal_rate} as well as convergence of the upper $2\times 2$ block (corresponding to the dominant eigenvalues). Note in this case we can only compute this upper block to an accuracy of about $10^{-8}$ in double precision due to the large perturbed entry. However, this is still much better than pseudospectral techniques. All the errors in Fig. \ref{non_normal_rate} were obtained via comparison with converged matrices computed using quadruple precision.

\begin{figure}%[htp]
\centering
\includegraphics[width=0.495\textwidth,trim={32mm 90mm 35mm 95mm},clip]{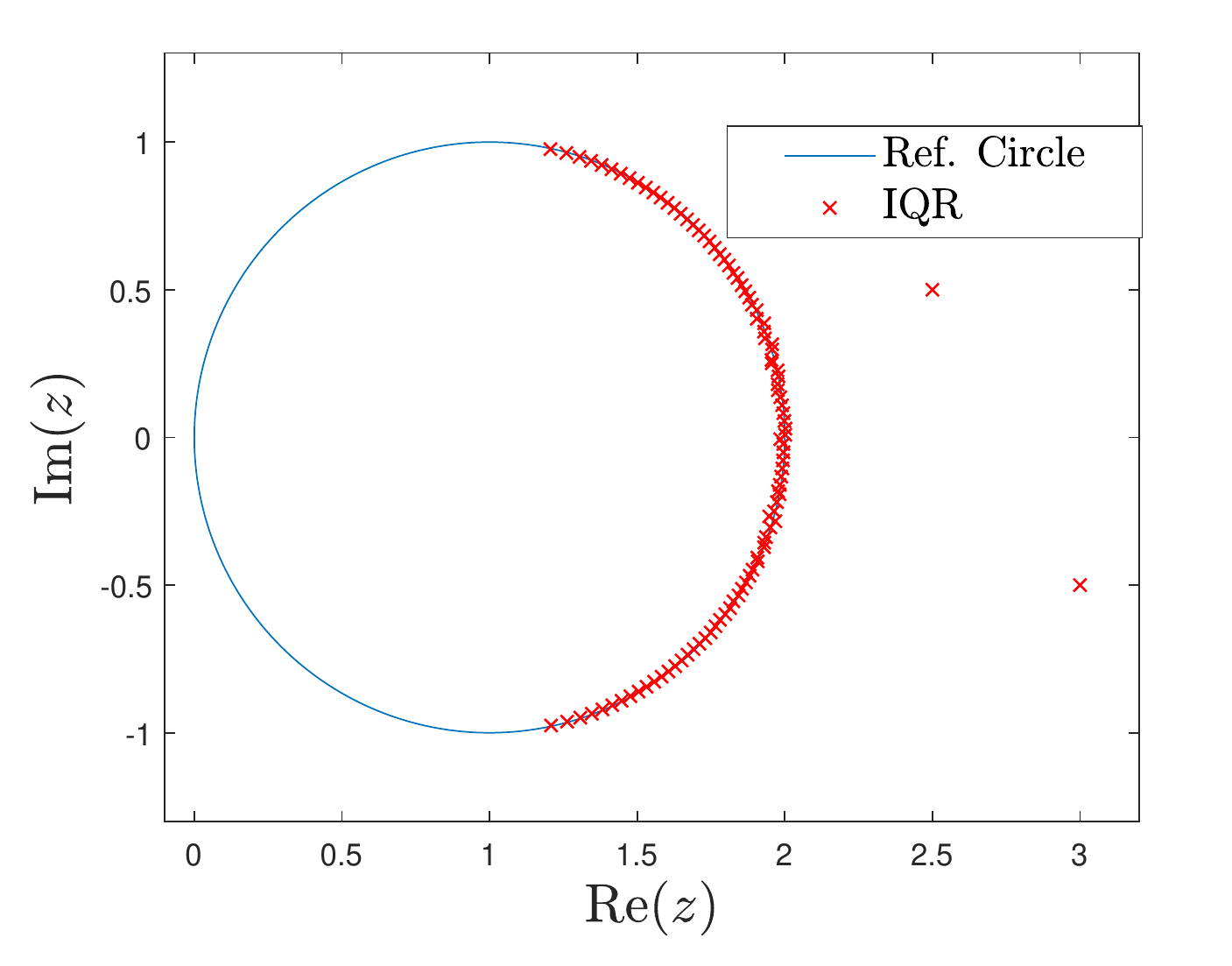}
\includegraphics[width=0.495\textwidth,trim={32mm 90mm 35mm 95mm},clip]{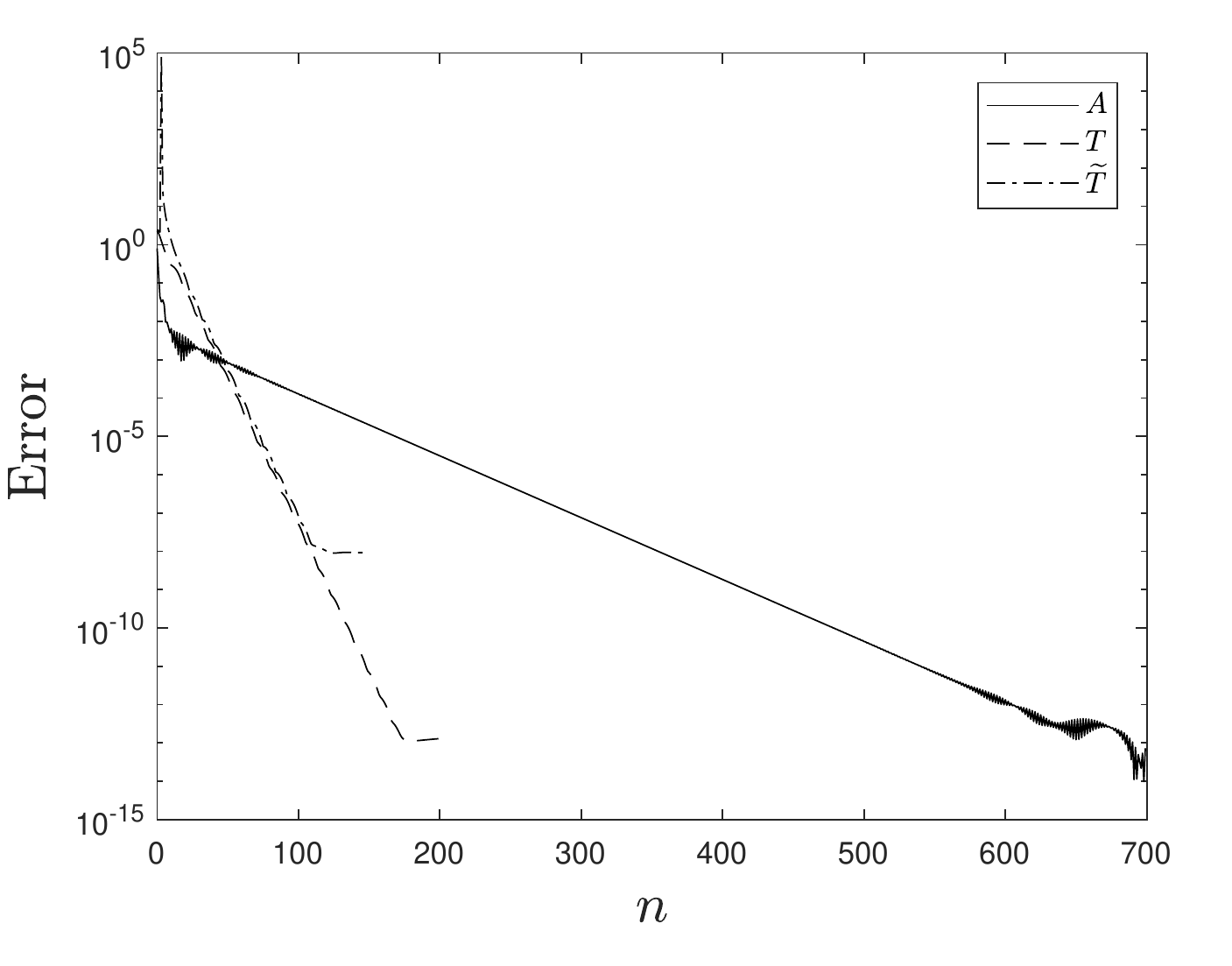}
\caption{Left: Output of the IQR algorithm $\sigma(P_{m}Q_n^*\widetilde{T}Q_n|_{P_{m}})$ for $m=100$ and $n=300$. The reference circle is the boundary of the essential spectrum. Right: Convergence of upper diagonal blocks for operators $A$, $T$ and $\widetilde{T}$.}
\label{non_normal_rate}
\end{figure}
\end{example}

\begin{example}[$PT$-symmetry in quantum mechanics]

Finally, we consider a so called $PT$ symmetric operator (non-normal), demonstrating the same phenomena. A Hamiltonian $H=p^2/2+V(x)$ is said to be $PT$ symmetric if it commutes with the action of the operator $PT$ where $P$ is the parity operator $\hat{x}\rightarrow{}-\hat{x},\hat{p}\rightarrow{}-\hat{p}$ and $T$ the time operator $\hat{p}\rightarrow{}-\hat{p},i\rightarrow{}-i$. Further distinction can be made between exact (unbroken) $PT$ symmetry when $H$ shares common ``eigenfunctions'' with $PT$ and broken $PT$ symmetry when they possess different eigenfunctions. Many $PT$ Hamiltonians possess the remarkable property that their spectra are real for small enough $\mathrm{Im}(V)$ but that the spectrum becomes complex above a certain threshold \cite{bender1998real}. This phase transition from exact to broken $PT$ phase is known as symmetry breaking. There has been a lot of interest in recent years, both theoretically and experimentally, in non-Hermitian $PT$ symmetric Hamiltonians \cite{regensburger2012parity,makris2008beam}.

We consider an operator on $l^2(\mathbb{Z})$ of the form
\begin{equation}
(H_1x)_n = x_{n-1}+x_{n+1}+V_nx_n.
\end{equation}
This commutes with (the discrete version of) $PT$ precisely when the potential has even real part and odd imaginary part. We tested the IQR algorithm on the potential
\begin{equation}
V_n=\begin{cases}
\cos(n)+i\gamma\sin(n), & \mod(n,2)=0 \\
0, & \mod(n,2)=1 \\
\end{cases},
\end{equation}
and found similar results for other potentials. Fig. \ref{infQRpt} shows the same qualitative behaviour as the last example for $\gamma=1,2$ at $m=500,n=3000$. We shifted by $2.2$ and $2.15$ for $\gamma=1,2$ respectively. For comparison we have shown converged resolvent norms. We found that spectral pollution with no IQR iterates was consistent as we varied $m$. However, for a fixed $m$, increasing the number of iterates ($n\rightarrow\infty$) caused $\sigma(P_mQ_n^*H_1Q_n|_{P_m\mathcal{H}})$ to approach the extremal part of the spectrum.

\begin{figure}
\centering
\includegraphics[width=0.48\textwidth,trim={30mm 90mm 40mm 97mm},clip]{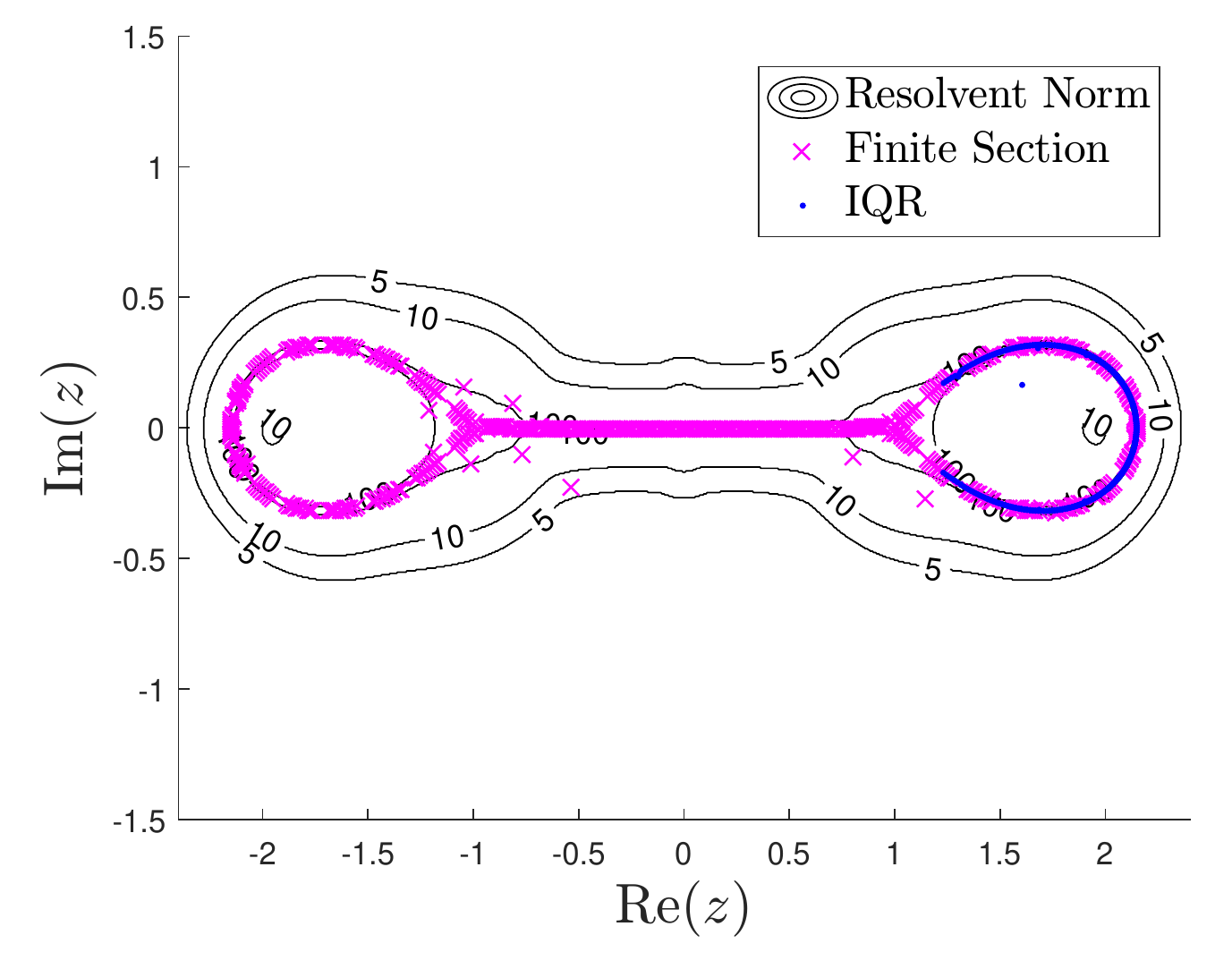} 
\includegraphics[width=0.48\textwidth,trim={30mm 90mm 40mm 97mm},clip]{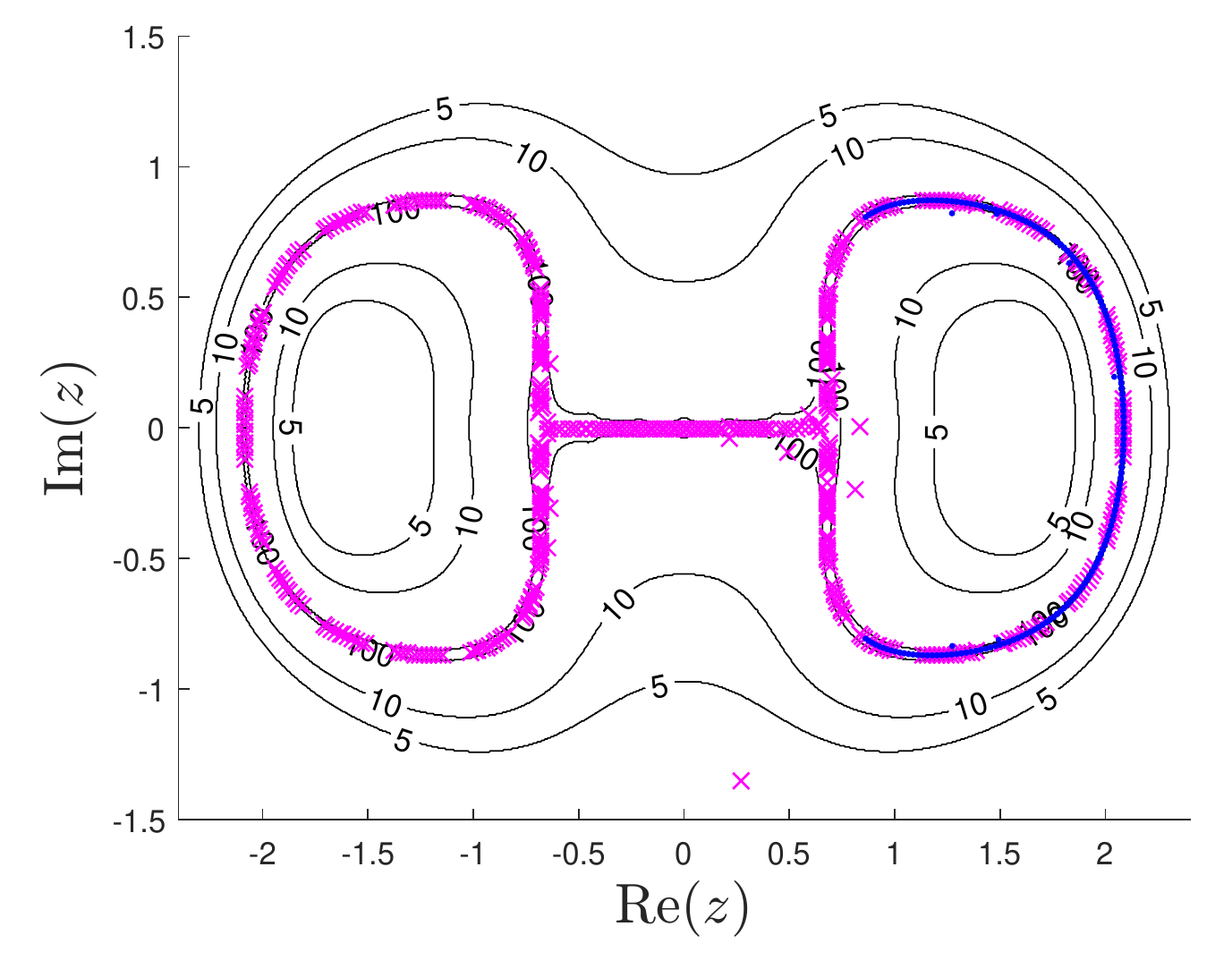}
\caption{The figures show finite sections $\sigma(P_mH_1|_{P_m\mathcal{H}})$ (magenta) and (shifted) $\sigma(P_mQ_n^*H_1Q_n|_{P_m\mathcal{H}})$ IQR iterates (blue) along with converged resolvent norm contours for $\gamma=1$ (left) and $\gamma=2$ (right). Both figures are for $m=500, n=3000$ and show the convergence to the extremal parts of the spectrum.}
\label{infQRpt}
\end{figure}

\end{example}

\subsection{Numerical examples III: random non-Hermitian operators and boundary conditions}\label{applications_Z}

In this final section, we explore examples where the $P_mQ_n^*TQ_nP_m$ naturally give rise to periodic boundary conditions (this was already seen for some examples of Laurent operators in Section \ref{s:num_1}). Both examples discussed here are physically motivated random tridiagonal operators on the lattice $\mathbb{Z}$. One of the key applications of studying such random operators can be found in condensed matter physics. The discrete models below have been used to study conductivity of disordered media, flux lines in superconductors and asymmetric hopping particles. Many such operators are also the discretisation of certain stochastic differential equations. As we will demonstrate, the IQR method can be a powerful way of avoiding spectral pollution caused by unnatural ``open" boundary conditions in forming the finite section $P_mTP_m$. In both of these examples, periodic boundary conditions are natural and we find that taking finite sections after iterating the IQR algorithm captures periodic boundary conditions.

\begin{example}[Hopping sign model in sparse neural networks]

The first example is a non-normal operator with random sub and super-diagonals, first studied by Feinberg and Zee \cite{feinberg1999non,holz2003remarkable,chandler2010eigenvalue}. The usual ``Hopping Sign Model" is defined via
\begin{equation*}
(H_2x)_n=x_{n-1}+b_{n}x_{n+1},
\end{equation*}
with $b_n\in\{\pm 1\}$ (say independent Bernoulli with parameter $p=1/2$). This describes a particle ```hopping'' on $\mathbb{Z}$ and can be mapped into a (complex-valued) random walk. We will consider a slightly different operator described by
\begin{equation}
(H_3x)_n=s_{n-1}^{-}\exp({-g})x_{n-1}+s_{n}^{+}\exp({g})x_{n+1},
\end{equation}
and appearing in \cite{amir2016non} in the context of sparse neural networks. We shall assume that $g$ is real and non-negative and that $s_j^{\pm}$ are i.d.d. random variables with Bernoulli distribution $p$. In other words 
$$\mathbb{P}(s_j^{\pm}=1)=1-\mathbb{P}(s_j^{\pm}=-1)=p.$$
We will only consider $g=1/10$ and $p=1/2$, but will vary $p$ in an effort to compute the spectrum of $H_3$ which only depends on the support of the distribution of the $s^{\pm{}}_j$'s. It is easy to prove that the spectrum (and pseudospectrum) of $H_3$ is almost surely constant and that there is no inessential spectrum. Furthermore, one can show that $\sigma(H_3)$ is contained in the annulus $\{z\in\mathbb{C}:2\sinh(g)\leq\left|z\right|\leq2\cosh(g)\}$.

Finite section calculations associated with this operator have some interesting properties and are extensively studied in \cite{amir2016non}. If one projects using the standard basis of $l^2(\mathbb{Z})$ then one obtains matrices of the form

\begin{equation*}
M_n^1 =
\left(
\begin{matrix}
0    &s_{-n+1}^{-}\exp({-g})&  &   \\
s_{-n+1}^{+}\exp({g})& 0 & \ddots &\\
    & \ddots & \ddots & s_{n-1}^{-}\exp({-g}) \\
    &  & s_{n-1}^{+}\exp({g}) & 0  \\

\end{matrix}
\right).
\end{equation*}
If we use open boundary conditions (i.e. we simply project onto the space spanned by $\{e_{-n},...,e_n\}$) then one can ``gauge" away $g$ by a similarity transformation, leading to

\begin{equation*}
M'_n =
\left(
\begin{matrix}
0    &s_{-n+1}^{-}&  &   \\
s_{-n+1}^{+}& 0 & \ddots &\\
    & \ddots & \ddots & s_{n-1}^{-} \\
    &  & s_{n-1}^{+} & 0  \\

\end{matrix}
\right).
\end{equation*}
On the other hand, the use of periodic boundary conditions leads to the matrix

\begin{equation*}
M_n^2 =
\left(
\begin{matrix}
0    &s_{-n+1}^{-}\exp({-g})&  &s_{n}^{+}\exp({g})\\
s_{-n+1}^{+}\exp({g})& 0 & \ddots &\\
    & \ddots & \ddots & s_{n-1}^{-}\exp({-g}) \\
s_{n}^{-}\exp({-g})&  & s_{n-1}^{+}\exp({g}) & 0  \\

\end{matrix}
\right),
\end{equation*}
which does not suffer from this setback.% In fact, periodic boundary conditions can be shown to be in some sense natural \cite{chandler2010eigenvalue}.

\begin{figure}
\centering
\vspace{-13mm}
%\hspace{-12mm}
\includegraphics[height=110mm]{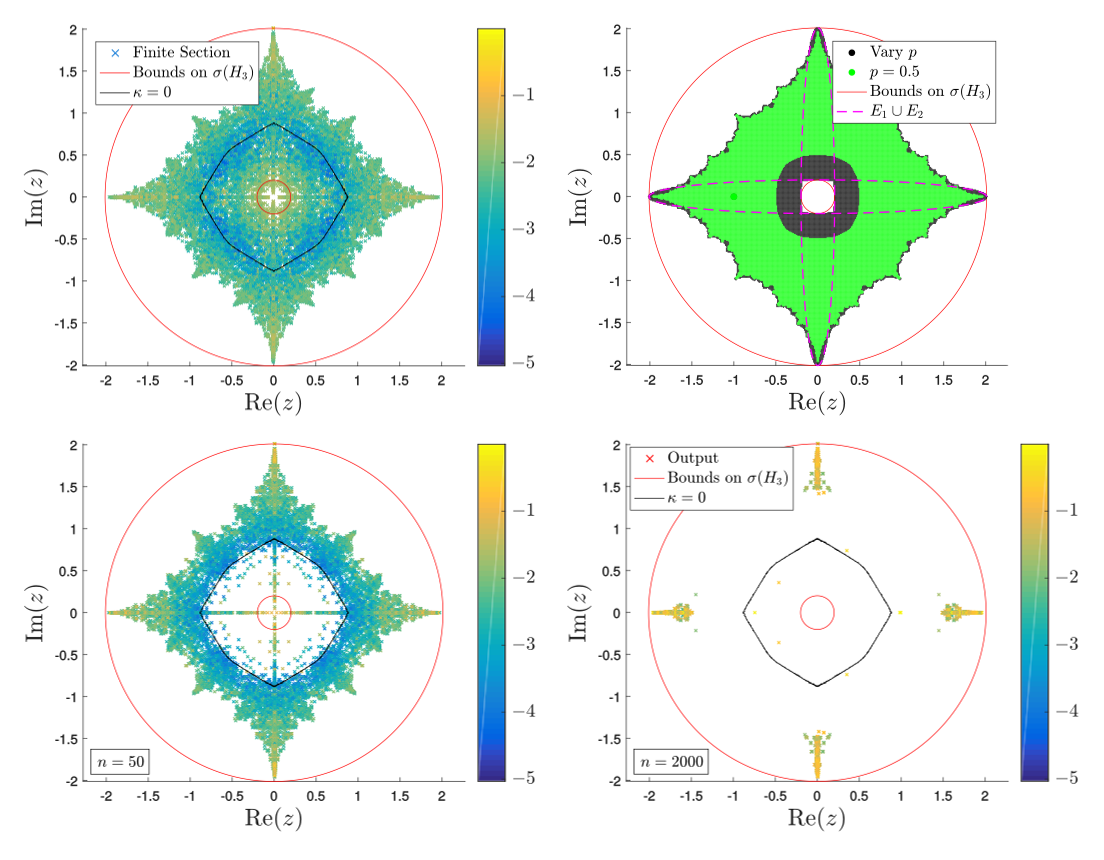}
\vspace{-4mm}
\caption{Top: Output of finite section over a random sample of $200$ matrices of size $200$ (left) and the estimates using pseudospectral techniques (right). Bottom: The output of IQR over $200$ samples computing $\sigma(P_mQ_n^*H_3Q_n|_{P_m\mathcal{H}})$ for $m=200$ and $n=50$ (left), $n=2000$ (right). Note that a few iterates seems to agree with periodic boundary conditions and then increasing the number of iterates leads to convergence to the extremal parts of the essential spectrum.}
\label{zee_fs}
\end{figure}

In \cite{amir2016non} this phenomena was studied via localisation of the eigenvalues of $M_n^{2}$, in particular using the Lyapunov exponent $\kappa(z)$ which is equal to the inverse of the localisation length. An eigenfunction $\psi$ with eigenvalue $z$ localised around $x_0$ behaves approximately as 
$$
\left|\psi(x)\right|\sim \exp({-\kappa(z)\left| x-x_0\right|}).
$$
If one defines recursively
$$
y_{n+1}(z)=\exp(g)\frac{\psi_{n+2}}{\psi_{n+1}}=-(s^{-}_{n-1}/s_n^+)/y_n(z)+z/s_n^+
$$
then (in the limit of large system sizes)
$$
\kappa(z;g)=\lim_{N\rightarrow\infty}\frac{1}{2N+1}\sum_{j=-N}^{N}\big(\log\left|y_j(z)\right|-g\big).
$$
This is known as the transfer matrix approach. %\cite{derrida2000lyapunov}.
 For fixed $z$, as we increase $g$, $\kappa(z;g)$ becomes negative. The heuristic is that a hole opens up in the spectrum corresponding to a negative Lyapunov exponent. Eigenvalues of $M_n^2$ inside the hole are swept up and become delocalised moving to the rim of the hole, whereas those outside remain largely undisturbed. Eigenvalues of $M_n^1$ inside the negative $\kappa$ zone correspond to edge states due to the finite system size approximation.

Fig. \ref{zee_fs} shows the output of a sample of $200$ finite sections with open boundary conditions and matrix size $200$. We have also shown the annular region that bounds the spectrum, as well as the contour $\kappa=0$. In order to calculate $\kappa$, we calculated the above sum on a grid with large $N$ to ensure convergence. The colour bar corresponds to the inverse participation ratio (log scale) of normalised eigenfunctions defined by
$$
1/P\equiv\frac{\sum_j\left|\psi_i\right|^4}{\sum_j\left|\psi_i\right|^2}.
$$
Note that this has a maximum value of $1$ (localised) and a minimum value of $1/N$ (delocalised), $N$ being the size of the matrix. Open boundary conditions produces spectral pollution in the hole with localised eigenfunctions and the contour $\kappa=0$ corresponds to the delocalised region. In order to compare to the spectrum of the infinite operator on $l^2(\mathbb{Z})$ we have plotted $\sigma_{\epsilon}(H_3)$, for $\epsilon=10^{-2}$, calculated using matrix sizes of order $10^5$. We note that the spectrum is independent of $p\in(0,1)$ so we have also shown the union of these estimates over $p=\{k/100\}_{k=1}^{99}$. Although the algorithm used to compute the pseudospectrum is guaranteed to converge to $\sigma_{\epsilon}(H_3)$, there are regions in the complex plane where this convergence is very slow. Taking unions over $p$ is simply a way to speed up this convergence. We found upon taking $\epsilon$ smaller that the spectrum appeared to have a fractal like nature. It also appears that the hole in the spectrum corresponds to the boundary of two ellipses. It is easy to prove that the ellipse
$$
E_1=\{\exp({g+i\theta})+\exp({-g-i\theta}):\theta\in[0,2\pi)\}
$$
is contained in $\sigma(H_3)$ and that the spectrum (and pseudospectrum) of $H_3$ has fourfold rotational symmetry. Denoting the rotation of $E_1$ by $\pi/4$ as $E_2$ we have shown $E_1\cup E_2$ in the figure.

Fig. \ref{zee_fs} also shows the effect of IQR iterations over random samples of size $200$ for $m=200$ and $n=50$ and $2000$. Remarkably, as we increase $n$, a few iterations is enough to capture periodic boundary conditions and sweep away the localised edge states. We have also shown the inverse participation ratio which, although now is defined with respect to a new basis, still gives an indication of how ``diagonal" the matrix $P_mQ_n^*H_3Q_n|_{P_m\mathcal{H}}$ is. If we increase $n$ further, the output approaches the edge of the spectrum with eigenvectors becoming more localised (in the new basis). We found exactly the same phenomena to occur if we shifted the operator $H_3$ with convergence to the corresponding extremal part of the essential spectrum.

\end{example}

\begin{example}[NSA Anderson model in superconductors]
Finally, we consider a non-normal operator with no inessential spectrum where the IQR algorithm does not seem to converge to the boundary of the essential spectrum, but rather to a curve associated with periodic boundary conditions in the large system size limit.

Over the past twenty years there has been considerable interest in non self-adjoint random operators, sparked by Hatano and Nelson studying a non self-adjoint Anderson model in the context of vortex pinning in type-II superconductors \cite{hatano1996localization}. Their model showed that an imaginary gauge field in a disordered one-dimensional lattice can induce a delocalisation transition. The operator in $\mathcal{B}(l^2(\mathbb{Z}))$ can be written as
\begin{equation}
(H_4x)_n=\exp({-g})x_{n-1}+\exp({g})x_{n+1}+V_nx_n
\label{NSAand}
\end{equation}
where $g>0$ and $V$ is a random potential. This operator also has applications in population biology \cite{nelson1998non} and the self-adjoint version of this model is widely studied for the phenomenon of Anderson localisation (absence of diffusion of waves) \cite{anderson1958absence,billy2008direct}. In the non self-adjoint case, complex values of the spectrum indicate delocalisation. Note that we now have randomness on the diagonal with fixed coupling coefficients $\exp({\pm g})$.

Standard finite section produces real eigenvalues since the matrix $P_mH_4|_{P_m\mathcal{H}}$ is similar to a real symmetric matrix. However, truncating the operator and adopting periodic boundary conditions gives rise to the famous ``bubble and wings''. If $V=0$ then the spectrum is an ellipse $E=\{\exp({g+i\theta})+\exp({-g-i\theta}):\theta\in[0,2\pi)\}$, but as we increase the randomness wings appear on the real axis. For a study of this phenomenon and the described phase transition we refer the reader to \cite{feinberg1999non}. Goldsheid and Khoruzhenko have studied the convergence of the spectral measure in the periodic case as $N\rightarrow\infty$ in \cite{goldsheid1998distribution}, $N$ being the number of sites. In general, the support of these measures as $N\rightarrow\infty$ can be very different to the spectrum of the operator on $l^2(\mathbb{Z})$ given by (\ref{NSAand}), highlighting the difficulty in computing the spectrum.

We consider the case $g=1/2$ with $V_n$ i.i.d. Bernoulli random variables taking values in $\{\pm 1\}$ with equal probability $p=1/2$. Again, there is no inessential spectrum and the spectrum/pseudospectrum is constant almost surely, depending only on the support of the distribution of the $V_n$. The following inclusion is also known, which gives boundaries to the spectrum: 
$$
\sigma(H_4)\subset(\overline{\mathrm{conv}}(E)+[-1,1])\cap (E+B_{1}),
$$
where $\overline{\mathrm{conv}}(E)$ its closed convex hull of $E$ and $B_1$ denotes the closed unit disk. The choice of $g$ ensures the spectrum has a hole in it. One may calculate the Lyapunov exponent, either by the transfer matrix approach or by calculating a potential related to the density of states. %It is then possible to prove \cite{goldsheid2000eigenvalue} that
 The limiting distribution of the eigenvalues of finite section with periodic boundary conditions is given by the complex curve 
$$
\{z\in\mathbb{C}\backslash\mathbb{R}:\kappa(z)=0\}\cup\{x\in\mathrm{supp}(dN):\kappa(x+i0)>0\}.
$$

The output of the IQR algorithm for $m=30$ and $n=15$ and $n=300$ over $200$ random samples are shown in Fig. \ref{anderson_fs}. Note that if we took $n=0$, the spectrum would be real in stark contrast to Fig. \ref{anderson_fs}. Taking a small number of IQR iterates approximates the bubble and wings with a few remaining real eigenvalues. However, upon increasing $n$, the output does not seem to converge to the extremal parts of the spectrum, but seems to remain stuck on the limit curve with the operator $P_mQ^*_nH_4Q_n|_{P_m\mathcal{H}}$. Shifting by $+4iI$ caused the output to recover the top part of the limit curve.

\begin{figure}
\centering
\includegraphics[width=0.49\textwidth,trim={35mm 90mm 40mm 97mm},clip]{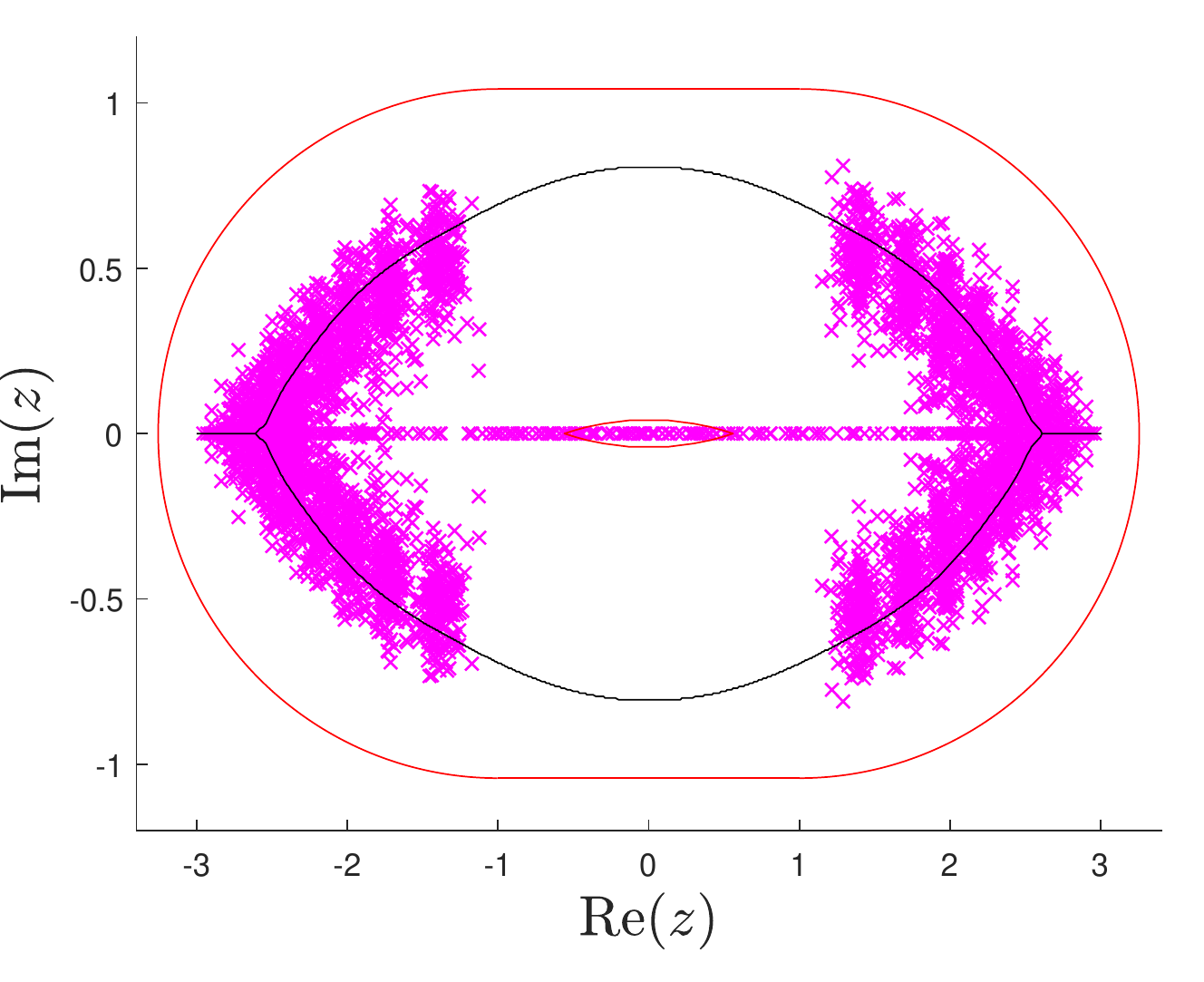}
\includegraphics[width=0.49\textwidth,trim={35mm 90mm 40mm 97mm},clip]{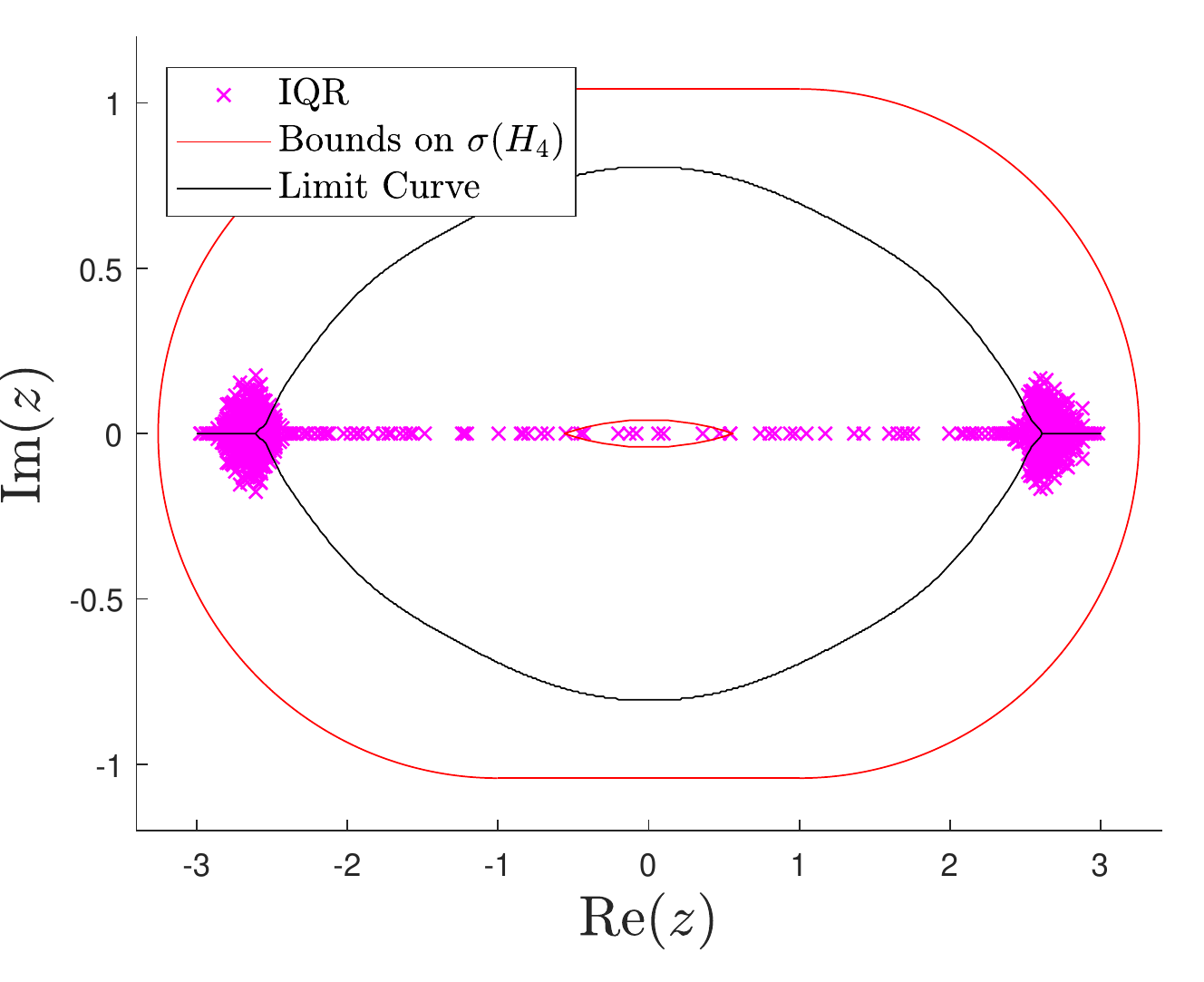}

\caption{The output of IQR over $200$ samples computing $\sigma(P_mQ_n^*H_4Q_n|_{P_m\mathcal{H}})$ for $m=30$ and $n=15$ (left), $n=300$ (right). Note that we appear to recover the periodic limit curve and increasing the number of iterates converges to the extremal parts. Applying shifts allowed us to recover the extremal parts of the limit curves.}
\label{anderson_fs}
\end{figure}

\end{example}
\begin{remark}
For any operator $T$ that has $Q_n$ unitary, the essential spectrum and spectrum of $Q_n^*TQ_n$ is equal to that of $T$. As the above two examples suggest, taking a small value of $n$ could be used as a method of testing eigenvalues of finite section methods that correspond to finite system size effects, such as open boundary conditions. This could be used in quasi-periodic systems or systems with very few symmetries, where there is no obvious choice of appropriate boundary conditions. However, detecting isolated eigenvalues of finite multiplicity within the convex hull of the essential spectrum still remains a challenge.
\end{remark}

\section{Concluding remarks and open problems}
\label{conc_s}

This paper discussed the generalisation of the famous QR algorithm to infinite dimensions. It was shown that for a large class of operators, encompassing many in scientific applications, the iterates of the IQR algorithm can be computed efficiently on a computer. For matrices with finitely many entries in each column, the computation collapses to a finite one. In general, for an invertible operator we can compute the iterates to any given accuracy in finite time. Furthermore, it was proven that for normal operators, the algorithm converges to the discrete spectrum outside the convex hull of the essential spectrum, with the rate of convergence generalising the well known result in finite dimensions. These were extended to more general invariant subspaces and non-normal operators in Theorems \ref{non_normal_prop} and \ref{non_normal_prop2}. Unfortunately the IQR algorithm cannot in general be sped up with the use of shift strategies, which considerably speed up the finite dimensional algorithm \cite{parlett1998symmetric}. This is due to two reasons. The first is simply that there is no final column of an infinite matrix, hence the usual link with inverse iteration cannot be made. Second, it is also possible for part of the spectrum to be lost in the limit (see Example \ref{lose_spec_exam}) and below the essential spectra radius there is no guarantee of convergence. %A possible strategy to overcome this is to consider the infinite dimensional QL algorithm

Despite these inherent drawbacks of the infinite dimensional setting, we showed how the IQR algorithm can be used to gain new classification results and convergent algorithms in the SCI hierarchy. In particular, we showed how to compute eigenvalues and eigenspaces outside the essential spectrum with error control for normal operators. This was extended to dominant invariant subspaces for general (possibly non-normal) operators as well as the spectrum of a large class of operators that includes compact normal operators with eigenvalues of distinct magnitude. These results present the first such algorithms that tackle these problems with error control.

Finally, we demonstrated that the IQR algorithm can be implemented both in theory and in practice. We demonstrated the convergence theorems in Section \ref{conv_main} as well as some examples (normal and non-normal) where the extremal parts of the essential spectrum also appear to be recovered. Based on this, we conjecture that there may be a large class of operators for which the IQR algorithm converges to the extreme parts of the essential spectrum.\footnote{This is false in general as is easily seen by considering the shift operator.} In particular, we conjecture that this holds for normal operators if the set of extremal points of the essential spectrum has size one. However, an example was given where convergence to the essential spectrum was only algebraic $O(n^{-\alpha})$ as $n\rightarrow\infty$ as opposed to the linear convergence rate $O(r^n)$ to the discrete spectrum/eigenvalues. It was also demonstrated that the truncations of the IQR algorithm have spectra agreeing with periodic boundary conditions for a range of operators in the class of ``pseudoergodic'' NSA random operators. %Further study will aim to explore the convergence properties in the non-normal case and aim to prove the conjecture on recovering the extreme part of the essential spectrum.
In some cases, the algorithm performed much better than standard finite section methods. We should stress that, as in the case of the finite section method, for fixed $n$ and $m\rightarrow\infty$, the output $\sigma(P_mQ_n^*TQ_n|_{P_{m}})$ will in general still suffer from the spectral pollution phenomenon and in some cases not recover the full spectrum. This was apparent in the numerical examples and is likely to hold true for many operators even when taking a mixture of double limits $m,n\rightarrow\infty$. However, examples of Laurent operators were given where it appears $\sigma(P_mQ_n^*TQ_n|_{P_{m}})$ converges to the spectrum as $m\rightarrow\infty$ for fixed $n>0$ but not $n=0$. We hope that the algorithm's potential use in sifting out spectral pollution/complying with appropriate boundary conditions via a canonical unitary transformation can also be exploited. %Future work will also aim at exploring the behaviour of the algorithm when taking a mixture of both $m,n\rightarrow\infty$.

Based on our findings, we end with a list of open problems for further study on the theoretical properties of the IQR algorithm:
\begin{itemize}
	\item Which conditions are needed on a possibly non-normal operator in order for the IQR algorithm to pick up the extreme points of the essential spectrum?
	\item Is the convergence rate to non-isolated points of the spectrum algebraic? 
	\item For operators which do not have a trivial QR decomposition, is there a way of choosing $n=n(m)$ such that $\sigma(P_{m}Q_{n(m)}^*TQ_{n(m)}|_{P_{{m}}})$ converges to the spectrum as $m\rightarrow\infty$? If not, then for which classes of operators does such a choice exist?
	\item Is there a link between the IQR algorithm and the finite section method with periodic boundary conditions for the class of pseudoergodic operators?
	\item Are there other cases where the IQR algorithm alleviates the need to provide natural boundary conditions when applying the finite section method?
	\item Extending the IQR algorithm to unbounded operators. Can the IQR algorithm also be extended to a continuous version for differential operators?
\end{itemize}

\section*{Acknowledgments} MJC acknowledges support from the UK Engineering and Physical Sciences Research Council (EPSRC) grant EP/L016516/1. ACH acknowledges support from a Royal Society University Research Fellowship as well as EPSRC grant EP/L003457/1.We would also like to thank the referees whose comments and suggestions led to the improvement of the manuscript.
\appendix

\section{Appendix}

\subsection{Example codes}

Here we show example code for the IQR algorithm in the case that the matrix has $k$ subdiagonals. The code can easily be adapted for the more general case considered in Section \ref{Quasi_banded_Subdiagonals}.

\begin{algorithm}\label{alg}
\footnotesize
\begin{verbatim}


% The Infinite_QR(A,n,k,m) takes a section P_{nk+m}AP_{nk+m} 
% of an infinite matrix A with k subdiagonals, performs n iterations 
% of the infinite dimensional QR algorithm and returns
% J = P_mQ_nAQ*_nP_m.   

function J = Infinite_QR(A,n,k,m)
d = size(A,2);
 for j=1:n
   A = Inf_QR(A,d-j*k,k);   % The output in each loop is actually
 end                        % U_(d-j*k)...U_1A_(j-1)U_1...U_(d-j*k)  
J = A(1:m,1:m);             % if A_j is the j-th term in the QR iteration. 
\end{verbatim}
\end{algorithm}

\begin{algorithm}
\footnotesize
\begin{verbatim}


% Inf_QR(A,n,k) takes a matrix A with k subdiagonals and performs 
% multiplication by n Householder transformation from the left and 
% right, i.e. B = U_n...U_1AU_1...U_n.

function B = Inf_QR(A,n,k)
B = A; d = size(A,1);
 for j = 1:n
    u = House(A(j:j+k,j));
    A(j:j+k,j:d) = A(j:j+k,j:d) - 2*u*(u'*A(j:j+k,j:d));
    B(j:j+k,1:d) = B(j:j+k,1:d) - 2*u*(u'*B(j:j+k,1:d));
    B(1:d,j:j+k) = B(1:d,j:j+k) - 2*(B(1:d,j:j+k)*u)*u';
 end 
\end{verbatim}
\end{algorithm}

\begin{algorithm}
\footnotesize
\begin{verbatim}


% House(x) takes a vector x and creates a unit vector u 
% such that (I - 2u*u')x = ce_1 where c is some complex 
% number (depending on x) and e_1 = [1,0...].

function u = House(x)
v = x;
if v(1) == 0
   v(1) = v(1) + norm(v);             %This is the classical way  
else                                  %of creating Householder reflections 
   v(1) = x(1) + sign(x(1))*norm(x);  %as in finite dimensions.
end
u = v/norm(v);
\end{verbatim}
\end{algorithm}
\normalsize

\subsection{Recalling the basics of the SCI hierarchy}\label{SCI_basics}
The corner stone in the SCI hierarchy is the definition of a computational problem, a general algorithm and towers of algorithms. 
The basic objects in a computational problem are as follows:
\begin{itemize}
\itemsep0em
\item[(i)] $\Omega$ is some set, called the \emph{domain}.
\item[(ii)] $\Lambda$ is a set of complex valued functions on $\Omega$ called the \emph{evaluation} set.
\item[(iii)] $\mathcal{M}$ is a metric space with metric $d_{\mathcal{M}}$.
\item[(iv)] $\Xi:\Omega\to \mathcal{M}$ is called the \emph{problem function}.
\end{itemize}
The set $\Omega$ is the set of objects that give rise to our computational problems. The problem function $\Xi : \Omega\to \mathcal{M}$ is what we are interested in computing. Moreover, the set $\Lambda$ is the collection of functions that provide us with the information we are allowed to read. This leads to the following definition.

\begin{definition}[Computational Problem]
Given a primary set $\Omega$, an evaluation set $\Lambda$, a metric space $\mathcal{M}$ and a problem function $\Xi:\Omega\to \mathcal{M}$ we call the collection $\{\Xi,\Omega,\mathcal{M},\Lambda\}$ a \emph{computational problem}.
\end{definition}

For instance, when computing the spectrum of bounded operators on $l^2(\mathbb{N})$, we let 
$\Omega$ be a subset of $\mathcal{B}(l^2(\mathbb{N}))$ (for example the set of self-adjoint operators or compact operators), $(\mathcal{M},d)$ be the set of all non-empty compact subsets of $\mathbb{C}$ provided with the Hausdorff metric $d=d_{H}$ in 
\eqref{eq:Hausdorff}. The evaluation functions in $\Lambda$ consist of the family of all functions $f_{i,j}: A\mapsto \langle Ae_j,e_i\rangle$, $i,j\in\mathbb{N}$, which provide the entries of the matrix representation of $A$ with respect to the canonical basis $\{e_i\}_{i\in \mathbb{N}}$. Finally, $\Xi: A \mapsto \sigma(A)$.

The goal is to find algorithms which approximate the function $\Xi$. More generally, the main pillar of our framework is the concept of a tower of algorithms, which is needed to describe problems that need several limits in the computation. However, first one needs the definition of a general algorithm.
\begin{definition}[General Algorithm]\label{Gen_alg}
Given a computational problem $\{\Xi,\Omega,\mathcal{M},\Lambda\}$, a \emph{general algorithm} is a mapping $\Gamma:\Omega\to \mathcal{M}$ such that for each $A\in\Omega$
\begin{itemize}
\item[(i)] there exists a finite subset of evaluations $\Lambda_\Gamma(A) \subset\Lambda$, 
\item[(ii)] the action of $\,\Gamma$ on $A$ only depends on $\{A_f\}_{f \in \Lambda_\Gamma(A)}$ where $A_f := f(A),$
\item[(iii)] for every $B\in\Omega$ such that $B_f=A_f$ for every $f\in\Lambda_\Gamma(A)$, it holds that $\Lambda_\Gamma(B)=\Lambda_\Gamma(A)$.
\end{itemize}
%We will sometimes write, with slight abuse of notation, $\Gamma(\{A_f\}_{f \in \Lambda_\Gamma(A)})$, in order to emphasize that $\Gamma(A)$ only depends on the results $\{A_f\}_{f \in \Lambda_\Gamma(A)}$ of finitely many evaluations. 
\end{definition}

Note that the definition of a general algorithm is more general than the definition of a Turing machine or a Blum-Shub-Smale (BSS) machine. A general algorithm has no restrictions on the operations allowed. The only restriction is that it can only take a finite amount of information, though it is allowed to \emph{adaptively} choose the finite amount of information it reads depending on the input. Condition (iii) assures that the algorithm reads the information in a consistent way. Note that the purpose of such a general definition is to get strong lower bounds. In particular, the more general the definition is, the stronger a proven lower bound will be.

With a definition of a general algorithm we can define the concept of towers of algorithms. However, before we define that, we will discuss the cases for which we may have a set valued function. 

\begin{remark}[Set valued functions]\label{rem:set_valued}
Occasionally we will consider a function $\Xi$ such that for $T \in \Omega$ we have that $\Xi(T) \subset \mathcal{M}$. In this case we will still require that a general algorithm produces a single valued out put i.e $\Gamma(T) \in \mathcal{M}$ for $T \in \Omega$. However, we replace the metric in order to define convergence. In particular, $\Gamma_n(T) \rightarrow \Xi(T)$, as $n \rightarrow \infty$ means 
\[
\inf_{y \in \Xi(T)} d_{\mathcal{M}}(\Gamma_n(T),y) \rightarrow 0. 
\]
%where we have by slight abuse considered $ \Gamma_n(T)$ as a subset rather than an element of $\mathcal{M}$.
\end{remark}

\begin{definition}[Tower of Algorithms]\label{tower_funct}
Given a computational problem $\{\Xi,\Omega,\mathcal{M},\Lambda\}$, a \emph{tower of algorithms of height $k$
 for $\{\Xi,\Omega,\mathcal{M},\Lambda\}$} is a family of sequences of functions
 $$\Gamma_{n_k}:\Omega
\rightarrow \mathcal{M},\ \Gamma_{n_k, n_{k-1}}:\Omega
\rightarrow \mathcal{M},\dots,\ \Gamma_{n_k, \hdots, n_1}:\Omega \rightarrow \mathcal{M},
$$
where $n_k,\hdots,n_1 \in \mathbb{N}$ and the functions $\Gamma_{n_k, \hdots, n_1}$ at the ``lowest level'' of the tower are general algorithms in the sense of Definition \ref{Gen_alg}. Moreover, for every $A \in \Omega$,
$$
\Xi(A)= \lim_{n_k \rightarrow \infty} \Gamma_{n_k}(A), \quad \Gamma_{n_k, \hdots, n_{j+1}}(A)= \lim_{n_j \rightarrow \infty} \Gamma_{n_k, \hdots, n_j}(A) \quad j=k-1,\dots,1.
$$
\end{definition}

In addition to a general tower of algorithms (defined above), we will focus on radical towers. The definition of a general algorithm allows for strong lower bounds, however, to produce upper bounds we must add structure to the algorithm and towers of algorithms. A radical tower allows for arithmetic operations, comparisons and radicals. 

\begin{definition}[Radical Towers]
Given a computational problem $\{\Xi,\Omega,\mathcal{M},\Lambda\}$, a \emph{Radical Tower of Algorithms} of height $k$
 for $\{\Xi,\Omega,\mathcal{M},\Lambda\}$ is a tower of algorithms where the lowest level functions 
 $$
 \Gamma = \Gamma_{n_k, \hdots, n_1} :\Omega \rightarrow \mathcal{M}
 $$ satisfy the following:
 For each $A\in\Omega$
the action of $\,\Gamma$ on $A$ consists of only finitely many arithmetic operations, comparisons and radicals ($\sqrt{\cdot}$) of positive numbers on 
$\{A_f\}_{f \in \Lambda_\Gamma(A)}$, where $A_f = f(A)$.
\end{definition}

In other words one may say that for the finitely many steps of the computation of the lowest functions $\Gamma = \Gamma_{n_k, \hdots, n_1} :\Omega \rightarrow \mathcal{M}$ only the four arithmetic operations $+, -, \cdot, /$ within the smallest (algebraic) field which is generated by the input $\{A_f\}_{f \in \Lambda_\Gamma(A)}$ are allowed. In addition we allow to extract radicals of positive real numbers.
 We implicitly assume that any complex number can be decomposed into a real and an imaginary part, and moreover we can determine whether $a = b$ or $a > b$ for all real numbers $a,b$ which can occur during the computations.
Given the definitions above we can now define the key concept, namely, the Solvability Complexity Index: 

\begin{definition}[Solvability Complexity Index]\label{complex_ind}
A computational problem $\{\Xi,\Omega,\mathcal{M},\Lambda\}$ is said to have \emph{Solvability Complexity Index $\mathrm{SCI}(\Xi,\Omega,\mathcal{M},\Lambda)_{\alpha} = k$}, with respect to a tower of algorithms of type $\alpha$, if $k$ is the smallest integer for which there exists a tower of algorithms of type $\alpha$ of height $k$. If no such tower exists then $\mathrm{SCI}(\Xi,\Omega,\mathcal{M},\Lambda)_{\alpha} = \infty.$ If there exists a tower $\{\Gamma_n\}_{n\in\mathbb{N}}$ of type $\alpha$ and height one such that $\Xi = \Gamma_{n_1}$ for some $n_1 < \infty$, then we define $\mathrm{SCI}(\Xi,\Omega,\mathcal{M},\Lambda)_{\alpha} = 0$. We may sometimes write $\mathrm{SCI}(\Xi,\Omega)_{\alpha}$ to simplify notation when $\mathcal{M}$ and $\Lambda$ are obvious. 
\end{definition}

 The definition of the SCI immediately induces the SCI hierarchy:

\begin{definition}[The Solvability Complexity Index Hierarchy]
\label{1st_SCI}
Consider a collection $\mathcal{C}$ of computational problems and let $\mathcal{T}$ be the collection of all towers of algorithms of type $\alpha$ for the computational problems in $\mathcal{C}$.
Define 
\begin{equation*}
\begin{split}
\Delta^{\alpha}_0 &:= \{\{\Xi,\Omega\} \in \mathcal{C} \ \vert \   \mathrm{SCI}(\Xi,\Omega)_{\alpha} = 0\}\\
\Delta^{\alpha}_{m+1} &:= \{\{\Xi,\Omega\}  \in \mathcal{C} \ \vert \   \mathrm{SCI}(\Xi,\Omega)_{\alpha} \leq m\}, \qquad \quad m \in \mathbb{N},
\end{split}
\end{equation*}
as well as
\[
\Delta^{\alpha}_{1} := \{\{\Xi,\Omega\}  \in \mathcal{C}   \  \vert \ \exists \ \{\Gamma_n\}_{n\in \mathbb{N}} \in \mathcal{T}\text{ s.t. } \forall A \ d(\Gamma_n(A),\Xi(A)) \leq 2^{-n}\}. 
\]
\end{definition}

\begin{remark}[The $\Delta_k$ notation]
Note that in this paper we only consider radical towers and hence the superscript $\alpha$ will be omitted throughout. Thus we will always write $\Delta_k$.
\end{remark}

Finally, we recall the definition of $\Sigma_1^{\alpha}$. 

\begin{equation*}
\begin{split}
\Sigma^{\alpha}_{1} &= \{\{\Xi,\Omega\} \in \Delta^{\alpha}_{2} \ \vert \  \exists \ \{\Gamma_{n}\}_{n\in \mathbb{N}} \in \mathcal{T}\text{ s.t. }  \ \Gamma_{n}(A) \subset \mathcal{N}_{2^{-n}}(\Xi(A))  \ \text{and}\ \Gamma_n(A)\rightarrow\Xi(A) \ \forall A \in \Omega\} \\
\end{split}
\end{equation*}
where $\mathcal{N}_{\delta}(\omega)$ denotes the $\delta$-neighbourhood of $\omega \subset \mathcal{M}.$

\small
\bibliographystyle{abbrv}
\bibliography{InfMatBib}

\end{document}